\documentclass{article}

\usepackage[a4paper,left=20mm,top=5mm,right=20mm,bottom=5mm,includeheadfoot,headheight=20mm, headsep=5mm, footskip=15mm]{geometry}
\usepackage[utf8]{inputenc}
\usepackage[all]{xy}
\usepackage{amsthm,amsmath,amssymb,amsfonts,amscd}
\usepackage[backend=bibtex]{biblatex}
\usepackage{bbm}
\usepackage{booktabs}
\usepackage{csquotes}
\usepackage[english]{babel}
\usepackage{enumerate}
\usepackage{fancyhdr}
\usepackage[final]{pdfpages}
\usepackage{graphicx}
\usepackage{imakeidx}
\usepackage{inconsolata}
\usepackage{latexsym,keyval,ifthen}
\usepackage[makeroom]{cancel}
\usepackage{mathalfa}
\usepackage{mathdots}
\usepackage{multicol}
\usepackage{multicol}
\usepackage{nameref}
\usepackage{newpxmath}
\usepackage{prerex}
\usepackage{pst-node}
\usepackage{randomwalk}
\usepackage{soul}
\usepackage{tgpagella}
\usepackage{theoremref}
\usepackage{tikz}
\usepackage{tikz-cd}
\usetikzlibrary{arrows,arrows.meta}
\usepackage{wrapfig}
\usepackage{xgalley}
\usepackage{yfonts}
\usepackage{calc}

\DeclareMathAlphabet{\mathboondoxfrak}{U}{BOONDOX-frak}{m}{n}

\DeclareMathOperator{\coker}{coker}

\DeclareMathOperator{\Cl}{Cl}

\DeclareMathOperator{\Gal}{Gal}
\DeclareMathOperator{\GL}{GL}

\DeclareMathOperator{\Hom}{\textnormal{\textbf{Hom}}}

\DeclareMathOperator{\PGL}{PGL}

\def\Yint#1{\mathchoice
    {\YYint\displaystyle\textstyle{#1}}
    {\YYint\textstyle\scriptstyle{#1}}
    {\YYint\scriptstyle\scriptscriptstyle{#1}}
    {\YYint\scriptscriptstyle\scriptscriptstyle{#1}}
      \!\int}
\def\YYint#1#2#3{{\setbox0=\hbox{$#1{#2#3}{\int}$}
    \vcenter{\hbox{$#2#3$}}\kern-.52\wd0}}

\def\mint{\Yint\times}

\newcommand{\A}{\mathbb{A}}

\newcommand{\bbm}{\begin{bm}}

\newcommand{\bes}{\begin{center}\begin{tikzcd}[ampersand replacement=\&]}

\newcommand{\bpf}{\begin{proof}}
\newcommand{\bsm}{\big(\begin{smallmatrix}}

\newcommand{\cla}[1]{\{#1\}}
\newcommand{\C}{\mathbb{C}}
\newcommand{\D}{\mathbb{D}}

\newcommand{\ebm}{\end{bm}}
\newcommand{\ees}{ \end{tikzcd} \end{center}}
\newcommand{\epf}{\end{proof}}
\newcommand{\esm}{\end{smallmatrix}\big)}

\newcommand{\foodnote}[1]{\ignorespaces}

\newcommand{\icla}[1]{\left\lbrace#1\right\rbrace}

\newcommand{\indi}{\mathbbm{1}}
\newcommand{\Ind}{\textnormal{Ind}}
\newcommand{\inv}{{^{-1}}}
\newcommand{\ipa}[1]{\left(#1\right)}
\newcommand{\iy}{\infty}

\newcommand{\lra}{\longrightarrow}

\newcommand{\mbT}{\mathbb{T}}
\newcommand{\mcA}{\mathcal{A}}

\newcommand{\mcD}{\mathcal{D}}
\newcommand{\mcE}{\mathcal{E}}
\newcommand{\mcF}{\mathcal{F}}
\newcommand{\mcG}{\mathcal{G}}

\newcommand{\mcI}{\mathcal{I}}
\newcommand{\mcK}{\mathcal{K}}

\newcommand{\mcL}{\mathcal{L}}
\newcommand{\mcM}{\mathcal{M}}

\newcommand{\mcO}{\mathcal{O}}
\newcommand{\mcP}{\mathcal{P}}

\newcommand{\mcU}{\mathcal{U}}
\newcommand{\mcV}{\mathcal{V}}

\newcommand{\mfg}{\mathfrak{g}}

\newcommand{\mfm}{\mathfrak{m}}

\newcommand{\mfo}{\mathfrak{o}}
\newcommand{\mfp}{\mathfrak{p}}

\newcommand{\mfq}{\mathfrak{q}}

\newcommand{\N}{\mathbb{N}}

\newcommand{\ord}{\textnormal{ord}}
\newcommand{\oti}{\otimes}

\newcommand{\ovl}{\overline}

\newcommand{\PP}{\mathbb{P}}

\newcommand{\Q}{\mathbb{Q}}

\newcommand{\ra}{\rightarrow}
\newcommand{\R}{\mathbb{R}}

\newcommand{\siun}{\Sigma_{\tno{un}}(K/F)}

\newcommand{\tbf}{\textbf}

\newcommand{\ti}{\times}

\newcommand{\tno}{\textnormal}

\newcommand{\udl}{\underline}
\newcommand{\uhp}{\mathfrak{H}}

\newcommand{\Z}{\mathbb{Z}}

\newtheorem{theorem}{Theorem}[section]
\newtheorem{lemma}[theorem]{Lemma}

\newtheorem{remark}[theorem]{Remark}
\newtheorem{hypothesis}[theorem]{Hypothesis}

\newtheorem{definition}[theorem]{Definition}
\newtheorem{proposition}[theorem]{Proposition}

\newenvironment{bm}{ \begin{pmatrix} }{ \end{pmatrix} } 

\bibliography{biblio}
\begin{document}
\title{Plectic points and Hida-Rankin $p$-adic $L$-functions}
\author{V\'ictor Hern\'andez Barrios and Santiago Molina Blanco}
\maketitle
\begin{abstract}
Plectic points were introduced by Fornea and Gehrmann in \cite{fornea2021plectic} as certain tensor products of local points on elliptic curves over arbitrary number fields $F$. They are constructed in rank $r\leq[F:\Q]$-situations, and they conjecturally come from $p$-adic regulators of basis of the Mordell-Weil group defined over dihedral extensions of $F$.

In this article we define anticyclotomic $p$-adic L-functions with variables of type weight and level attached to a family of overconvergent modular symbols defined over totally real fields $F$ and a quadratic extension $K/F$. Their restriction to the weight space provide Hida-Rankin $p$-adic L-functions. 

If such a family passes through an overconvergent modular symbol attached to a modular elliptic curve $E$,  we obtain a $p$-adic Gross-Zagier formula that computes higher derivatives of such Hida-Rankin $p$-adic L-functions in terms of plectic points. This result generalizes that of Bertolini and Darmon in \cite{BD09}, which has been key to the recent approach towards the algebraicity of Darmon points.
\end{abstract}


\section{Introduction}

Let $E$ be a modular elliptic curve over a totally real number field $F$ of degree $d=[F:\Q]$.
One of the most important breakthroughs towards the Birch and Swinnerton-Dyer conjecture is the well-known Gross-Zagier formula, that relates the first derivative of a Rankin-Selberg L-function attached to $E$ with the height of a classical Heegner point. Such Heegner points are defined over abelian extensions of a CM number field $K$ which is quadratic over $F$. The theory of Heegner points together with the Gross-Zagier formula are the main ingredients to prove cases of the BSD conjecture in rank 0 and 1 situations, which is the best progress towards the conjecture we have so far.

Due to the success of Heegner points, such a construction has been generalized to other settings, for instance, the cases where 
the extension $K/F$ is just quadratic not necessarily CM. The first of these constructions is due to Darmon \cite{Darmon2001}, and it covers the case where $F=\mathbb{Q}$ and $K$ is a real quadratic field. Darmon's construction applies in rank one situations, and even though it is a purely $p$-adic construction, it is conjectured to produce algebraic points defined over abelian extensions of $K$. 
The original construction has been generalized 
in \cite{Dasgupta}, \cite{Greenberg}, \cite{Gartner}, \cite{GMS} and \cite{guitart2017automorphic}, providing points in $E(K_v)$, for certain completions $K_v$ of $K$, which are conjectured to be algebraic.

Very recently, a new construction that conjecturally works for ranks $r\geq 1$ situations has emerged due to Fornea and Gehrmann (see \cite{fornea2021plectic}). Such a construction applies when there are $r$ prime ideals $\mathfrak{P}_1,\cdots,\mathfrak{P}_r$ of $F$ above a rational prime $p$ where the curve $E$ has multiplicative reduction. Moreover, Fornea and Gehrmann additionally assume in \cite{fornea2021plectic} that the quadratic extension $K/F$ is inert at those primes, although such a construction can be extended to the non-split setting as we explain in \S \ref{plecticpoints}. In the ramified situations the arithmetic properties of the plectic points will differ somewhat from those conjectured by Fornea-Gehrmann since the local signs of the functional equations are more subtle in these scenarios. For simplicity, we will stick to the inert case in this introduction; otherwise the notations will get a bit more complicated. Given a dihedral finite order character $\xi$ of the absolute Galois group of $K$ unramified at all places $\mathfrak{P}_i$, $i=1,\cdots,r$, one obtains an element
\[
P_\xi^{S_0}\in \bigotimes_{j=1}^r \left(\hat E(K_{\mathfrak{P}_j})\otimes_{\Z_p}\bar\Z_p\right),\qquad S_0=\{\mathfrak{P}_1,\cdots,\mathfrak{P}_r\},
\]
where $\hat E(K_{\mathfrak{P}_j})$ is the $p$-adic completion of $E(K_{\mathfrak{P}_j})$, $\bar\Z_p$ is the integer ring of $\bar\Q_p$, the tensor product is taken as $\bar\Z_p$-modules and the field $K_{\mathfrak{P}_j}$ is the completion of $K$ at the prime ideal above $\mathfrak{P}_j$. When $r=1$, one recovers the aforementioned Darmon points.
As can be seen, such a construction is $p$-adic, namely, as a result we get (tensor product of) points defined over finite extensions of $\mathbb{Q}_p$. In analogy with the theory of Heegner points, it is conjectured in \cite{fornea2021plectic} that such points are defined over certain explicit abelian extensions of $K$. More precisely, if we write $E(K)_{\xi}$ for the $\xi$-isotypical component of the set of points defined over the maximal abelian extension of $K$ and the rank of $E(K)_\xi$ is $r$, then it is expected that there exist $r$ points $P_1,\cdots,P_r\in E(K)_\xi$ such that 
\[
P_\xi^{S_0}=\det(P_1\wedge\cdots\wedge P_r),
\]
where
\[
\det:\bigwedge^r E(K)_\xi\longmapsto \bigotimes_{j=1}^r \left(\hat E(K_{\mathfrak{P}_j})\otimes_{\Z_p}\bar\Z_p\right);\qquad \det(P_1\wedge\cdots\wedge P_r):= \det\left(\begin{array}{ccc}\imath_{\mathfrak{P}_1}(P_1)&\cdots&\imath_{\mathfrak{P}_r}(P_1)\\&\cdots&\\\imath_{\mathfrak{P}_1}(P_r)&\cdots&\imath_{\mathfrak{P}_r}(P_r)\end{array}\right),
\]
and 
$\imath_{\mathfrak{P}}:E(K)_\xi\rightarrow \left(\hat E(K_{\mathfrak{P}_j})\otimes_{\Z_p}\bar\Z_p\right)$
is the natural morphism provided by the inclusion $K\subset K_{\mathfrak{P}}$.

If we go back to the case $r=1$, some progress towards algebraicity of Darmon points has been made. In \cite{BD09} and \cite{longo} a strategy for $F=\Q$ and $K$ real quadratic is carried out that consists of realising Darmon points as derivatives of Hida–Rankin $p$-adic L-functions. This result allows one to show that certain linear combination of Darmon points with coefficients given by values of genus characters of $K$ comes from a global point defined over the Hilbert class field of $K$. Moreover, it also provides the crucial bridge between Darmon points and generalised Kato classes arising from $p$-adic families of diagonal cycles. The results of \cite{BD09} and \cite{longo} can also be exploited to make a similar comparison with Heegner points, and the explicit comparison between Heegner or Darmon points and generalised Kato classes is the main tool of \cite{BDRSV} to finally prove algebraicity of Darmon points, under the non-negligible assumption of the finitude of the $p$-part of the Tate-Shafarevich group.

If we want to follow the footsteps of \cite{BDRSV} in order to prove algebraicity of plectic points, we need to prove an analogous $p$-adic Gross-Zagier formula that realizes plectic points as derivatives of Hida–Rankin $p$-adic L-functions. In this note, we will construct such Hida–Rankin $p$-adic L-functions as restrictions to the weight space of $p$-adic L-functions with variables of type weight and level interpolating anticyclotomic $p$-adic L-functions. 

Let $B$ be a quaternion algebra over $F$ and let $G$ be the algebraic group attached to $B^\ti/F^\ti$. Assume that $E$ is associated with a weight 2 automorphic form of $G$, and let $\pi$ be the automorphic representation that it generates. The Harder-Eichler-Shimura isomorphism realizes $\pi$ in the cohomology spaces of certain arithmetic subgroups of $G$. Thus, our automorphic forms give rise to cohomology classes $\phi_\lambda^p$ that we will call (automorphic) modular symbols (see \S \ref{GenAutforms} for a precise description of $\phi^p_\lambda$).
Let $K/F$ be a quadratic extension admitting an embedding $K\hookrightarrow B$ and such that the set of infinite places where $K/F$ splits coincides with that of $B$. In \cite{fornea2021plectic} and \cite{HerMol1} an anticyclotomic $p$-adic L-function associated with $E$ and $K/F$ is constructed by means of $\phi^p_\lambda$. 
Such a $p$-adic L-function is thought of as a $p$-adic distribution $\mu_{\phi^p_\lambda}$ of the  Galois group $\mcG_T$ of the anticyclotomic abelian extension of $K$. Moreover, the main result in both \cite{fornea2021plectic} and \cite{HerMol1} provides a $p$-adic Gross-Zagier formula that computes higher derivatives of $\mu_{\phi^p_\lambda}$ in terms of plectic points and L-invariants. Following the philosophy of Mazur-Tate in \cite{MT}, we understand the evaluation of a distribution at a character $\xi$ as its image through the ring homomorphism
\[
\varphi_\xi:{\rm Meas}(\mcG_T)\longrightarrow\C_p,\qquad \varphi_\xi(\mu)=\int_{\mcG_T}\xi(\gamma)d\mu(\gamma).
\]
Therefore, if we write $I_\xi=\ker(\varphi_\xi)$ for the augmentation ideal, the distribution $\mu_{\phi^p_\lambda}$ lies in $I_\xi^r$ when it has a zero of order $r$ at $\xi$. If this is the case, the value of
the $r$-th derivative of $\mu_{\phi^p_\lambda}$ at $\xi$ is understood to be its image in $I_\xi^r/I_\xi^{r+1}$. There is a natural way to send elements of $\bigotimes_{j=1}^r \hat E(K_{\mfp_j})_\xi$ to $I_\xi^r$, and the $p$-adic Gross-Zagier formula relates $\mu_{\phi^p_\lambda}$ with the image of $P_\xi^S$ modulo $I_\xi^{r+1}$. In this paper we will obtain a similar result but exchanging $\mu_{\phi^p_\lambda}$ with a new $p$-adic L-function $L_p(\phi_\lambda^p,\xi,\underline{k})$ depending on the weight $\underline{k}+2\in (2\N)^d$ of the automorphic modular symbol. For these purposes, we need to extend the definition of $\mu_{\phi^p_\lambda}$ to scenarios where the automorphic modular symbols $\phi^p_\lambda$ have arbitrary even weight and interpolate the distributions $\mu_{\phi^p_\lambda}$ in families.

In this paper we first give a general construction of locally polynomial anticyclotomic distributions associated with automorphic modular symbols $\phi_\lambda^p$ of even weight $\underline{k}+2$. We will generalize our construction of $\mu_{\phi^p_\lambda}$ described in \cite{HerMol1} for weight 2 to the setting of arbitrary even weight. We will study in \S \ref{section:IntProp} the \emph{interpolation properties} that relate $\mu_{\phi^p_\lambda}$ to classical Rankin-Selberg L-functions attached to the representation $\pi$ twisted by a locally polynomial character. Unlike the rest of the paper, in this section we will make a careful choice of the modular symbol $\phi^p_\lambda$ under certain simplifying hypothesis (hypothesis \ref{simplhyp}) in order to obtain more explicit interpolation properties.    
Finally, we will show that in the ordinary scenario our locally polynomial $\mu_{\phi^p_\lambda}$ extends to a continuous measure. Following the ideas of Pollack-Stevens in \cite{PS}, such continuous measure $\mu_{\phi^p_\lambda}$ will be obtained as the cap-product of an overconvergent modular symbol $\hat\phi_\lambda^p$ extending $\phi_\lambda^p$ and a fundamental class associated with the quadratic extension $K/F$.

Once we have defined $\mu_{\phi^p_\lambda}$ for arbitrary even weights, one needs to interpolate such continuous measures as the weight varies. For this, we need to introduce the weight space. We will restrict ourselves to the ordinary case, hence it will be enough to work with classical Iwasawa algebras, although 
we believe that everything can be extended to the non-critical setting of Coleman families using standard techniques.
Let 
\[
\Lambda_F=\Z_p[[\mfo_{F_p}^\ti]],\qquad \mfo_{F_p}:=\mfo_F\oti_\Z\Z_p,
\]
be the Iwasawa algebra associated with $\mfo_{F_p}^\ti$. It classifies continuous characters $\chi_p:\mfo_{F_p}^\ti\rightarrow R^\ti$, for any complete $\Z_p$-algebra $R$. In particular, any even weight $\underline{k}=(k_\sigma)\in (2\N)^d$ is seen as a point in the spectrum of $\Lambda_F$ because it defines the character $x\mapsto x^{\underline{k}}:=\prod_\sigma \sigma(x)^{k_\sigma}$. These types of weights will usually be referred to as \emph{classical weights}.
The idea behind the construction of the Hida-Rankin $p$-adic L-function $L_p(\phi_\lambda^p,\xi,\underline{k})$ is to compute the image through $\varphi_\xi$ of a family that interpolates the measures $\mu_{\phi^p_\lambda}$ (see a similar construction by Bergdall-Hansen in \cite{BH}). This will produce an element in $\Lambda_F$ whose classical specializations are values at $\xi$ of anticyclotomic distributions associated with automorphic modular symbols of different weights.


In \S \ref{HidaFamilies} we will introduce our concept of family of overconvergent modular symbols $\Phi_\lambda^p$. Since we have assumed that our setting is ordinary, $\Phi_\lambda^p$ will be a cohomology class of $\Lambda_F$-valued measures analogous to classical Hida families. Although our concept of family differs from the classical one treated for instance in \cite{BDJ} using local systems, we will prove in \S \ref{locsystevsgroupcoho} that both concepts coincide. The analysis of \S \ref{locsystevsgroupcoho} together with the results of \cite{BDJ} will ensure the existence of the Hida family $\Phi_\lambda^p$ passing through $\hat\phi_\lambda^p$. Given such a $\Phi_\lambda^p$, one can naturally construct a $\Lambda_F$-valued measure $\mu_{\Phi_\lambda^p}$ using our formalism. We may think of the measure $\mu_{\Phi_\lambda^p}$ as a $p$-adic L-function 
 with variables of type weight and level passing through $\mu_{\phi_\lambda^p}$. We define $L_p(\phi_\lambda^p,\xi,{\bf k}_p)$ as the restriction of $\mu_{\Phi_\lambda^p}$ to the weight space, more precisely, 
\[
L_p(\phi_\lambda^p,\xi,{\bf k}_p)=\varphi_\xi(\mu_{\Phi_\lambda^p})= \int_{\mcG_T}\xi(\gamma)d\mu_{\Phi_\lambda^p}(\gamma)\in \Lambda_F\oti_\Z\bar\Q.
\]
The specialization at parallel weight 2 provides a morphism
\[
s_1\colon\Lambda_F\longrightarrow\Z_p,
\]
Thus, the evaluation of $L_p(\phi_\lambda^p,\xi,{\bf k}_p)$ at weight 2 is given by its image through $s_1$. The $p$-adic L-function has a zero of order at least $r$ at weight 2 if $L_p(\phi_\lambda^p,\xi,{\bf k}_p)\in I^{r}\oti_\Z\bar\Q$, where $I=\ker s_1$ is the corresponding augmentation ideal. If this is the case, the value of its $r$-th derivative is given by the image of $L_p(\phi_\lambda^p,\xi,{\bf k}_p)$ in $I^{r}\oti_\Z\bar\Q/I^{r+1}\oti_\Z\bar\Q$. The main result of the paper computes such higher derivatives in terms of plectic points (see theorem \ref{mainTHM}):
\begin{theorem}\label{mainTHMintro}
Assume that for all primes above $p$ the curve $E$ has either ordinary good or multiplicative reduction and let $\phi_\lambda^p$ be a modular symbol associated with $E$.
Let $S_0$ be the set of places $\mfp$ above $p$ where $E$ has multiplicative reduction and $\mfp$ does not split in $K$, and assume that $r=\#S_0\neq 0$.
Then $L_p(\phi_\lambda^p,\xi,{\bf k}_p)$ has a zero of order at least $r$ at weight 2
and there exists an explicit non-zero constant $C_T$ depending on $T(F)$ such that 
\[
L_p(\phi^p_\lambda,\xi,{\bf k}_p)\equiv C_T\cdot\epsilon^{S_0}(\pi,\xi)^{\frac{1}{2}}\cdot\ell_{{\bf a}_{S_0}}\circ{\rm Tr}(P_\xi^{S_0})\;\;({\rm mod}\;I^{r+1}\oti_\Z\bar\Q),
\]
where ${\rm Tr}:\bigotimes_{\mfp\in S_0}\hat E(K_\mfp)\rightarrow \bigotimes_{\mfp\in S_0}\hat E(F_\mfp)$ is a natural trace map provided by $P\mapsto \sum_{\sigma\in \Gal(K_\mfp/F_\mfp)} P^\sigma$, $\hat E(F_\mfp)= E(F_\mfp)\oti_\Z\Z_p$ and $\epsilon^{S_0}(\pi,\xi)$ is a product of explicit Euler factors at primes dividing $p$ outside $S_0$.
\end{theorem}
The natural morphism $\ell_{{\bf a}_{S_0}}:\bigotimes_{\mfp\in S_0}\hat E(F_\mfp)\rightarrow I^r$ will be explained in \S \ref{Linv} and it can be thought as the tensor product of the classical $p$-adic logarithms of the groups $E(F_\mfp)$. 

Very recently, results of \cite{fornea2021plectic} on the higher derivatives of $\mu_{\phi_\lambda^p}$ have been used in \cite{forneagehrmann2} to prove algebraicity of plectic points in some particular cases. More concretely, they prove that the minus part for the action of partial Frobenii of a plectic point $P_\xi^S$ is algebraic, in some specific situations where $P_\xi^{S_0}$ can be compared with Heegner points. In addition to the aforementioned general strategy of generalizing \cite{BDRSV}, Theorem \ref{mainTHMintro} opens the door to prove similar results for the corresponding plus part ${\rm Tr}(P_\xi^{S_0})$ of $P_\xi^{S_0}$, providing more evidence on the conjectures that predict algebraicity of general plectic points.

\subsection{On the fact that the base field $F$ is totally real}

In many of the parts of this work we can avoid the assumption that the base field $F$ is totally real. Indeed, if $F$ is any arbitrary number field we can define the cohomology classes $\phi_\lambda^p$, $\hat\phi_\lambda^p$ and $\Phi_\lambda^p$, moreover, our formalism allows us to construct the corresponding distributions $\mu_{\phi_\lambda^p}$ and $\mu_{\Phi_\lambda^p}$. Furthermore, the results of \cite{preprintsanti2} and \cite{preprintsanti3} used in \S \ref{section:IntProp} are also available in this situation, hence the distributions $\mu_{\phi_\lambda^p}$ have good interpolation properties. 

The main issue with the non-totally real situation is  
the fact that, given an overconvergent $\hat\phi_\lambda^p$ (extending $\phi_\lambda^p$), it is much harder to construct a Hida family $\Phi_\lambda^p$ passing through $\hat\phi_\lambda^p$. This also has implications in the theory of plectic points over $F$ since, although we can construct points on certain $p$-adic Tate curves, we cannot guarantee that such a curve is isogenous to $E$ (see \S \ref{Linv}).


\subsection*{Acknowledgements.}
This project has received funding from the European Research Council
(ERC) under the European Union's Horizon 2020 research and innovation
programme (grant agreement No. 682152). Moreover, the second author has been partially funded by the project PID2021-124613OB-I00 from Ministerio de Ciencia e innovaci\'on. We would also like to thank Lennart Gehrmann for all his helpful remarks and comments.

\tikzcdset{arrow style=tikz, diagrams={>=stealth'}}
\section{Setup and notation}\label{section:setup}
For any number field or non-archimedean local field $L$ we will write $\mfo_L$ for its ring of integers. Sometimes we denote 
$\mfo_{\ovl\Q_p}$ by $\ovl\Z_p$. 

For any pair of topological groups $H_1\subset H_2$, a ring $R$ and a $R[H_1]$-module $M$, we will write
\begin{eqnarray*}
{\rm coInd}_{H_1}^{H_2}M&=&\left\{f:H_2\rightarrow M\;\mbox{locally constant, }f(h_1h_2)=h_1f(h_2),\;h_1\in H_1,\;h_2\in H_2\right\},\\
{\rm Ind}_{H_1}^{H_2}M&=&\left\{f\in {\rm Ind}_{H_1}^{H_2}M,\;f\mbox{ compactly supported modulo }H_1\right\},
\end{eqnarray*}
both endowed with the $H_2$-action given by $h_2 f(\cdot)=f(\cdot h_2)$, for every $h_2\in H_2$. If $M=R$ a topological ring and the action is given by a continuous character $\chi:H_1\rightarrow R$, we will also write by abuse of notation 
\[
{\rm coInd}_{H_1}^{H_2}(\chi)=\left\{f:H_2\rightarrow R\;\mbox{continous, }f(h_1h_2)=\chi(h_1)f(h_2),\;h_1\in H_1,\;h_2\in H_2\right\},
\]
and similarly for ${\rm Ind}_{H_1}^{H_2}(\chi)$. If $\chi$ is also locally constant, we will denote by ${\rm coInd}_{H_1}^{H_2}(\chi)^0$ and ${\rm Ind}_{H_1}^{H_2}(\chi)^0$ the subset of locally constant functions. Note that if $H_2/H_1$ is compact then ${\rm coInd}_{H_1}^{H_2}(\chi)^\ast={\rm Ind}_{H_1}^{H_2}(\chi)^\ast$, for $\ast=\emptyset,0$.

Let $F$ be a totally real number field of degree $d$.
We will often write $\iy$ for the set of infinite places of $F$ and $p$ for the set of places of $F$ above a rational prime $p$. 
Let $K/F$ be a quadratic extension, and let $\Sigma_{\tno{un}}(K/F)$ be the set of infinite places of $F$ that split at $K$.
If we write $u$ for the cardinality of $\Sigma_{\tno{un}}(K/F)$, 
it is clear that the set $\Sigma_\R^\C(K/F)=\infty\setminus\siun$ has cardinality $r_{K/F}=d-u$. 

For any finite set of places $S$ of $F$ we write $F_S=\prod_{v\in S} F_v$, and if the places of $S$ are finite, we write $\mfo_{F_S}=\prod_{v\in S} \mfo_{F_v}$. 
We denote by $\A_F$ the ring of adèles of $F$, and by $\A_F^S$ the ring of adèles outside $S$, namely, $\A_F^S=\A_F\cap\prod_{v\not\in S}F_v$. For a finite place $\mfp$, we will write $\varpi_\mfp$ for a fixed uniformizer, $v_\mfp:F_\mfp^\times\ra\Z$ for the $p$-adic valuation such that $v_\mfp(\varpi_\mfp)=1$,  and $q_\mfp$ for the cardinal of the residue field $\mfo_{F_\mfp}/\mfp$. Similarly, we write $v_p:\C_p^\ti\rightarrow\Q$ for the $p$-adic valuation such that $v_p(p)=1$, and we denote by $g_p$ the number of primes dividing $p$.

Let $B/F$ be a quaternion algebra for which there exists an embedding $\iota:K\hookrightarrow B$, that we fix now. Let $\Sigma_B$ be the set of archimedean classes of $F$ splitting the quaternion algebra $B$.
We can define an algebraic group $G$ associated to $B^\ti/F^\ti$ as follows:
$G$ represents the functor that sends any $\mfo_F$-algebra $R$ to
\[G(R)=(\mcO_B\oti_{\mfo_F} R)^\ti/R^\ti,\]
where $\mcO_B$ is a maximal order in $B$ that will be fixed throughout the paper.
Similarly, we define the algebraic group $T$ associated to $K^\ti/F^\ti$ by
\[T(R)=(\mfo_{c_0}\oti_{\mfo_F} R)^\ti/R^\ti,\]
where $\mfo_{c_0}:=\mcO_B\cap K$ is an order of conductor $c_0$ in $\mfo_K$ (namely, $\mfo_{c_0}=\mfo_F+c_0\mfo_K$).
Note that $T\subset G$.
We denote by $G(F_\iy)_+$ and $T(F_\iy)_+$ the connected component of the identity of $G(F_\iy)$ and $T(F_\iy)$, respectively.
We also define $T(F)_+:=T(F)\cap T(F_\iy)_+$ and $G(F)_+:=G(F)\cap G(F_\iy)_+$.
Throughout the paper we will fix a rational prime $p$ such that $G(F_p)\simeq\PGL_2(F_p)$.

For any number field $L$ and any place $v$, we choose the Haar measure $dx_v^\times$ for $L_v^\times$:
\[
d^\times x_v=\zeta_v(1)|x_v|_v^{-1}dx_v;\qquad\mbox{where}\quad \left\{\begin{array}{ll}
 dx_v\mbox{ is $[L_v:\R]$ times the usual Lebesgue measure,}    &v\mid\infty;  \\
   dx_v\mbox{ is the Haar measure of $F_v$ such that }{\rm vol}(\mfo_{F,v})=|d_{F_v}|_v^{1/2},  &v\nmid\infty, 
\end{array}\right.
\]
$d_{F_v}$ is the different of $L_v$, $\zeta_v(s)=(1-q_v^{-s})^{-1}$, if $v\nmid\infty$, $\zeta_v(s)=\pi^{-s/2}\Gamma(s/2)$, if $L_v=\R$, and $\zeta_v(s)=2(2\pi)^{-s}\Gamma(s)$, if $L_v=\C$. The product of such measures provides a Tamagawa measure $d^\times x$ on $\A_L^\times/L^\times$. If we choose $d^\times t$ to be the quotient measure for $T(\A_F)/T(F)=\A_K^\times/\A_F^\times K^\times$, one has that ${\rm vol}(T(\A_F)/T(F))=2L(1,\psi_K)$, where $\psi_K$ is the quadratic character associated with $K/F$.

Given two topological spaces $X$ and $Y$, we will denote by $C(X,Y)$ and $C^0(X,Y)$ the spaces of continuous and locally constant functions. Sometimes $C(X,Y)$ will be regarded as a topological space by means of the compact-open topology, namely, the spaces $V(K,U)=\{f\in C(X,Y);\;f\mid_K\in C(K,U)\}$ for $K\subseteq X$ compact and $U\subseteq Y$ open form a subbase. It is clear that $C^0(X,Y)\subseteq C(X,Y)$ is a dense subset. For both $\ast=\emptyset,0$, we will denote by $C_c^\ast(X,Y)\subseteq C^\ast(X,Y)$ the subset of compactly supported functions. For the case $X=T(\A_F)$, we have that  
\begin{equation}\label{eqdescribingCfc}
    C_{\rm fc}^0(T(\A_F),R):=\{f\in C^0(T(\A_F),R);\; f(x,\cdot)\in C_c^0(T(\A_F^\infty),R),\mbox{ for all }x\in T(F_\infty)\}\simeq \Ind_{T(F)_+}^{T(F)}C_c^0(T(\A_F^\infty),R),
\end{equation}
since $T(F_\infty)/T(F_\infty)_+\simeq T(F)/T(F)_+$.

Let $M$ be a $G(\A_F^S)$-representation over a field $L$, and let $\rho$ be an irreducible $G(\A_F^{S'})$-representation over $L$ with $S'\subseteq S$. We will write 
\[
M_\rho:=\Hom_{G(\A_F^S)}\left(\rho\mid_{G(\A_F^S)},M\right).
\]
as representations over $L$.

For any ring $R$ and a positive even integer $k$, let $\mcP(k)_R$ be the $R$-module of homogeneous polynomials in two variables with coefficients in $R$, together with the following $\GL_2(R)$-action;
\[\ipa{\gamma \cdot P}(x,y):=\ipa{\det \gamma}^{-k/2}\cdot P\ipa{(x,y)\cdot\bbm a & b \\ c & d \ebm};\qquad \mbox{for}\quad\gamma=\bbm a & b \\ c & d\ebm.\]
Note that the factor $\ipa{\det \gamma}^{-k/2}$ makes the central action trivial.

Denote by $V(k)_R$ the dual space $\Hom_R(\mcP(k)_R,R)$.
In the case $R=\C$ we set $V(k):=V(k)_\C$.
Given a vector $\udl k = (k_i) \in (2\N)^n$ we define
\[
V\ipa{\udl k}_R:= \bigotimes_{i=1}^n V(k_i)_R,\qquad \mcP(\udl k)_R:=\bigotimes_{i=1}^n\mcP(k_i)_R.
\]
Notice that the fixed embedding $\iota:K\hookrightarrow B$ provides an isomorphism $B\oti_FK\simeq{\rm M}_2(K)$, that we will also fix throughout the article. For a given $\sigma:F\hookrightarrow\R$, the composition $B\hookrightarrow B\oti_FK\simeq {\rm M}_2(K)$ together with a fixed extension $\sigma_K:K\hookrightarrow\C$ of $\sigma$, gives rise to an embedding $G(F_\sigma)\hookrightarrow\PGL_2(\C)$. This provides an action of $G(F_\iy)$ on $V(\udl k):=V(\udl k)_\C$ and $\mcP(\udl k):=\mcP(\udl k)_\C$, for any $\udl k=(k_{\sigma})\in (2\N)^d$ indexed by $\sigma\mid\iy$. The subspaces $V(\udl k)_{\bar\Q}$ and $\mcP(\udl k)_{\bar\Q}$ are preserved by the action of the subgroup $G(F)\subseteq G(F_\infty)$.

The morphism $K\stackrel{\iota}{\hookrightarrow}B\hookrightarrow B\oti_FK\simeq{\rm M}_2(K)$ maps $k$ to $\bsm k&\\&\bar k\esm$, where $(k\mapsto\bar k)\in \Gal(K/F)$ denotes the non-trivial automorphism. This implies that we have a $T(F_\iy)$-equivariant morphism
\begin{equation}\label{PtoCoverC}
    \mcP(\udl k)\longrightarrow C(T(F_\iy),\C);\qquad \bigotimes_{\sigma} P_{\sigma}\longmapsto\ipa{(t_\sigma)_{\sigma\mid\iy}\mapsto\prod_{\sigma\mid\infty}P_{\sigma}\ipa{1,\sigma_K\ipa{\frac{t_\sigma}{\bar t_\sigma}}}\sigma_K\ipa{\frac{t_\sigma}{\bar t_\sigma}}^{-\frac{k_{\sigma}}{2}}}.
\end{equation}

Note that, if we fix $\bar\Q\hookrightarrow \C_p$, we can associate to a prime $\mfp$ of $F$ above $p$ the set $\Sigma_\mfp$ of embeddings $\sigma: F\hookrightarrow \bar\Q$ such that $v_p\ipa{\bar\sigma(\varpi_\mfp)}>0$. Hence $\sigma\in\Sigma_\mfp$ can be seen as an embedding $\sigma:F_\mfp\hookrightarrow\ovl\Q_p$.
We have fixed a bijection between embeddings $\sigma:F\hookrightarrow\R$ and $\sigma:F_\mfp\hookrightarrow\C_p$ for all $\mfp\mid p$. Similarly, our choice of $\sigma_K$ induces an extension $\sigma_K:K_\mfp\rightarrow\bar\Q_p\subset\C_p$ of $\sigma$.
Thus, for any $\udl k\in (2\N)^d$ we have also a natural action of $G(F_p)$ on $V(\udl k)_{\ovl\Q_p}$ and $\mcP(\udl k)_{\ovl\Q_p}$, and an analogous $T(F_p)$-equivariant morphism 
\begin{equation}\label{PtoCoverQp}
    \mcP(\udl k)_{\ovl\Q_p}\longrightarrow C(T(F_p),\bar\Q_p);\qquad \bigotimes_{\sigma} P_{\sigma}\longmapsto\ipa{(t_\mfp)_{\mfp\mid p}\mapsto\prod_{\mfp\mid p}\prod_{\sigma\in\Sigma_\mfp}P_{\sigma}\ipa{1,\sigma_K\ipa{\frac{t_\mfp}{\bar t_\mfp}}}\sigma_K\ipa{\frac{t_\mfp}{\bar t_\mfp}}^{-\frac{k_{\sigma}}{2}}}.
\end{equation}

\section{Fundamental classes of tori}\label{section:fundcla}

In this section we will define certain fundamental classes associated with the torus $T$.

\subsection{The fundamental class $\eta$ of the torus}\label{section.fundamentalclass}

Let $\mcU=\prod_{\mfq}T(\mfo_{F_\mfq})$ and denote by ${\rm U}:=T(F)\cap \mcU$ the group of relative units. Similarly, we 
define ${\rm U}_+:={\rm U}\cap T(F_\iy)_+$ to be the group of totally positive relative units. By an straightforward argument using \emph{Dirichlet's Unit Theorem}
\[\tno{rank}_\Z {\rm U}=\tno{rank}_\Z {\rm U}_+=(2u+r_{K/F}-1)-(d-1)=u.\] We also define the class group
\[
\Cl(T)_+:=T(\A_{F})/\ipa{T(F)\cdot \mcU\cdot T(F_\iy)_+}\simeq T(\A_{F}^\iy)/\ipa{T(F)_+\cdot \mcU}.
\]

Recall that we have the usual homomorphism $T(F_\iy)_+\ra \R^u$ given by $z\mapsto (\log|\sigma z|)_{\sigma\in\siun}$.
Moreover, as in the proof of \emph{Dirichlet's Unit Theorem}, under this isomorphism the image of ${\rm U}_+$ is a $\Z$-lattice $\Lambda\subset \R^u$.
Thus, we can define a fundamental class $\xi\in H_u(\Lambda,\Z)$ by choosing a generator.

On the other hand, notice that the group $\Cl(T)_+$ fits in the following exact sequence
\begin{center}
\begin{tikzcd}
0 \ar[r] & T(F)_+/{\rm U}_{+} \ra T(\A_F^\iy)/\mcU \ar[r,"p"]  & \Cl(T)_+ \ar[r] & 0. \\[-2em]
\end{tikzcd}
\end{center}
We fix preimages $\ovl t_i\in T(\A_F^\iy)$ for every element $t_i\in \Cl(T)_+$ and we consider the compact set
\[\mcF=\bigcup_i \ovl{t_i}\mcU\subset T(\A_F^\iy).\]
Write $\eta$ for the cap product
\begin{equation}\label{def:fundamentalclass}
\eta=\indi_\mcF\cap \xi\in H_u({\rm U}_+,C(\mcF,\Z)),
\end{equation}
where $\indi_\mcF\in H^0({\rm U}_+,C(\mcF,\Z))$ is the indicator function and $\xi\in H_u({\rm U}_+,\Z)$ is the image of $\xi$ through the corestriction morphism once identified $\Lambda$ with a preimage in ${\rm U}_+$.
By \cite[lemma 3.3]{HerMol1}, we have  an isomorphism of $T(F)$-modules
$\Ind_{U_+}^{T(F)}(C(\mcF,\Z))\simeq C^0_{\rm fc}(T(\A_F),\Z)$. 
Thus, by Shapiro's lemma one may regard
\[\eta\in H_u(T(F),C^0_{\rm fc}(T(\A_F),\Z))\simeq H_u(T(F)_+,C^0(T(\A_F^\iy),\Z)).\]

\subsubsection{The $S$-fundamental classes $\eta^S$}\label{section.fundclS}

As shown in \cite{preprintsanti2}, by means of the fundamental class $\eta$ we can compute certain periods related with critical values of classical L-functions.
In this section we are going to define different types of fundamental classes $\eta^S$ that will be used to define plectic points. At the end of the section we will relate both fundamental classes.

Let $S$ be a set of places $\mfp$ of $F$ above $p$ where $T$ does not split. 
Similarly as above, we consider:
\begin{equation}\label{notation.letusrecall}
\begin{tabular}{cc}
$\mcF^S=\bigsqcup_i \ovl s_i \mcU^S, \quad$ & $\quad \mcU^S =\prod_{\mfq\not\in S} T(\mfo_{F_\mfq}),$ 
\end{tabular}
\end{equation}
where $\ovl s_i\in T(\A_F^{S\cup\iy})$ are representatives of the elements of $\Cl(T)_+^S$ and
\begin{align}\label{definition.cltplusmfp}
\Cl(T)_+^S=T(\A_F^{S\cup\iy})/\ipa{\mcU^S T(F)_+}.
\end{align}
Let us consider also the set of totally positive relative $S$-units ${\rm U}_+^S:=\mcU^S\cap T(F)_+$.
Note that we have an exact sequence
\begin{equation} \label{diagram.exact}
\begin{tikzcd}
0 \ar[r] & T(F_S)/{\rm U}^S_+T(\mfo_{F_S})  \ar[r] & \Cl(T)_+ \ar[r] & \Cl(T)_+^S \ar[r] & 0. \\[-2em]
\end{tikzcd}
\end{equation}

It is clear that the quotient $T(F_S)/T(\mfo_{F_S})$ is finite.
Moreover, we have the natural exact sequence
\begin{equation}\label{diagram.notexact}
\begin{tikzcd}
0 \ar[r] & {\rm U}_+ \ar[r] & {\rm U}_+^S \ar[r] & T(F_S)/T(\mfo_{F_S}). \\[-2em]
\end{tikzcd}
\end{equation}
This implies that the $\Z$-rank of ${\rm U}_+^S$ is $u$.
Hence, similarly as above, we consider the image $\xi^S\in H_{u}({\rm U}_+^S,\Z)$ of $\xi$ through the corestriction morphism.
We have as well
\[
C^0_c(T(\A_F^{S\cup\iy}),\Z)=\Ind_{{\rm U}_+^S}^{T(F)_+} C(\mcF^S,\Z).
\]
Thus, the cap product $\indi_{\mcF^S}\cap \xi^S$ provides an element
\begin{equation}\label{definition.fundamentalclassmfp}
\eta^S:=\indi_{\mcF^S}\cap \xi^S \in H_{u}({\rm U}_+^S,C(\mcF^S,\Z))=H_{u}(T(F)_+,C^0_c(T(\A_F^{S\cup\iy}),\Z)).
\end{equation}

The following result relates the previously defined fundamental classes:
\begin{lemma}\cite[lemma 3.4]{HerMol1}\label{lemma.onedotseven}
We have that
\[[{\rm U}_+^S:{\rm U}_+]\cdot \eta = \eta^S \cap \indi_{T(F_S)},\]
where $\indi_{T(F_S)}\in H^0(T(F),C_c^0(T(F_S),\Z))$ is the constant function.
\end{lemma}

\section{Cohomology classes attached to automorphic forms}\label{GenAutforms}

In this section we will describe cohomology classes associated with automorphic forms by means of the Eichler-Shimura map.


Let $H\subset G(F)$ be a subgroup,
$R$ a ring,
and $S$ a finite set of places of $F$ above $p$.
The subgroup $H$ will usually be $G(F), G(F)_+, T(F)$ or $T(F)_+$.
For any $R[H]$-module $N$ let
\begin{equation}\label{definition.automorphicforms}
\mcA^{S\cup\iy}(N):=\icla{\phi: G(\A^{S\cup\iy}_F)\ra N,\tno{ \begin{tabular}{c} there exists an open compact\\ subgroup $U\subset G(\A_F^{S\cup\iy})$\\ with $\phi(\cdot U)=\phi(\cdot)$ \end{tabular}}}.
\end{equation}
For any pair of $R[H]$-modules $M$ and $N$ we will also write $\mcA^{S\cup\iy}(M,N):=\mcA^{S\cup\iy}(\Hom_R(M,N))$.
Then the space $\mcA^{S\cup\iy}(N)$ has a natural action of $H$ and $G(\A_F^{S\cup\iy})$, namely 
\[(h\phi)(x):=h\cdot \phi(h\inv x),\qquad (y\phi)(x):=\phi(xy),\]
where  $h \in H$ and $x,y \in G(\A_F^{S\cup\iy})$. For an open compact subgroup $U\subset G(\A_F^{S\cup\iy})$ we denote the subspace of $U$-invariant functions by $\mcA^{S\cup\iy}(\Hom_R(M,N))^U$. We define the cohomology space
\[
H_\ast^r(H,\mcA^{S\cup\iy}(N)):=\varinjlim_U H^r(H,\mcA^{S\cup\iy}(N)^U), \qquad U\subset G(\A_F^{S\cup\iy})\;\mbox{ open compact subgroup.}
\]
\begin{remark}\label{remonadmrepcoho}
    Note that $H_\ast^r(H,\mcA^{S\cup\iy}(N))$ is an admissible $G(\A_F^{S\cup\infty})$-representation. Moreover, for a given $U\subset G(\A_F^{S\cup\iy})$ as above, ${\rm Pic}(U):=G(F)_+\backslash G(\A^{S\cup\infty})\slash U$ is finite. Hence, for a set of representatives $\{g_1,\cdots,g_r\}$ of ${\rm Pic}(U)$,
    \[
    \mcA^{S\cup\iy}(N)^U=
    \bigoplus_{i=1}^r{\rm coInd}_{\Gamma_{g_i}}^{G(F)_+}(N);\qquad \Gamma_{g_i}=G(F)_+\cap g_iUg_i^{-1}.
    \]
    This implies that 
    $H^r(G(F)_+,\mcA^{S\cup\iy}(N)^U)=\bigoplus_{i=1}^r H^r(\Gamma_{g_i},N)$. The groups $\Gamma_{g_i}$ are $S$-arithmetic, hence of type (VFL). By \cite[remarque p.101]{Ser71}:
    \begin{itemize}
    \item The $R$-module $H^r(G(F)_+,\mcA^{S\cup\iy}(N)^U)$ is finitely generated, if $N$ is a finitely generated module over $R$ noetherian. 
    \item The functor $N\mapsto H_\ast^r(G(F)_+,\mcA^{S\cup\iy}(N))$ commutes with direct limits. Since any flat module is a direct limit of free modules of finite rank, this implies that the natural morphism
    \[
    H_\ast^r(G(F)_+,\mcA^{S\cup\iy}(N))\otimes_RM\longrightarrow H_\ast^r(G(F)_+,\mcA^{S\cup\iy}(N\oti_RM))
    \]
    is an isomorphism for any flat $R$-module $M$. 
    \end{itemize}
\end{remark}

\subsection{Automorphic cohomology classes}\label{section:autcla}

Let $E/F$ be a modular elliptic curve. Hence, attached to $E$, we have an automorphic form for $\PGL_2/F$ of parallel weight 2. Let us assume that it admits a Jacquet-Langlands lift to $G$, and denote by $\pi$ the corresponding automorphic representation. Let $s:=\#\Sigma_B$. As shown in \cite{ESsanti}, once fixed a character $\lambda:G(F)/G(F)_+\rightarrow\pm 1$, the image through the Harder-Eichler-Shimura isomorphism of any element of such Jacquet-Langlands lift provides a cohomology class in
$H_\ast^s(G(F)_+,\mcA^\iy(\C))^\lambda$,
where the super-index $\lambda$ stands for the subspace where the natural action of $G(F)/G(F)_+$ on the cohomology groups is given by the character $\lambda$. Moreover, since the coefficient ring of the automorphic forms is $\Z$, by remark \ref{remonadmrepcoho}, such a class is the extension of scalars of a class
\[
\phi_\lambda\in H_\ast^s(G(F)_+,\mcA^\iy(\Z))^\lambda.
\]
For general automorphic forms of arbitrary even weight $(\udl k+2)\in (2\N)^d$, the Eichler-Shimura morphism provides a class  
\[
\phi_\lambda\in H_\ast^s(G(F)_+,\mcA^\iy(V(\udl k)_{\ovl\Q}))^\lambda.
\]
The $G(\A_F^\iy)$-representation $\rho$ over $\ovl\Q$ generated by $\phi_\lambda$ satisfies $\rho\otimes_{\ovl\Q}\C\simeq\pi^\iy:=\pi\mid_{G(\A_F^\iy)}$. This implies that $\phi_\lambda$ defines an element 
\[
\varphi_\lambda\in H_\ast^s(G(F)_+,\mcA^\iy(V(\udl k)_{\ovl\Q}))_{\rho}^\lambda.
\]
Notice that $\varphi_\lambda$ can also be seen as an element of $H_\ast^s(G(F)_+,\mcA^\iy(V(\udl k)))_{\pi^\iy}^\lambda$.
\begin{remark}\label{remonAiyS}
By \cite[remark 2.1]{guitart2017automorphic}, for any set $S$ of places above $p$ and any $R[G(F_S)]$-module $M_S$ we have
\[
\Hom_{G(F_S)}\ipa{M_S,\mcA^\iy(V(\udl k)_R)}=\mcA^{S\cup\iy}(M_S,V(\udl k)_R).
\]
In particular, for any compact subgroup $C_S\subset G(F_S)$,
\[
\mcA^\iy(V(\udl k)_R)^{C_S}=\Hom_{G(F_S)}\ipa{{\rm Ind}_{C_S}^{G(F_S)}1,\mcA^\iy(V(\udl k)_R)}=\mcA^{S\cup\iy}\ipa{{\rm Ind}_{C_S}^{G(F_S)}1,V(\udl k)_R}=\mcA^{S\cup\iy}\ipa{{\rm coInd}_{C_S}^{G(F_S)}(V(\udl k)_R)},
\]
by Frobenius reciprocity, where the last equation follows from the $G(F_S)$-equivariant isomorphism
\[
\Hom_R\ipa{{\rm Ind}_{C_S}^{G(F_S)}1,V(\udl k)_R}\stackrel{\simeq}{\longrightarrow}{\rm coInd}_{C_S}^{G(F_S)}(V(\udl k)_R);\qquad \varphi\longmapsto\ipa{g\mapsto g\varphi(g\inv\indi_{C_S})}.
\]
\end{remark}

Write $V_S:=\rho\mid_{G(F_S)}$, with $V_S=\bigotimes_{\mfp\in S}V_\mfp$, and for any ring $R$ we denote by $V_S^R=\bigotimes V_\mfp^R$ the $R[G(F_S)]$-module generated by $\phi_\lambda$. By the above remark and \cite[proposition 1.4]{L-P} we have that 
\begin{equation}\label{autPhi}
\varphi_\lambda\in H_\ast^s(G(F)_+,\mcA^\iy(V(\udl k)_{\ovl\Q}))_{\rho}^\lambda= H_\ast^s(G(F)_+,\mcA^{S\cup \iy}(V_S,V(\udl k)_{\ovl\Q}))_{\rho}^\lambda.
\end{equation}
For any $x^S\in \rho^{S}:=\rho\mid_{G(\A_F^{S\cup\iy})}$, 
the image $\varphi_\lambda(x^S)$ defines an element 
\[
\phi_\lambda^S\in H^s_\ast(G(F)_+,\mcA^{S\cup\iy}(V_S,V(\udl k)_{\ovl \Q}))^\lambda.
\]
We will usually treat $\phi_\lambda^S$ as an element of the cohomology group $H^s_\ast(G(F)_+,\mcA^{S\cup\iy}(V_S^L,V(\udl k)_L))^\lambda$, for $L=\C$, $\ovl\Q_p$ or $\C_p$, since by remark \ref{remonadmrepcoho} and \cite[proposition 1.4]{L-P}
\begin{eqnarray*}
\left(H_\ast^s(G(F)_+,\mcA^{S\cup\iy}(V_S,V(\udl k)_{\ovl \Q}))\otimes_{\ovl\Q}{L}\right)_{\rho^S\otimes L}&=&\left(\Hom_{G(F_S)}(V_S,H_\ast^s(G(F)_+,\mcA^{\iy}(V(\udl k)_{\ovl \Q})))\otimes_{\ovl\Q}{L}\right)_{\rho^S\otimes L}\\
&=&\Hom_{G(F_S)}(V_S^L,H_\ast^s(G(F)_+,\mcA^{\iy}(V(\udl k)_{L})))_{\rho^S\otimes L}\\
&=&H_\ast^s(G(F)_+,\mcA^{S\cup\iy}(V_S^L,V(\udl k)_L))_{\rho^S\otimes L}.
\end{eqnarray*}
The classes $\phi_\lambda^S$ are essential in our construction of anticyclotomic $p$-adic L-functions, and plectic points. 


\subsection{Distributions}\label{ss:distributions}
Assume that $G(F_p)\simeq\PGL_2(F_p)$. For any set $S$ of primes above $p$, let us consider 
\[
\mcL_\mfp
:=\left\{(x,y)\in\mfo_{F_\mfp}\times\mfo_{F_\mfp};\;(x,y)\not \in\mfp\times\mfp\right\},\qquad \mcL_S:=
\prod_{\mfp\in S} \mcL_\mfp.
\]

For any complete $\Z_p$-algebra $R$, and any continuous character
$\chi_S:\mfo_{F_S}^\ti\rightarrow R^\ti$,
let us consider the space of homogeneous functions
\[
C_{\chi_S}(\mcL_S,R)=\left\{f:\mcL_S\rightarrow R,\;\mbox{continuous s.t.}\; f(ax,ay)=\chi_S(a)\cdot f(x,y),\;\tno{ for }a\in\mfo_{F_S}^\times\right\}.
\]
Notice that, if $\chi_S=\prod_{\mfp\in S}\chi_\mfp$ where $\chi_\mfp:\mfo_{F_\mfp}^\ti\rightarrow R^\times$, we have a natural injective morphism
\begin{equation}\label{relPp}
    \bigotimes_{\mfp\in S}C_{\chi_\mfp}(\mcL_\mfp,R)\longrightarrow C_{\chi_S}(\mcL_S,R)
\end{equation}

We write $\D_{\chi_S}(R)$ for the continuous $R$-dual space of $C_{\chi_S}(\mcL_S,R)$, namely,
\[
\D_{\chi_S}(R):=\Hom_{\rm cont}(C_{\chi_S}(\mcL_S,R),R).
\]

For any continuous extension $\hat\chi_S$ of $\chi_S$, namely, a continuous character $\hat\chi_S:F_{S}^\ti\rightarrow R^\ti$ such that $\hat\chi_S\mid_{\mfo_{F_S}^\ti}=\chi_S$, we can consider the induced representation
\[
{\rm Ind}_P^G(\hat\chi_S)=\left\{f:G(F_S)\rightarrow R,\;\mbox{continuous },f\left(\bbm x_1 & y \\ & x_2 \ebm g \right)=\hat\chi_S\left(\frac{x_1}{x_2}\right)\cdot f(g)\right\},
\]
Notice that a choice of the extension $\hat\chi_S$ 
provides a natural $G(F_S)$-action on $C_{\chi_S^{-2}}(\mcL_S,R)$ since, by the Iwasawa decomposition, we have an isomorphism
\begin{eqnarray}
&&\psi_S:C_{\chi_S^{-2}}(\mcL_S,R)\rightarrow {\rm Ind}_P^G(\hat\chi_S),\\\label{eq:varphi}
&&\psi_S(f)\left(\bbm x_1 & y \\ & x_2 \ebm k \right)=\hat\chi_S\left(\frac{x_1}{x_2}\right)\cdot f(c,d)\cdot \chi_S(\det(k)),\qquad k=\bbm a&b\\c&d\ebm\in \GL_2(\mfo_{F_S}).
\end{eqnarray}
This gives rise to a well defined $G(F_S)$-actions on $C_{\chi_S^{-2}}(\mcL_S,R)$ and $\D_{\chi_S^{-2}}(R)$ depending on the extension $\hat\chi_S$.
\begin{definition}
To define an extension $\hat\chi_S:F_S^\ti\rightarrow R^\ti$ of $\chi_S$ amounts to choosing a tuple $\underline{\alpha}=(\alpha_\mfp)_{\mfp\in S}$, where $\alpha_\mfp\in R^\ti$. Indeed, the extension depends on a choice $\alpha_\mfp=\hat\chi_S(\varpi_\mfp)\inv\in R^\ti$, for the fixed uniformizers $\varpi_\mfp$. 
We will denote by 
$\D_{\chi_S^{-2}}(R)_{\underline{\alpha}}$ and $C_{\chi_S^{-2}}(\mcL_S,R)_{\udl\alpha}$
the spaces $\D_{\chi_S^{-2}}(R)$ and $C_{\chi_S^{-2}}(\mcL_S,R)$ endowed with the actions of $G(F_S)$ provided by the corresponding extension. 
\end{definition}

\begin{remark}
Given the extension $\hat\chi_S:F_S^\ti\rightarrow R^\ti$, we can directly describe the action of $g\in G(F_S)$ on $f\in C_{\chi_S^{-2}}(\mcL_S,R)_{\udl\alpha}$ compatible with $\psi_S$. Indeed, for $(c,d)\in\mcL_S$,
\begin{equation}\label{actiononC}
    (gf)(c,d)=\hat\chi_S(x)^{-2}\cdot\hat\chi_S(\det g)\cdot f(x\inv(c,d)g),
\end{equation}
where $x\in F_S^\ti$ is such that $x\inv(c,d)g\in \mcL_S$.
\end{remark}

Assume that $R\subseteq\C_p$ and $\chi_S(a)=a^{-\frac{\udl k}{2}}\chi_S^0(a)$, for some locally constant character $\chi_S^0$ and some $\udl k\in 2\N^{\Sigma_S}$ with $\Sigma_S=\bigcup_{\mfp\in S}\Sigma_\mfp$. Then we can consider the subspace $C_{\chi_S^{-2}}^{\udl k}(\mcL_S,R)$ of locally polynomial functions of homogeneous degree $\udl k$. If $\chi_S=\chi_S^0$ then the subspace is $C_{\chi_S^{-2}}^{0}(\mcL_S,R)$ the set of locally constant functions. For any extension $\hat\chi_S$ as above, we define 
\[
{\rm Ind}_P^G(\hat\chi_S)^{\ast}:=\psi_S\ipa{C_{\chi_S^{-2}}^{\ast}(\mcL_S,R)},\qquad\mbox{ where }\quad \ast=0,\udl k.
\]
We also write $\D^*_{\chi_S^{-2}}(R)$ for the dual space of $C^{*}_{\chi_S^{-2}}(\mcL_p,R)$, where $\ast=\udl k, 0$. 
Notice that, if $R=L$ a field of characteristic 0, 
we have $G(F_S)$-equivariant isomorphisms 
\begin{eqnarray}
\kappa:{\rm Ind}_P^G(\hat\chi_S^0)^{0}\otimes_L \mcP(\udl k)_L&\stackrel{\simeq}{\longrightarrow}&{\rm Ind}_P^G(\hat\chi_S)^{\udl k},\qquad \kappa\ipa{f\otimes P}\bbm a&b\\c&d\ebm=f\bbm a&b\\c&d\ebm\cdot P(c,d)\cdot \ipa{ad-bc}^{-\frac{\udl k}{2}},\\
\kappa^*:\D^{\udl k}_{\chi_S^{-2}}(L)_{\udl \alpha}&\stackrel{\simeq}{\longrightarrow}&\Hom_L\ipa{{\rm Ind}_P^G(\hat\chi_S^0)^{0}, V(\udl k)_L},\label{defkappa}
\end{eqnarray}
where $\udl\alpha=(\alpha_\mfp)_{\mfp\in S}$ with $\alpha_\mfp=\hat\chi_S^0(\varpi_\mfp)\inv\varpi_\mfp^{\frac{\udl k}{2}}$.

\begin{lemma}\label{lemmadmOC}
If we assume that $\alpha_\mfp\in\mfo_{F_\mfp}^\ti$ for all $\mfp\in S$, then any $\mu\in \D^{\udl k}_{\chi_S^{-2}}(\mfo_{\C_p})_{\udl \alpha}\oti_{\mfo_{\C_p}}\C_p$ lifts to a unique continuous distribution $\mu \in \D_{\chi_S^{-2}}(\C_p)_{\udl \alpha}$.
\end{lemma}
\begin{proof}
Notice that the space $\D^{\udl k}_{\chi_S^{-2}}(\mfo_{\C_p})_{\udl \alpha}\oti_{\mfo_{\C_p}}\C_p$ corresponds to the space of bounded distributions, hence continuous. Since $C^{\udl k}_{\chi_S^{-2}}(\mcL_p,\C_p)=C^{\udl k}_{\chi_S^{-2}}(\mcL_S,\mfo_{\C_p})\oti_{\mfo_{\C_p}}\C_p$ is dense in $C_{\chi_S^{-2}}(\mcL_S,\C_p)$, the result follows.
\end{proof}

\subsection{Functions in $C(\mfo_{F_p},R)$ versus functions in $C_{\chi_p^{-2}}(\mcL_p,R)$}

Given a complete $\Z_p$-algebra $R$ and a set $S$ of primes above $p$, write $C(\mfo_{F_S},R)$ and $D(\mfo_{F_S},R)$ for the set of $R$-valued continuous functions of $\mfo_{F_S}$ and its continuous dual, respectively. Write $\mcI_S\subset G(F_S)$ for the usual Iwahori subgroup
\[
\mcI_S=\prod_{\mfp\in S}\mcI_\mfp,\qquad \mcI_\mfp:=\PGL_2(\mfo_{F_\mfp})\cap\bbm\mfo_{F_\mfp}^\ti&\mfo_{F_\mfp}\\
\varpi_\mfp\mfo_{F_\mfp}&\mfo_{F_\mfp}^\ti\ebm/\mfo^\ti_{F_p}.
\]
Given a continuous character $\chi_S:\mfo_{F_S}^\ti\rightarrow R^\ti$, one can define a $\mcI_S$-action on $C(\mfo_{F_S},R)$  by means of the formula
\begin{equation}\label{eqactIwah}
(if)(x)=f\ipa{\frac{b+dx}{a+cx}}\cdot\chi_S(\det i)\cdot\chi_S^{-2}(a+cx),\qquad i=\bsm a&b\\c&d\esm\in \mcI_S,\quad f\in C(\mfo_{F_S},R).
\end{equation}
This provides the usual action on $D(\mfo_{F_S},R)$ given by $(i\mu)(f)=\mu(i\inv f)$. 
Notice that such an action can be extended to the semigroup $\Xi_S\inv$ of inverses of
\[
\Xi_S:=\prod_{\mfp\in S}\Xi_\mfp,\qquad \Xi_\mfp:=\PGL_2(F_\mfp)\cap\bbm\mfo_{F_\mfp}^\ti&\mfo_{F_\mfp}\\
\varpi_\mfp\mfo_{F_\mfp}&\mfo_{F_\mfp}\ebm/\mfo^\ti_{F_p}.
\]
Indeed, if we write $\langle\alpha\rangle=\frac{\alpha}{\prod_{\mfp\in S}\varpi_\mfp^{v_\mfp(\alpha)}}\in\mfo_{F_S}^\ti$ for any $\alpha\in F_S^\ti$, the action of $\Xi_S$ on $C(\mfo_{F_S},R)$ is given by
\[
(gf)(x)=f\ipa{\frac{b+dx}{a+cx}}\cdot\chi_S\ipa{\langle\det g\rangle}\cdot\chi_S^{-2}(a+cx),\qquad g=\bsm a&b\\c&d\esm\in \Xi_S.
\]
\begin{remark}\label{equivmorpCC}
Notice that, with respect to the above action, we have a $\mcI_S$-equivariant morphism
\[
C(\mfo_{F_S},R)\longrightarrow C_{\chi_S^{-2}}(\mcL_S,R);\qquad f\longmapsto \hat f(x,y)=\chi_S(x)^{-2}\cdot f\left(\frac{y}{x}\right)\cdot \indi_{\mfo_{F_S}^\ti\ti\mfo_{F_S}}(x,y).
\]
In fact, given $g=\bsm a&b\\c&d\esm\in\Xi_S$
\begin{eqnarray*}
g\inv\widehat{gf}(x,y)
&=&\hat\chi_p(\alpha)^{-2}\cdot\hat\chi_S\inv(\det g)\cdot \widehat{gf}\ipa{\frac{1}{\alpha\det(g)}(dx-cy,ay-bx)},
\end{eqnarray*}
by \eqref{actiononC}, where $\alpha\in F_S^\ti$ is such that $\alpha\inv(x,y)g\inv\in \mcL_S$. Notice that, since $a\in\mfo_{F_S}^\ti$ and $c\in \prod_{\mfp\in S}\varpi_\mfp$, a necessary condition for $(dx-cy,ay-bx)$ being in  $F_S^\ti(\mfo_{F_S}^\ti\ti\mfo_{F_S})$ is $x\in\mfo_{F_S}^\ti$. Since the support of $\widehat{gf}$ is precisely $\mfo_{F_S}^\ti\ti\mfo_{F_S}$, we conclude 
\[
x\in\mfo_{F_S}^\ti,\qquad
v_\mfp\ipa{\frac{y}{x}-\frac{b}{a}}=v_\mfp(ay-bx)\geq v_\mfp(dx-cy)=v_\mfp\ipa{d-c\frac{y}{x}}=v_\mfp\ipa{\frac{\det(g)}{a}+c\ipa{\frac{b}{a}-\frac{y}{x}}},\quad \mfp\in S.
\]
Thus, $\frac{y}{x}\in \frac{b}{a}+\det(g)\mfo_{F_S}$ and $\alpha=1$. Hence, if we write $U(g):=\frac{b}{a}+\det(g)\mfo_{F_S}$, we compute,
\[
g\inv\widehat{gf}(x,y)
=\hat\chi_S\inv(\det g)\cdot\chi_S^{-2}\ipa{\frac{dx-cy}{\det(g)}}\cdot (gf)\ipa{\frac{ay-bx}{dx-cy}}\cdot\indi_{U(g)}\ipa{\frac{y}{x}}=\frac{\chi_S\ipa{\langle\det g\rangle}}{\hat\chi_S(\det g)}\cdot\widehat{f\cdot\indi_{U(g)}}(x,y).
\]
\end{remark}

For any ${\udl N}=(N_\mfp)_{\mfp\in S}\in\N^{S}$ and $\beta=(\beta_\mfp)_{\mfp\in S}\in\mfo_{F_S}$, we write $\beta+p^{\udl N}\mfo_{F_S}=\prod_{\mfp\in S}(\beta_\mfp+\mfp^{N_\mfp})\subseteq\mfo_{F_S}$. For any $R\subseteq\C_p$ and ${\udl k}\in\N^{\Sigma_S}$, let 
$C^{\udl k}_{\udl N}(\mfo_{F_S},R)\subset C(\mfo_{F_S},R)$ be the submodule of functions that are polynomical of degree less than $\udl k$ when restricted to $\beta+p^{\udl N}\mfo_{F_S}$, for all $\beta\in\mfo_{F_S}$. Notice that, if $\chi_S(a)=a^{-\frac{\udl k}{2}}\chi_S^0(a)$ and ${\udl N}$ is bigger or equal than the conductor of $\chi_S^0$, the $R$-free submodule $C_{\udl N}^{\udl k}(\mfo_{F_S},R)$ is $\mcI_S$ and $\Xi_S$-invariant. 
We write  $D^{\udl k}_{\udl N}(\mfo_{F_S},R)$ for the dual of $C^{\udl k}_{\udl N}(\mfo_{F_S},R)$ endowed with the above actions of $\mcI_S$ and $\Xi_S^{-1}$. We can also consider $C^{\udl k}(\mfo_{F_S},R):=\varinjlim_{{\udl N}} C^{\udl k}_{\udl N}(\mfo_{F_S},R)$ and $D^{\udl k}(\mfo_{F_S},R):=\varprojlim_{{\udl N}} D^{\udl k}_{\udl N}(\mfo_{F_S},R)$.

\begin{remark}\label{remarkHecke}
For any $a\in\mfo_{F_\mfp}$ we write $g_a=\bsm1&a\\&\varpi_\mfp\esm\in\Xi_S$, and for any $\phi\in {\rm coInd}_{\mcI_S}^{G(F_S)}D(\mfo_{F_S},R)$, we define 
\[
U_\mfp \phi(g):=\sum_{i\in \mfo_{F_\mfp}/\mfp}g_i^{-1} \phi\ipa{g_ig},\qquad g\in G(F_S),\qquad\mfp\in S.
\]
Notice that $\mcI_\mfp\bsm1&\\&\varpi_\mfp\esm\mcI_\mfp=\sqcup_{a\in \mfo_{F_\mfp}/\mfp}\mcI_\mfp g_a$, hence we can write $U_\mfp\phi(g)=\mcI_\mfp\bsm1&\\&\varpi_\mfp^{-1}\esm\mcI_\mfp\phi\ipa{\mcI_\mfp\bsm1&\\&\varpi_\mfp\esm\mcI_\mfp g}$ and the above is indeed well defined. By the above observations, in case $R\subseteq\C_p$, it restricts to
\[
U_\mfp: {\rm coInd}_{\mcI_S}^{G(F_S)}D_{\udl N}^{\udl k}(\mfo_{F_S},R)\rightarrow {\rm coInd}_{\mcI_S}^{G(F_S)}D_{\udl N}^{\udl k}(\mfo_{F_S},R),\quad\mbox{ if }\;\chi_S(a)=a^{-\frac{\udl k}{2}}\chi_S^0(a)\;\mbox{ and ${\udl N}$ is the conductor of $\chi_S^0$}.
\]
Moreover, we have a natural $\mcI_S$-equivariant morphism 
\[
{\rm res}:\D_{\chi_S^{-2}}(R)_{{\underline{\alpha}}}\longrightarrow D(\mfo_{F_S},R);\qquad \int_{\mfo_{F_S}}f(z)d({\rm res}\mu)(z):=\int_{\mfo_{F_S}^\ti\ti\mfo_{F_S}}\hat f(x,y)d\mu(x,y),
\]
and one can check, using remark \ref{equivmorpCC}, that this induces a $G(F_S)$-equivariant morphism
\begin{equation}\label{eqvarphiUp}
    \varphi:\D_{\chi_S^{-2}}(R)_{{\underline{\alpha}}}\longrightarrow\bigcap_{\mfp\in S}\ker(U_\mfp-\alpha_\mfp)\subseteq {\rm coInd}_{\mcI_S}^{G(F_S)}D(\mfo_{F_S},R),\qquad \varphi(\mu)(g):={\rm res}(g\mu),\quad {\udl \alpha}=(\alpha_\mfp)_{\mfp\in S}. 
\end{equation} 
\end{remark}
\begin{proposition}\label{propDLDO}
For any complete $\Z_p$-algebra $R$, the following sequence
\begin{equation}\label{eSDDD1}
    0\longrightarrow \D_{\chi_S^{-2}}(R)_{{\underline{\alpha}}}\stackrel{\varphi}{\longrightarrow}{\rm coInd}_{\mcI_S}^{G(F_S)}D(\mfo_{F_S},R)\stackrel{\oplus_\mfp(U_\mfp-\alpha_\mfp)}{\longrightarrow}\bigoplus_{\mfp\in S}{\rm coInd}_{\mcI_S}^{G(F_S)}D(\mfo_{F_S},R)
\end{equation}
is exact. Moreover,
in the case $R=\mfo_{\C_p}$ or $\C_p$, $\chi_S(a)=a^{-\frac{\udl k}{2}}\chi_S^0(a)$, $\alpha_\mfp\in\mfo_{\C_p}^\ti$ and ${\udl N}$ is bigger or equal than the conductor of $\chi_S^0$, we also have an analogous exact sequence:
    \begin{equation}\label{eSDDD2}
0\longrightarrow \D^{\udl k}_{\chi_S^{-2}}(R)_{{\underline{\alpha}}}\stackrel{\varphi}{\longrightarrow}{\rm coInd}_{\mcI_S}^{G(F_S)}D_{\udl N}^{\udl k}(\mfo_{F_S},R)\stackrel{\oplus_\mfp(U_\mfp-\alpha_\mfp)}{\longrightarrow}\bigoplus_{\mfp\in S}{\rm coInd}_{\mcI_S}^{G(F_S)}D_{\udl N}^{\udl k}(\mfo_{F_S},R).
\end{equation}
\end{proposition}
\begin{proof}
This result is a dualized version of \cite[proposition 2.4]{KS} adapted to our settings.
Given ${\udl n}=(n_\mfp)_{\mfp\in S}\in \N^{S}$ and ${\udl b}=(b_\mfp)_{\mfp\in S}\in \PP^1(F_S)$, we write
\[
U_S({\udl b},{\udl n}):=\prod_{\mfp\in S}U_\mfp(b_\mfp,n_\mfp)\subset\mcL_S,\qquad U_\mfp(b_\mfp,n_\mfp)=\left\{\begin{array}{ll}\{(x,y)\in \mcL_\mfp;\;x\in\mfo_{F_\mfp}^\ti,\,yx\inv\equiv b_\mfp\;{\rm mod}\;\varpi_\mfp^{n_\mfp}\};&b\in \mfo_{F_\mfp},\\
\{(x,y)\in \mcL_\mfp;\;y\in\mfo_{F_\mfp}^\ti,\,xy\inv\equiv b_\mfp\inv\;{\rm mod}\;\varpi_\mfp^{n_\mfp}\};&b\not\in \mfo_{F_\mfp}.\end{array}\right.
\]
It turns out that any function in $C_{\chi_S^{-2}}(\mcL_S,R)$ with support $U_S({\udl b},{\udl n})$ must be of the form $f\left(\frac{\frac{y}{x}-{\udl b}}{p^{\udl n}}\right)\chi_S(x)^{-2}$, for some $f\in C(\mfo_{F_S},R)$, where
\[
\left.f\left(\frac{\frac{y}{x}-{\udl b}}{p^{\udl n}}\right)\chi_S(x)^{-2}\right|_{U_\mfp(b_\mfp,n_\mfp)}=\left\{\begin{array}{ll}
f\left(\frac{\frac{y}{x}-{b_\mfp}}{\varpi_\mfp^{n_\mfp}}\right)\chi_\mfp(x)^{-2};&b_\mfp\in\mfo_{F_\mfp},\\
f\left(\frac{\frac{x}{y}-{b_\mfp\inv}}{\varpi_\mfp^{n_\mfp}}\right)\chi_\mfp(y)^{-2};&b_\mfp\not\in\mfo_{F_\mfp}.
\end{array}\right.
\]
If we write ${\udl \alpha}^{-\udl n}:=\prod_{\mfp\in S}\alpha_\mfp^{-n_\mfp}$, then a  simple computation using equation \eqref{actiononC} shows that
\begin{equation}\label{eqaux4.8}
    {\udl \alpha}^{-\udl n}\varphi(\mu)(g_{\udl b,\udl n})(f)=\int_{U_S({\udl b},{\udl n})}f\left(\frac{\frac{y}{x}-{\udl b}}{p^{\udl n}}\right)\chi_S(x)^{-2}d\mu,\qquad g_{\udl b,\udl n}=(g_{b_\mfp,n_\mfp})_{\mfp\in S},\quad g_{b_\mfp,n_\mfp}=\left\{\begin{array}{cc}
    \bbm 1&b_\mfp\\&\varpi_\mfp^{n_\mfp}\ebm, & b_\mfp\in\mfo_{F_\mfp},  \\
    \bbm b_\mfp&1\\\varpi_\mfp^{n_\mfp}&\ebm, & b_\mfp\not\in\mfo_{F_\mfp}.
\end{array}\right.
\end{equation}
This implies that $\varphi$ in \eqref{eSDDD1} is injective, because the vanishing of $\varphi(\mu)$ would imply that $\mu$ annihilates all functions with support in $U_S({\udl b},{\udl n})$, and these open sets cover $\mcL_S$. Analogously, if $\chi_S(a)=a^{-\frac{\udl k}{2}}\chi_S^0(a)$, any function in $C_{\chi_S^{-2}}^{\udl k}(\mcL_S,R)$ is linear combination of functions $f\left(\frac{\frac{y}{x}-{\udl b}}{p^{\udl n}}\right)\chi_S(x)^{-2}$, for some ${\udl n}$ and some $f\in C_{\udl N}^{\udl k}(\mfo_{F_S},R)$. This shows that $\varphi$ in \eqref{eSDDD2} is also injective.

It is clear that ${\rm Im}(\varphi)\subseteq\cap_{\mfp\in S}\ker(U_\mfp-\alpha_\mfp)$ in both sequences \eqref{eSDDD1} and \eqref{eSDDD2}, hence it only remains to show that $\cap_{\mfp\in S}\ker(U_\mfp-\alpha_\mfp)\subseteq {\rm Im}(\varphi)$. Hence, we have to proof that given $\phi\in \cap_{\mfp\in S}\ker(U_\mfp-\alpha_\mfp)$ there exists $\mu_\phi \in \D^{\ast}_{\chi_S^{-2}}(R)_{{\underline{\alpha}}}$ such that $\varphi(\mu_\phi)=\phi$. Taking into account \eqref{eqaux4.8}, one can try to define $\mu_\phi$ by means of the expression
\begin{equation}\label{intassphi}
    \int_{U_S({\udl b},{\udl n})}f\left(\frac{\frac{y}{x}-{\udl b}}{p^{\udl n}}\right)\chi_S(x)^{-2}d\mu_\phi:={\udl \alpha}^{-\udl n}\phi(g_{\udl b,\udl n})(f), \qquad f\in 
    C_\bullet^{\ast}(\mfo_{F_S},R),\quad \ast={\udl k},\emptyset,\;\bullet={\udl N},\emptyset.
\end{equation}
But \eqref{intassphi} provides a well-defined measure $\mu_\phi\in \D^{\ast}_{\chi_S^{-2}}(R)_{\udl \alpha}$, if and only if, it is independent of the choice of ${\udl b}\in \PP^1(F_S)$ in a fixed class of $\PP^1(\mfo_{F_S}/p^{\udl n})$ and it satisfies additivity. The independence of the choice of ${\udl b}$ follows easily from the relation ${\udl \alpha}^{-\udl n}\phi(\gamma g_{\udl b,\udl n})(f)={\udl \alpha}^{-\udl n}\phi(g_{\udl b,\udl n})(\gamma^{-1}f)$, for any $\gamma\in \mcI_S$. To prove additivity, let ${\udl b}_i=(b_{\mfp,i})_{\mfp\in S}\in \PP^1(F_S)$ be such that 
\[
b_{\mfp,i}=b_\mfp,\quad\mbox{ if $\mfp\neq \mfq$ };\qquad\qquad b_{\mfq,i}=\left\{\begin{array}{ll}
   b_\mfq+i\varpi_{\mfq}^{n_\mfq},& \mbox{if $b_\mfq\in\mfo_{F_\mfq}$};  \\
    (b_\mfq^{-1}+i\varpi_{\mfq}^{n_\mfq})^{-1},&\mbox{otherwise.}
\end{array}\right. 
\]
If $\udl n+1\mfq:=((n_\mfp)_{\mfp\neq\mfq},n_\mfq+1)\in \N^S$ then $U_S({\udl b},{\udl n})=\sqcup_{i\in\mfo_{F_\mfq}/\mfq}U_S({\udl b}_i,{\udl n}+1\mfq)$. Thus, additivity follows from
\begin{eqnarray*}
\sum_{i\in\mfo_{F_\mfq}/\mfq}\int_{U_S({\udl b}_i,{\udl n}+1\mfq)}f\left(\frac{\frac{y}{x}-{\udl b}}{p^{{\udl n}}}\right)\chi_S^{-2}(x)d\mu_\phi&=&\sum_{i\in\mfo_{F_\mfq}/\mfq}\int_{U_S({\udl b}_i,{\udl n}+1\mfq)}\left(g_i f\right)\left(\frac{\frac{y}{x}-{\udl b}_i}{p^{{\udl n}+1\mfq}}\right)\chi_S^{-2}(x)d\mu_\phi\\
&=&\sum_{i\in\mfo_{F_\mfq}/\mfq}{\udl \alpha}^{-{\udl n}-1\mfq}\phi(g_{{\udl b}_i,{\udl n}+1\mfp})(g_i f)=\alpha_\mfq^{-1}{\udl \alpha}^{-{\udl n}}\sum_{i\in\mfo_{F_\mfq}/\mfq}g_i^{-1}\phi(g_ig_{{\udl b},{\udl n}})(f)\\
&=&\frac{{\udl \alpha}^{-{\udl n}}}{\alpha_\mfq}U_\mfq\phi(g_{{\udl b},{\udl n}})(f)={\udl \alpha}^{-{\udl n}}\phi(g_{{\udl b},{\udl n}})(f)=\int_{U_S({\udl b},{\udl n})}f\left(\frac{\frac{y}{x}-{\udl b}}{p^{\udl n}}\right)\chi_S^{-2}(x)d\mu_\phi.
\end{eqnarray*}
It is clear that $\varphi(\mu_\phi)$ and $\phi$ coincide at $\sqcup_{\udl b,\udl n}\mcI_Sg_{\udl b,\udl n}=\sqcup_{\udl n}\mcI_S\bsm 1&\\&p^{\udl n}\esm G(\mfo_{F_S})$. By \cite[lemma 12]{SS91}, for all $g\in G(F_S)$, there exists ${\udl m}\in \N^S$ such that $\mcI_S\bsm1&\\&p^{\udl m}\esm\mcI_S g\subseteq \sqcup_{\udl n}\mcI_S\bsm 1&\\&p^{\udl n}\esm G(\mfo_{F_S})$. Since both $\varphi(\mu_\phi)$ and $\phi$ lie in $\cap_{\mfp\in S}\ker(U_\mfp-\alpha_\mfp)$,
\[
\phi(g)={\udl \alpha}^{-\udl m}\mcI_S\bsm1&\\&p^{-\udl m}\esm\mcI_S\phi\left(\mcI_S\bsm1&\\&p^{\udl m}\esm\mcI_Sg\right)={\udl \alpha}^{-\udl m}\mcI_S\bsm1&\\&p^{-\udl m}\esm\mcI_S\varphi(\mu_\phi)\left(\mcI_S\bsm1&\\&p^{\udl m}\esm\mcI_Sg\right)=\varphi(\mu_\phi)(g).
\]
Thus, $\phi=\varphi(\mu_\phi)\in {\rm Im}(\varphi)$ and the result follows.
\end{proof}


\begin{lemma}\label{exhausUpap}
Assume that $R$ is a complete local $\Z_p$-algebra. Then, we the have exact sequence  
\begin{equation}\label{exseqLem4.9}
    0\longrightarrow \D^{\ast}_{\chi_\mfp^{-2}}(R)_{\alpha_\mfp} \stackrel{\varphi}{\longrightarrow}{\rm coInd}_{\mcI_\mfp}^{G(F_\mfp)}D_{\bullet}^{\ast}(\mfo_{F_\mfp},R)\stackrel{(U_\mfp-\alpha_\mfp)}{\longrightarrow} {\rm coInd}_{\mcI_\mfp}^{G(F_\mfp)}D_{\bullet}^{\ast}(\mfo_{F_\mfp},R)\longrightarrow 0,
\end{equation}
in the following situations:
\begin{itemize}
    \item [(a)] The ring $R\subseteq\mfo_{\C_p}$, $\chi_\mfp(a)=a^{-\frac{k_\mfp}{2}}\chi_\mfp^0(a)$ and $(\ast,\bullet)=(k_\mfp,N_\mfp)$ with $N_\mfp\geq{\rm cond}(\chi_\mfp^0)$. 
    
    \item [(b)] The topology of $R$ is defined by its maximal ideal $m_R$ and $(\ast,\bullet)=(\emptyset,\emptyset)$.
\end{itemize}
\end{lemma}
\begin{proof}
By proposition \ref{propDLDO}, it remains to show surjectivity of $U_\mfp-\alpha_\mfp$. Notice that ${\rm coInd}_{\mcI_\mfp}^{G(F_\mfp)}D_{\bullet}^{\ast}(\mfo_{F_\mfp},R)$ is the dual of ${\rm Ind}_{\mcI_\mfp}^{G(F_\mfp)}C_{\bullet}^{\ast}(\mfo_{F_\mfp},R)$. Hence, we can define the adjoint operator of $U_\mfp$:
\[
 U_\mfp^\ast:{\rm Ind}_{\mcI_\mfp}^{G(F_\mfp)}C(\mfo_{F_\mfp},R)\longrightarrow {\rm Ind}_{\mcI_\mfp}^{G(F_\mfp)}C(\mfo_{F_\mfp},R);\qquad 
U_\mfp^\ast f(g)=\sum_i g_i f(g_i\inv g).
\]

 To prove case (a), we will first show that we have the (dual to \eqref{exseqLem4.9}) exact sequence
 \begin{equation}\label{exseqLem4.92}
 0\longrightarrow {\rm Ind}_{\mcI_\mfp}^{G(F_\mfp)}C_{N_\mfp}^{k_\mfp}(\mfo_{F_\mfp},R)\stackrel{(U_\mfp^\ast-\alpha_\mfp)}{\longrightarrow} {\rm Ind}_{\mcI_\mfp}^{G(F_\mfp)}C_{N_\mfp}^{k_\mfp}(\mfo_{F_\mfp},R)\longrightarrow C_{\chi_\mfp^{-2}}(\mcL_\mfp,R)_{\alpha_\mfp}\longrightarrow 0,
 \end{equation}
 where, by proposition \ref{propDLDO}, the third arrow is  given by the adjoint of $\varphi$: 
 \[
 \varphi^\ast:\coker(U_\mfp^\ast-\alpha_\mfp)\longrightarrow C_{\chi_\mfp^{-2}}(\mcL_\mfp,R)_{\alpha_\mfp};\qquad \varphi^\ast(\phi)= \sum_{g\in \mcI_\mfp\backslash G(F_\mfp)}g^{-1}\widehat{\phi(g)}.
 \]
 To show the exactness of \eqref{exseqLem4.92}, it remains to prove that $(U_\mfp^\ast-\alpha_\mfp)$ is injective and $\varphi^\ast$ is an isomorphism.

To show injectivity of $(U_\mfp^\ast-\alpha_\mfp)$ notice that, for a fixed $j\in\mfo_{F_\mfp}/\mfp$, we have $g_i^{-1}g_j\in \mcI_\mfp$, and this implies that
   \begin{equation}\label{eqauxUa}
   (U_\mfp^\ast-\alpha_\mfp)f(g)=\sum_i g_i f(g_i\inv g_jg_j\inv g)-\alpha_\mfp f(g)=\sum_i g_i g_i\inv g_jf(g_j\inv g)-\alpha_\mfp f(g)=q_\mfp\cdot  g_jf(g_j\inv g)-\alpha_\mfp f(g).
   \end{equation}
  Thus, we obtain that $(U_\mfp^\ast-\alpha_\mfp)$
   is injective since $f$ must be compactly supported in $\mcI_\mfp\backslash G(F_\mfp)$. 

To show that $\varphi^\ast$ is an isomorphism, notice that we have the following formula analogous to \eqref{intassphi} for functions
 \begin{equation}\label{eqauxfs}
       {\alpha_\mfp}^{n_\mfp}\cdot f\left(\frac{\frac{y}{x}- b_\mfp}{\varpi_\mfp^{n_\mfp}}\right)\cdot \chi_\mfp(x)^{-2}\cdot \indi_{U_\mfp(b_\mfp,n_\mfp)}(x,y)=g_{b_\mfp,n_\mfp}^{-1} \widehat{f},\qquad f\in C^{k_\mfp}_{N_\mfp}(\mfo_{F_\mfp},R),
   \end{equation} 
   proving that $\varphi^\ast$ is clearly surjective, but also injective since any element in $ \coker(U_\mfp^\ast-\alpha_\mfp)$ has a representative with support in $\sqcup_{b_\mfp,n_\mfp}\mcI_\mfp g_{b_\mfp,n_\mfp}=\sqcup_{n_\mfp}\mcI_\mfp\bsm 1&\\&\varpi_\mfp^{n_\mfp}\esm G(\mfo_{F_\mfp})$ (see \cite[proposition 2.4]{KS}).
   
 Finally, since $C^{k_\mfp}_{\chi_\mfp^{-2}}(\mcL_\mfp,R)$ is direct limit in $n_\mfp$ of functions that are polynomical when restricted to 
$U_\mfp(b_\mfp,n_\mfp)$ for all $b_\mfp$, we obtain that ${\rm Ext}_R^1(C^{k_\mfp}_{\chi_\mfp^{-2}}(\mcL_\mfp,R)_{\alpha_\mfp},R)=0$ by lemma \ref{lemmaonEXT}. Thus, the exactness of the sequence \eqref{exseqLem4.92} implies that exactness of the sequence \eqref{exseqLem4.9}. 

In case (b), notice that we have an isomorphisms
\[
D(\mfo_{F_\mfp},R)\stackrel{\simeq}{\longrightarrow} \Hom_R\ipa{C^0(\mfo_{F_\mfp},R),R},\qquad {\rm coInd}_{\mcI_\mfp}^{G(F_\mfp)}D(\mfo_{F_\mfp},R)\stackrel{\simeq}{\longrightarrow} \Hom_R\ipa{C_c(\mcI_\mfp\backslash G(F_\mfp),C^0(\mfo_{F_\mfp},R)),R}.
\]
Indeed, $C^0(\mfo_{F_\mfp},R)$ is dense in $C(\mfo_{F_\mfp},R)$, and so, any distribution in  $D(\mfo_{F_\mfp},R)$ is characterized by its restriction to $C^0(\mfo_{F_\mfp},R)$. Moreover,  we can extend continuously any $\mu\in \Hom_R\ipa{C^0(\mfo_{F_\mfp},R),R}$ to $D(\mfo_{F_\mfp},R)$ by means of 
\[
\int_{\mfo_{F_\mfp}}fd\mu=\lim_{n}\mu(f_n),\qquad \lim_n f_n=f;\quad f_n\in C^0(\mfo_{F_\mfp},R),
\]
since $f_{n_1}-f_{n_2}\in C^0(\mfo_{F_\mfp},m_R^N)=m_R^N C^0(\mfo_{F_\mfp},R)$ for $n_1,n_2$ big enough, and therefore $\ipa{\mu(f_n)}_n$ is Cauchy. 
Moreover, formula \eqref{eqauxUa} induces a well defined injective morphism
\[
(U_\mfp^\ast-\alpha_\mfp)^0:C_c(\mcI_\mfp\backslash G(F_\mfp),C^0(\mfo_{F_\mfp},R))\longrightarrow C_c(\mcI_\mfp\backslash G(F_\mfp),C^0(\mfo_{F_\mfp},R)).
\]
By equation \eqref{eqauxfs} we can describe $\coker(U_\mfp^\ast-\alpha_\mfp)^0$ as the set of functions in $C_{\chi_\mfp^{-2}}(\mcL_\mfp,R)$ of the form 
\begin{equation}\label{eqrelCLCO}
    \chi_\mfp^{-2}(x)f_1\ipa{\frac{y}{x}}\indi_{\mfo_{F_\mfp}}\ipa{\frac{y}{x}}+\chi_{\mfp}^{-2}(y)f_2\ipa{\frac{x}{\varpi_\mfp y}}\indi_{\mfp}\ipa{\frac{x}{y}},\qquad f_1,f_2\in C^0(\mfo_{F_\mfp},R).
\end{equation}
Thus, we deduce that $\coker(U_\mfp^\ast-\alpha_\mfp)^0$ is free, and ${\rm Ext}_R^1(\coker(U_\mfp^\ast-\alpha_\mfp)^0,R)=0$. Dualizing the exact sequence 
\[
0\longrightarrow C_c(\mcI_\mfp\backslash G(F_\mfp),C^0(\mfo_{F_\mfp},R))\stackrel{(U_\mfp^\ast-\alpha_\mfp)^0}{\longrightarrow} C_c(\mcI_\mfp\backslash G(F_\mfp),C^0(\mfo_{F_\mfp},R))\longrightarrow \coker(U_\mfp^\ast-\alpha_\mfp)^0\longrightarrow 0,
\]
we obtain that $U_\mfp-\alpha_\mfp$ is then surjective.
\end{proof}


\begin{lemma}\label{lemmaonEXT}
    Let $M=\varinjlim_\alpha M_\alpha$ be a direct limit of free $R$-modules, where the transition maps $\imath_{\alpha,\alpha'}:M_\alpha\rightarrow M_{\alpha'}$ are injective and $\coker(\imath_{\alpha,\alpha'})$ is $R$-free. 
    Then, for any $R$-module $N$, we have that ${\rm Ext}_R^1(M,N)=0$.
\end{lemma}
\begin{proof}
     We have a spectral sequence with $E_2$-term $(R^i\varprojlim_\alpha){\rm Ext}_R^i(M_\alpha,N)$ and converging to ${\rm Ext}_R^{i+j}(M,N)$. On the other hand, ${\rm Ext}_R^1(\coker(\imath_{\alpha,\alpha'}),N)=0$, hence the natural morphism provided by $\coker(\imath_{\alpha,\alpha'})$ 
     \[
     \imath_{\alpha,\alpha'}^\ast:\Hom_R(M_{\alpha'},N)\longrightarrow \Hom_R(M_\alpha,N),
     \]
     is surjective. This implies that the inverse system $\Hom_R(M_\alpha,N)_\alpha$ satisfies the Mittag-Leffler condition, and therefore $R^1\varprojlim\Hom_R(M_\alpha,N)_\alpha=0$. We deduce that ${\rm Ext}_R^1(M,N)=\varprojlim_\alpha {\rm Ext}_R^1(M_\alpha,N)=0$. 
\end{proof}


\subsubsection{Koszul complexes}

Assume that $R$ is a complete $\Z_p$-algebra and we have one of the following settings:
\begin{itemize}
    \item [(a)] The ring $R\subseteq\mfo_{\C_p}$, the character $\chi_p(a)=a^{-\frac{\udl k}{2}}\chi_p^0(a)$ and we have $\udl N=(N_\mfp)_{\mfp\mid p}$ with $N_\mfp\geq{\rm cond}(\chi_\mfp^0)$.
    
    \item [(b)] The topology of $R$ is defined by its maximal ideal $m_R$.
\end{itemize}
For any set $S$ of primes above $p$,
let $D_S$ be either 
$D^{k_S}_{N_S}(\mfo_{F_S},R)$, with $N_S=(N_\mfp)_{\mfp\in S}$ and $k_S=(k_\sigma)_{\sigma\in \Sigma_S}$, in case (a), or $D(\mfo_{F_S},R)$ in case (b). Similarly, let $C_S$ be either 
$C^{k_S}_{N_S}(\mfo_{F_S},R)$ or $C(\mfo_{F_S},R)$. For any tuple ${\udl \alpha}=(\alpha_\mfp)_{\mfp\mid p}\in (R^\times)^{g_p}$, if we write $y_\mfp:=U_\mfp-\alpha_\mfp$, then the polynomial ring $R[y_\mfp]_{\mfp\mid p}$ acts naturally on ${\rm coInd}_{\mcI_p}^{G(F_p)}D_p$. We consider the \emph{Koszul complex}:
\[
C^\bullet:=\Hom_R\ipa{\bigwedge^\bullet R^{g_p},{\rm coInd}_{\mcI_p}^{G(F_p)}D_p},
\]
where, for the canonical basis $\{e_{\mfp}\}_{\mfp}$ of $R^{g_p}$, the boundary maps are given by
\[
d^k:C^{k-1}\longrightarrow C^{k};\quad d^k(\varphi)(e_{\mfp_{i_1}}\wedge\cdots\wedge e_{\mfp_{i_k}})=\sum_{j=1}^k(-1)^{j+1}y_{\mfp_{i_j}}\varphi(e_{\mfp_{i_1}}\wedge\cdots\wedge\hat e_{\mfp_{i_j}}\wedge\cdots\wedge e_{\mfp_{i_k}}).
\]
\begin{proposition}\label{propKoszul}
    The cohomology of the Koszul complex $H^k(C^\bullet)=0$, for all $k>0$.
\end{proposition}
\begin{proof}
    By \cite[X.9.6 remarque 4]{Bour}, it is enought to prove that, for $\mfp$ and any $S$ such that $\mfp\not\in S$,
    \[
    U_\mfp-\alpha_\mfp:\bigcap_{\mfq\in S}\ker(U_\mfq-\alpha_\mfq)\subseteq {\rm coInd}_{\mcI_p}^{G(F_p)}D_p\longrightarrow  \bigcap_{\mfq\in S}\ker(U_\mfq-\alpha_\mfq)\subseteq {\rm coInd}_{\mcI_p}^{G(F_p)}D_p,
    \]
    is surjective. 
    
    Notice that in case (a), ${\rm coInd}_{\mcI_p}^{G(F_p)}D_p=\Hom_R\ipa{{\rm Ind}_{\mcI_S}^{G(F_S)}C_S,{\rm coInd}_{\mcI_{p\setminus S}}^{G(F_{p\setminus S})}D_{p\setminus S}}$. Hence, by proposition \ref{propDLDO},
    \begin{eqnarray*}
        \bigcap_{\mfq\in S}\ker(U_\mfq-\alpha_\mfq)&=&\Hom_R\ipa{C^{k_S}_{\chi_S^{-2}}(\mcL_S,R)_{\alpha_S},{\rm coInd}_{\mcI_{p\setminus S}}^{G(F_{p\setminus S})}D_{p\setminus S}}\\
        &=&\Hom_R\ipa{C^{k_S}_{\chi_S^{-2}}(\mcL_S,R)_{\alpha_S}\oti_R{\rm Ind}_{\mcI_{p\setminus S\cup\mfp}}^{G(F_{p\setminus S\cup\mfp})}C_{p\setminus S\cup\mfp},{\rm coInd}_{\mcI_{\mfp}}^{G(F_{\mfp})}D_{\mfp}},
    \end{eqnarray*}
    where $\chi_S=\prod_{\mfp\in S}\chi_\mfp$.
    Thus, since ${\rm Ind}_{\mcI_{p\setminus S\cup\mfp}}^{G(F_{p\setminus S\cup\mfp})}C_{p\setminus S\cup\mfp}$ is $R$-free, applying lemmas \ref{exhausUpap} and \ref{lemmaonEXT} the result follows.

    In case (b), we have 
    similarly by proposition \ref{propDLDO},
    \begin{eqnarray*}
        \bigcap_{\mfq\in S}\ker(U_\mfq-\alpha_\mfq)&=&\Hom_R\ipa{C^{0}_{\chi_S^{-2}}(\mcL_S,R),{\rm coInd}_{\mcI_{p\setminus S}}^{G(F_{p\setminus S})}D_{p\setminus S}}\\
        &=&\Hom_R\ipa{C^{0}_{\chi_S^{-2}}(\mcL_S,R)\oti_RC_c\ipa{\mcI_{p\setminus S\cup\mfp}\backslash G(F_{p\setminus S\cup\mfp}),C_{p\setminus S\cup\mfp}^0},{\rm coInd}_{\mcI_{\mfp}}^{G(F_{\mfp})}D_{\mfp}},
    \end{eqnarray*}
    where $C^{0}_{\chi_S^{-2}}(\mcL_S,R)$ is the set of functions in $C_{\chi_S^{-2}}(\mcL_S,R)$ that restricted to $\mcL_\mfp$ are of the form \eqref{eqrelCLCO}.
    The result follows from lemma \ref{exhausUpap} since $C^{0}_{\chi_S^{-2}}(\mcL_S,R)\oti_RC_c\ipa{\mcI_{p\setminus S\cup\mfp}\backslash G(F_{p\setminus S\cup\mfp}),C_{p\setminus S\cup\mfp}^0}$ is $R$-free.
\end{proof}

\subsection{Overconvergent modular symbols}\label{OCmodSymb}

In this section we extend ordinary modular symbols to overconvergent modular symbols. 
Assume that $\pi_\mfp$ is 
principal series or Steinberg for all $\mfp\mid p$. Thus, we have a projection 
\[
r:{\rm Ind}_P^G(\hat\chi_p^0)^0\rightarrow V_p^{\C_p}=\bigotimes_{\mfp\mid p}V_\mfp^{\C_p},\qquad \hat\chi_p^0=\prod_{\mfp\mid p}\hat\chi_\mfp^0,
\]
as $G(F_p)$-representations for some locally constant character $\hat\chi_p^0:F_p^\ti\rightarrow\C_p^\ti$. Write $\chi_p^0$ for the restriction of $\hat\chi_p^0$ to $\mfo_{F_p}^\ti$. 
Similarly, we write $\hat\chi_p=z^{-\frac{\udl k}{2}}\hat\chi_p^0:F_p^\ti\rightarrow\C_p^\ti$ and $\chi_p$ for its restriction to $\mfo_{F_p}^\ti$.


Under the above assumptions, our modular symbol
\[
\phi_\lambda^p\in H_\ast^s\ipa{G(F)_+,\mcA^{p\cup\infty}(V_p^{\C_p},V(\udl k)_{\C_p})}^\lambda
\]
satisfies by equation \eqref{defkappa}
\[
r^\ast\phi_\lambda^p\in H_\ast^s\ipa{G(F)_+,\mcA^{p\cup\infty}\ipa{{\rm Ind}_P^G(\hat\chi_p^0)^0,V(\udl k)_{\C_p}}}^\lambda\stackrel{(\kappa^\ast)^{-1}}{\simeq}H_\ast^s\ipa{G(F)_+,\mcA^{p\cup\infty}\ipa{\D^{\udl k}_{\chi_p^{-2}}(\C_p)_{\udl \alpha}}}^\lambda 
\]
where $\udl\alpha=(\alpha_\mfp)_\mfp$ with $\alpha_\mfp=\hat\chi_\mfp^0(\varpi_\mfp)\inv\varpi_\mfp^{\frac{\udl k}{2}}=\hat\chi_\mfp(\varpi_\mfp)\inv$.

\begin{proposition}\label{OCmodsymb}
Assume that for all $\mfp\mid p$ we have $\alpha_\mfp\in\mfo_{\C_p}^\times$,
then any cohomology class $\phi_\lambda^p$ extends to a unique 
\[
\hat\phi_\lambda^p\in H_\ast^s\ipa{G(F)_+,\mcA^{p\cup\infty}(\D_{\chi_p^{-2}}(\C_p)_{\udl\alpha})}^\lambda.
\]
Namely, 
\[
\kappa^\ast\hat\phi_\lambda^p=r^\ast \phi_\lambda^p\in H_\ast^s\ipa{G(F)_+,\mcA^{p\cup\infty}({\rm Ind}_P^G(\hat\chi_p^0)^0,V(\underline{k})_{\C_p})}^\lambda.
\]
\end{proposition}
\begin{proof}

By propositions \ref{propDLDO} and \ref{propKoszul}, the Koszul complex provides a finite right resolution of $\D^{\udl k}_{\chi_p^{-2}}(R)_{{\underline{\alpha}}}$ by powers of the induced representation ${\rm coInd}_{\mcI_p}^{G(F_p)}D_{\udl N}^{\udl k}(\mfo_{F_S},R)$, where ${\udl N}$ is the conductor of $\chi_p^0$ and $R=\mfo_{\C_p}$ or $\C_p$.
By Remark \ref{remonadmrepcoho} and Remark \ref{remonAiyS},
\begin{eqnarray*}
    H_\ast^s\ipa{G(F)_+,\mcA^{p\cup\iy}\ipa{{\rm coInd}_{\mcI_p}^{G(F_p)}D_{\udl N}^{\udl k}(\mfo_{F_S},\C_p)}}^{\lambda}&=&H_\ast^s\ipa{G(F)_+,\mcA^{\iy}\ipa{D^{\udl N}_{\chi_{p}}(\mfo_{F_{p}},\C_p)}^{ \mcI_{p\setminus S}}}^{\lambda}\\
    &=&H_\ast^s\ipa{G(F)_+,\mcA^{\iy}\ipa{D^{\udl N}_{\chi_{p}}(\mfo_{F_{p}},\mfo_{C_p})}^{ \mcI_{p\setminus S}}}^{\lambda}\oti_{\mfo_{\C_p}}\C_p\\
    &=&H_\ast^s\ipa{G(F)_+,\mcA^{p\cup\iy}\ipa{{\rm coInd}_{\mcI_p}^{G(F_p)}D_{\udl N}^{\udl k}(\mfo_{F_S},\mfo_{C_p})}}^{\lambda}\oti_{\mfo_{\C_p}}\C_p.
\end{eqnarray*}
Thus, we deduce 
\begin{equation}\label{isooncohoDs}
   H_\ast^s\ipa{G(F)_+,\mcA^{p\cup\iy}\ipa{\D^{\udl k}_{\chi_p^{-2}}(\mfo_{\C_p})_{\udl \alpha}\oti_{\mfo_{\C_p}}\C_p}}^{\lambda}=H_\ast^s\ipa{G(F)_+,\mcA^{p\cup\iy}\ipa{\D^{\udl k}_{\chi_p^{-2}}(\C_p)_{\udl \alpha}}}^{\lambda}.
\end{equation}
By Lemma \ref{lemmadmOC}, the natural inclusion $\D^{\udl k}_{\chi_p^{-2}}(\mfo_{\C_p})_{\udl \alpha}\oti_{\mfo_{\C_p}}\C_p\hookrightarrow \D^{\udl k}_{\chi_p^{-2}}(\C_p)_{\udl \alpha}$ factors through
\[
\D^{\udl k}_{\chi_p^{-2}}(\mfo_{\C_p})_{\udl \alpha}\oti_{\mfo_{\C_p}}\C_p\longrightarrow \D_{\chi_p^{-2}}(\C_p)_{\udl \alpha}\stackrel{\rm res}{\longrightarrow}\D^{\udl k}_{\chi_p^{-2}}(\C_p)_{\udl \alpha}.
\]
Thus \eqref{isooncohoDs} implies $H_\ast^s\ipa{G(F)_+,\mcA^{p\cup\iy}\ipa{\D_{\chi_p^{-2}}(\C_p)_{\udl \alpha}}}^{\lambda}=H_\ast^s\ipa{G(F)_+,\mcA^{p\cup\iy}\ipa{\D^{\udl k}_{\chi_p^{-2}}(\C_p)_{\udl \alpha}}}^{\lambda}$, and the result follows.

\end{proof}

\begin{remark}
If the condition $\alpha_\mfp\in\mfo_{\C_p}^\ti$ is fulfilled, we say that the modular symbol $\phi_\lambda^p$ is \emph{ordinary}. In this paper, we will restrict ourselves to the ordinary setting.
\end{remark}

\begin{remark}\label{remintMS}
    Let $S$ be a set of primes above $p$ and write $\hat\chi_S=\prod_{\mfp\in S}\hat\chi_\mfp$ analogously as above.
    Assuming that $\udl k=(0,\cdots,0)$, we have 
    \[
    \D^{\udl k}_{\chi_S^{-2}}(\mfo_{\C_p})_{\udl \alpha}\oti_{\mfo_{\C_p}}\C_p=\D^{0}_{\chi_S^{-2}}(\mfo_{\C_p})_{\udl \alpha}\oti_{\mfo_{\C_p}}\C_p=\Hom\ipa{{\rm Ind}_P^G(\hat\chi_S^0)_{\mfo_{\C_p}}^0,\mfo_{\C_p}}\oti_{\mfo_{\C_p}}\C_p.
    \]
    where ${\rm Ind}_P^G(\hat\chi_S^0)_{\mfo_{\C_p}}^0$ is the induced representation of $\hat\chi_S=\hat\chi_S^0$ with coefficients in $\mfo_{\C_p}$, and $\D^{0}_{\chi_S^{-2}}(\mfo_{\C_p})$ are the corresponding locally constant distributions of $\mcL_{S}:=\prod_{\mfp\in S}\mcL_\mfp$.
    Similarly as in the proof of Proposition \ref{OCmodsymb} and by Remark \ref{remonadmrepcoho}, we have
    \begin{equation}\label{eqBchangeI}
        H_\ast^s\ipa{G(F)_+,\mcA^{S\cup\infty}({\rm Ind}_P^G(\hat\chi_S^0)^0,\C_p)}^\lambda\simeq H_\ast^s\ipa{G(F)_+,\mcA^{S\cup\infty}({\rm Ind}_P^G(\hat\chi_S^0)_{\mfo_{\C_p}}^0,\mfo_{\C_p})}^\lambda\oti_{\mfo_{\C_p}}\C_p.
    \end{equation}
    Moreover, in case that at $\mfp\mid p$ the representation is Steinberg, by definition $\hat \chi_\mfp^0=\varepsilon_\mfp\mid_P$, for a character $\varepsilon_\mfp=(\pm1)^{v_\mfp\circ\det}$, and $V_\mfp={\rm St}(F_\mfp)(\varepsilon_\mfp):={\rm Ind}_P^G(\hat\chi_\mfp^0)^0/\varepsilon_\mfp$. This implies that we have the exact sequence 
    \[
    0\longrightarrow \Hom\ipa{V^R_{\mfp},R}\longrightarrow\Hom\ipa{{\rm Ind}_P^G(\hat\chi_\mfp^0)_{R}^0,R}\longrightarrow R(\varepsilon_\mfp)\longrightarrow 0,
    \]
    where $R=\C_p$ or $\mfo_{\C_p}$. Thus, equation \eqref{eqBchangeI} and a simple induction imply that 
    \[
    H_\ast^s\ipa{G(F)_+,\mcA^{p\cup\infty}(V_p^{\C_p},\C_p)}^\lambda=H_\ast^s\ipa{G(F)_+,\mcA^{p\cup\infty}(V_p^{\mfo_{\C_p}},\mfo_{\C_p})}^\lambda\oti_{\mfo_{\C_p}}\C_p.
    \]
    Hence, up to multiplying by a suitable constant, we can assume that $\phi_\lambda^p\in H_\ast^u\ipa{G(F)_+,\mcA^{p\cup\infty}(V_p^{\mfo_{\C_p}},\mfo_{\C_p})}^\lambda$. 
    
    If $\phi_\lambda^p$ is associated with an ordinary elliptic curve, we have that $\alpha_\mfp\in \Z_p^\times$ for all $\mfp$, and  our modular symbol has coefficients in $\Q_p$. Thus, the same arguments show that (up-to-constant) $\phi_\lambda^p\in H_\ast^s\ipa{G(F)_+,\mcA^{p\cup\infty}(V_p^{\Z_p},\Z_p)}^\lambda$. In fact, for any set of primes $S$ above $p$, an elliptic curve is associated with a modular symbol 
    \[\phi_\lambda^S\in H_\ast^s\ipa{G(F)_+,\mcA^{S\cup\infty}(V_S^{\Z_p},\Z_p)}^\lambda,\qquad\mbox{ where }\quad V_S^{\Z_p}=\bigotimes_{\mfp\in S}V_\mfp^{\Z_p}.
    \]
\end{remark}

\subsection{Pairings}\label{section:pairings}

In order to relate the $p$-arithmetic cohomology groups defined in \S \ref{section:autcla} and the homology groups defined in \S \ref{section:fundcla}, we will define certain pairings that will allow us to perform cap products. 
For this purpose, we assume the following hypothesis:
\begin{hypothesis}\label{hypothesis}
Assume that $\Sigma_B=\Sigma_{\tno{un}}(K/F)$. Hence, in particular, $u=s$ and $G(F)/G(F)_+=T(F)/T(F)_+$. 
\end{hypothesis}

As above, let $S$ be a set of primes $\mfp$ above $p$.
For any $T(F)$-modules $M$ and $N$, let us consider the $T(F)_+$-equivariant pairing
\begin{equation}\label{equation.pairingplus}
\begin{tikzcd}
\langle\cdot,\cdot\rangle_+: &[-3em] C_c^0(T(\A_F^{S\cup\iy}),M)\ti\mcA^{S\cup\iy}(M,N) \ar[r] &[-2em] N, \\[-2em]
& (f,\phi) \ar[r,mapsto] & \langle f,\phi\rangle_+ := \int_{T(\A_F^{S\cup\iy})}\phi(t)(f(t)) d^\ti t, \\[-2em]
\end{tikzcd}
\end{equation}
where $d^\times t$ is the corresponding Haar measure. This implies that, once fixed a character $\lambda:G(F)/G(F)_+=T(F)/T(F)_+\rightarrow\pm 1$, it induces a well-defined $T(F)$-equivariant pairing
\begin{equation}\label{equation.pairingnoplus}
\begin{tikzcd}
\langle\cdot,\cdot\rangle:&[-4em] \Ind_{T(F)_+}^{T(F)} C_c^0(T(\A_F^{S\cup\iy}),M)\ti \mcA^{S\cup\iy}(M,N)(\lambda) \longrightarrow N,\\[-2em]
&\langle f,\phi\rangle:=[T(F):T(F)_+]\inv\sum_{\ovl t\in T(F)/T(F)_+}t\cdot \langle f(t\inv),t\inv \phi\rangle_+, \\[-2em]
\end{tikzcd}
\end{equation}
where $\mcA^{S\cup\iy}(M,N)(\lambda)$ is the twist of the $T(F)$-representation $\mcA^{p\cup\iy}(M,N)$ by the character $\lambda$.

Assume that $M=C_c^0(T(F_S),R)\otimes_R V$ for a ring $R$ and a finite rank $R$-module $V$, then we have 
\[
\Ind_{T(F)_+}^{T(F)}C_c^0(T(\A_F^{S\cup\iy}),M)=C^0_{\rm fc}(T(\A_F),R)\oti_R V.
\]
This implies that \eqref{equation.pairingnoplus} provides a final $T(F)$-equivariant pairing 
\begin{equation}\label{equation.finalpairing}
\begin{tikzcd}
\langle \cdot|\cdot \rangle: C_{\rm fc}^0(T(\A_F),R)\oti_R V\ti \mcA^{S\cup\iy}(C_c^0(T(F_S),R)\oti_R V,N)(\lambda) \longrightarrow N, \\[-2em]
\langle f_S\otimes f^S\oti v|\phi\rangle =[T(F):T(F)_+]\inv\sum_{\ovl x\in T(F)/T(F)_+}\lambda(x)\inv\int_{T(\A^{S\cup\iy})}f^S(x,t)\cdot \phi(t)(f_S\oti v)d^\ti t. 
\end{tikzcd}
\end{equation}

All the pairings above induce cap products in $H$-(co)homology by their $H$-equivariance.
Now denote by $f_\lambda$ the projection of $f$
to the subspace
\[C_{\rm fc}^0(T(\A_F),R)_\lambda:=\icla{f\in C_{\rm fc}^0(T(\A_F),R)\;\tno{ with }f(x,\cdot)=\lambda(x)f(1,\cdot);\;\mbox{ for all }x\in T(F_\infty)}.\]
One easily computes that
\[
\langle f\oti v|\phi\rangle=\langle f_\lambda\oti v|\phi\rangle=\langle f_\lambda\mid_{T(\A^{\iy})}\oti v,\phi\rangle_+, \qquad v\in V,\quad f\in C_{\rm fc}^0(T(\A_F),R),\quad\phi\in \mcA^{S\cup\iy}(C_c^0(T(F_S),R),N)(\lambda).
\]
Since we can identify
$H_\ast^u(G(F)_+,\bullet)^\lambda\simeq H_\ast^u(G(F),\bullet(\lambda))$,
we deduce that for all $f\oti v\in H_u(T(F),C_{\rm fc}^0(T(\A_F),R)\oti_R V)$ and $\phi\in H_\ast^u(T(F)_+,\mcA^{S\cup\iy}(C_c^0(T(F_S),R)\oti_R V,N))^\lambda$,
\begin{equation}\label{equation:capprod}
(f\oti v)\cap\phi=(f_\lambda\oti v)\cap\phi=(f_\lambda\mid_{T(\A_F^{\iy})}\oti v)\cap \tno{res}^{T(F)}_{T(F)_+}\phi\in N, 
\end{equation}
where $\tno{res}^{T(F)}_{T(F)_+}$ is the restriction morphism and the cap products are the induced by \eqref{equation.pairingplus},\eqref{equation.finalpairing}, respectively.

\section{Anticyclotomic $p$-adic L-functions}\label{antipLfunct}

In this section we will define the anticyclotomic $p$-adic L-functions associated with $p$, $T$ and the automorphic cohomology class $\phi_\lambda^p$. From this point on we will assume that hypothesis \ref{hypothesis} is fulfilled.

\subsection{Defining the distribution}\label{Secdefdist}

Let $C_{\udl k}(T(F_p),  \ovl\Q_p)$ be the space of $\ovl\Q_p$-valued locally polynomial functions of $T(F_p)$ of degree less that $\udl k$. These correspond to functions $f:T(F_p)\rightarrow \ovl\Q_p$ such that, in every small enough neighbourhood $U\subseteq T(F_p)$,
\[
f(s_p)=\sum_{|m_{\sigma}|\leq\frac{k_{\sigma}}{2}}a_{\udl m}(U)\prod_{\mfp\mid p}\prod_{\sigma\in\Sigma_\mfp}\sigma_K \ipa{\frac{s_\mfp}{\bar s_\mfp}}^{m_{\sigma}},
\]
where $\udl{m}=(m_{\sigma})\in \Z^d$, $a_{\udl{m}}(U)\in\ovl\Q_p$ and $s_p=(s_\mfp)_{\mfp\mid p}\in U$. By equation \eqref{PtoCoverQp}, there is a $T(F_p)$-equivariant isomorphism 
\begin{equation}\label{eqlocpol}
 C^0(T(F_p), \ovl\Q) \oti_{\bar\Q} \mcP(\udl k)_{\bar\Q}\oti\ovl\Q_p \longrightarrow C_{\udl k}(T(F_p), \ovl\Q_p).
\end{equation}

Let $\mcG_T$ be the Galois group of the abelian extension of $K$ associated with $T$.
By Class Field Theory, there is a continuous surjective morphism
\begin{align}\label{definition.galoisT}
\Theta:\ipa{T(F_\iy)/T(F_\iy)_+\ti T(\A_F^\iy)}/T(F) \ra \mcG_T,
\end{align}
whose kernel is the connected component of $\ipa{T(F_\iy)/T(F_\iy)_+\ti T(\A_F^\iy)}/T(F)$.

Let us consider the subspace of $C(T(\A_F),\ovl\Q_p)$:
\[
C_{\udl k}(T(\A_F),\ovl\Q_p):=\cla{f:T(\A_F^p)\rightarrow C_{\udl k}(T(F_p),\ovl\Q_p),\mbox{ locally constant}},
\]
and write also $C_{\udl k}(\mcG_T,\ovl\Q_p)$ for the subspace of continuous functions $f:\mcG_T\rightarrow\ovl\Q_p$ such that $\Theta^*(f)\in C_{\udl k}(T(\A_F),\ovl\Q_p)$.
The pullback of $\Theta$ together with the cap product by $\eta\in H_u(T(F),C^0_{\rm fc}(T(\A_F),\Z))$ provide the following morphism
\begin{equation}\label{definition.deltamfp}
\begin{tikzcd}
 &[-3em] C_{\udl k}(\mcG_T,\ovl\Q_p) \ar[r,"\Theta^*"] & H^0(T(F),C_{\udl k}(T(\A_F),\ovl\Q_p)) \ar[r,"\cap\eta"] & H_u(T(F),C_{\udl k,\rm fc}(T(\A_F),\ovl\Q_p)). \\[-2em]
\end{tikzcd}
\end{equation}
where $C_{\udl k,\rm fc}(T(\A_F),\ovl\Q_p))$ is the subspace of functions in $f\in C_{\udl k}(T(\A_F),\ovl\Q_p))$ such that $f(x,\cdot)$ is compactly supported for all $x\in T(F_\infty)$, and the cap product is given by usual multiplication of functions in $C(T(\A_F),\ovl\Q_p)$. Notice that \eqref{eqlocpol} provides isomorphisms
\begin{equation}\label{descCkc}
C_{\udl k}(T(\A_F),\ovl\Q_p))=
 C^0(T(\A_F), \ovl\Q)\oti_{\ovl\Q} \mcP(\udl k)_{\ovl\Q}\oti\ovl\Q_p,\qquad
C_{\udl k,\rm fc}(T(\A_F),\ovl\Q_p))=
 C_{\rm fc}^0(T(\A_F), \ovl\Q)\oti_{\ovl\Q} \mcP(\udl k)_{\ovl\Q}\oti\ovl\Q_p .    
\end{equation}

In order to define our distribution, we will construct a $T(F_p)$-equivariant morphism $\delta_p: C_c^0(T(F_p),\ovl\Q) \rightarrow V_p$. Given such a $\delta_p$ and $\phi_\lambda^p\in H_\ast^u\ipa{G(F)_+,\mcA^{p\cup\iy}(V_p,V(\udl k)_{\ovl\Q})}^\lambda$, we  define the distribution $\mu_{\phi^p_\lambda}$  as:
\begin{align}\label{definition.distribution}
\mu_{\phi^p_\lambda}(g)=\int_{\mcG_T}g d\mu_{\phi^p_\lambda}:=\ipa{\Theta^\ast g\cap\eta}\cap\delta_p^\ast\phi^p_\lambda,\quad\tno{ for all }g\in C_{\udl k}(\mcG_T,\ovl\Q_p)
\end{align}
where the cap product is induced by the pairing \eqref{equation.finalpairing} and 
\[
\delta_p^*: \mcA^{p\cup\iy}(V_p,V(\udl k)_{\ovl\Q})\lra \mcA^{p\cup\iy}\ipa{C_c^0(T(F_p),\ovl\Q)\oti_{\ovl\Q} \mcP({\udl k})_{\ovl\Q},\ovl\Q_p}
\]
is the corresponding $T(F)$-equivariant pullback. Indeed, the pairing \eqref{equation.finalpairing} with $S=\{\mfp\mid p\}$, $R=\bar\Q$, $V=\mcP({\udl k})_{\ovl\Q}$ and $N=\bar\Q_p$ is given by 
\[
\langle \cdot|\cdot \rangle: C_{\rm fc}^0(T(\A_F),\bar\Q)\oti_{\bar\Q} \mcP({\udl k})_{\ovl\Q}\ti \mcA^{p\cup\iy}(C_c^0(T(F_p),\bar\Q)\oti_{\bar\Q} \mcP({\udl k})_{\ovl\Q},\bar\Q_p)(\lambda) \longrightarrow \bar\Q_p.
\]
Thus, 
the cap product in \eqref{definition.distribution} is well-defined.

\subsection{The morphism $\delta_p$}\label{deltap}

As seen in the previous sections, we want to construct a $T(F_p)$-equivariant morphism
\[
    \delta_p: C_c^0(T(F_p),\ovl\Q)\simeq \bigotimes_{\mfp\mid p}C_c^0(T(F_\mfp),\ovl\Q) \longrightarrow V_p\simeq\bigotimes_{\mfp\mid p}V_\mfp.
\]
Hence, it is enough to construct $T(F_\mfp)$-equivariant morphisms $\delta_\mfp: C_c^0(T(F_\mfp),\ovl\Q) \longrightarrow V_\mfp$.


Let us fix a place $\mfp\mid p$, and let $\pi_\mfp$ be the local representation. Recall that by hypothesis $G(F_\mfp)\simeq\PGL_2(F_\mfp)$, hence we fix such an isomorphism. From now on we will make the following assumptions:
\begin{hypothesis}\label{hypothesisPSt}
Let $P\subset \GL_2(F_\mfp)$ be the subgroup of upper triangular matrices, then we assume that $\imath(K_\mfp^\ti)\not\subset P$. Moreover, we will assume that $\pi_\mfp$ is either principal series or Steinberg.
\end{hypothesis}
\begin{remark}
    Since $T(F_\mfp)=K_\mfp^\times/F_\mfp^\ti$, we will usually denote by $\tilde t\in K_\mfp^\times$ any preimage of $t\in T(F_\mfp)$. 
    Let us consider
\[
X_\mfp:=\left\{\begin{array}{ll}\uhp_\mfp=K_\mfp\setminus F_\mfp, &\mbox{if $T$ does not split at $\mfp$},\\
\PP^1(F_\mfp),&\mbox{ if splits at $\mfp$}.\end{array}\right.
\]
In both cases $X_\mfp$ comes equipped with a natural action of $G(F_\mfp)$ given by fractional linear transformations. We write $\tau_\mfp$ and $\bar\tau_\mfp$ for the two fixed points by $\iota(T(F_\mfp))$ in $X_\mfp$. Since $\iota(K_\mfp^\ti)\cap P=F_\mfp^\ti$, in the split case $\tau_\mfp,\bar\tau_\mfp\neq\iy$. In fact, $\tau_\mfp$ and $\bar\tau_\mfp$ define a pair of simultaneous eigenvectors $v_1^\mfp$ and $v_2^\mfp$. Indeed, we can write $v_1^\mfp=(1,-\bar\tau_\mfp)$ and $v_2^\mfp=(1,-\tau_\mfp)$ since we have the following identity
\begin{equation}\label{equation.eigenvalue}
(1,-\bar\tau_\mfp)\iota(\tilde t)=(1,-\bar\tau_\mfp)\bbm a & b \\ c & d \ebm =(a-c\bar\tau_\mfp) (1,-\iota(\tilde t\inv)\bar \tau_\mfp )=\lambda_{\tilde t}(1,-\bar\tau_\mfp),\qquad (1,-\tau_\mfp)\iota(\tilde t)=\bar \lambda_{\tilde t}(1,-\tau_\mfp),
\end{equation}
where $\lambda_{\tilde t}$ and $\bar\lambda_{\tilde t}$ are the corresponding eigenvectors. By abuse of notation, we will sometimes denote by $t$ the quotient $t=\lambda_{\tilde t}/\bar\lambda_{\tilde t}$, identifying $T(F_\mfp)$ with $F_\mfp^\times$, in the split case, or with $K_\mfp^1=\{x\in K_\mfp^\times;\;{\rm N}_{K_\mfp/F_\mfp}(x)=1\}$, otherwise. Under this identification, we will sometimes think of $C(T(F_\mfp),\cdot)$ as a set of functions on $F_\mfp^\times$ or $K_\mfp^1$, respectively. 
\end{remark}

By the previous hypothesis \ref{hypothesisPSt}, $V_\mfp$ is a quotient of 
\[
\Ind_P^G(\hat\chi_\mfp^0)^0=\icla{f\in\GL_2(F_\mfp)\rightarrow\ovl\Q,\mbox{ locally constant }\;f\left(\bbm x_1 & y \\ & x_2 \ebm g \right)=\hat\chi_\mfp^0\left(\frac{x_1}{x_2}\right)\cdot f(g)},
\]
for a locally constant character $\hat\chi_\mfp^0$.
Moreover $\imath(K_\mfp^\ti)\cap P=F_\mfp^\ti$, hence we construct  
\begin{equation}\label{eq:deltaSexplicit}
\begin{tikzcd}[ampersand replacement=\&]
\delta_\mfp: C_c^0(T(F_\mfp),\ovl\Q)\ar[r] \&[-3em] \Ind_P^G(\hat\chi_\mfp^0)^0, \\[-2em] 
f\ar[r,mapsto] \& \delta_\mfp(f)\left(g\right):=
\begin{cases}
\hat\chi_\mfp^0\left(\frac{x_1}{x_2}\right)\cdot f(t^{-1}), & g=\bsm x_1 \& y \\ \& x_2 \esm\imath(\tilde t)\in P\iota(K_\mfp^\ti), \\
\quad 0, &  g\not\in P\iota(K_\mfp^\ti).\\
\end{cases}
\end{tikzcd}
\end{equation}
It is clearly $T(F_\mfp)$-equivariant. Moreover, it induces the desired $T(F_\mfp)$-equivariant morphism
\[
\delta_\mfp:C_c^0(T(F_\mfp),\ovl\Q)\longrightarrow V_\mfp.
\]
The following result is a generalization of \cite[Lemma 5.2]{HerMol1} for more general induced representations:
\begin{lemma}\label{lemma:mordelta}
The morphism $\delta_\mfp$ is given by  
\[
\delta_\mfp:C^0_c(T(F_\mfp),\ovl\Q)\longrightarrow \Ind_P^G(\hat\chi_\mfp^0)^0;\qquad\delta_\mfp(f)\bsm a&b\\c&d\esm=\hat\chi_\mfp^0\ipa{\frac{ad-bc}{(c\tau_\mfp+d)(c\ovl\tau_\mfp+d)}}\cdot f\ipa{\frac{c\bar\tau_\mfp+d}{c\tau_\mfp+d}}.
\]
In particular, if $\mfp$ does not split in $K$ then $\delta_\mfp$ is bijective.
\end{lemma}
\begin{proof}
From the relations
$(1,-\bar\tau_\mfp)\iota(\tilde t)=\lambda_{\tilde t}(1,-\bar\tau_\mfp)$ and $(1,-\tau_\mfp)\iota(\tilde t)=\bar \lambda_{\tilde t}(1,-\tau_\mfp)$
we deduce
\begin{equation}\label{eqiotat}
\imath(\tilde t)=\frac{1}{\tau_\mfp-\ovl \tau_\mfp}\bbm \lambda_{\tilde t}\tau_\mfp-\ovl\lambda_{\tilde t}\ovl \tau_\mfp & \tau_\mfp\ovl\tau_\mfp(\ovl\lambda_{\tilde t}-\lambda_{\tilde t}) \\ \lambda_{\tilde t}-\ovl\lambda_{\tilde t}  & \ovl\lambda_{\tilde t}\tau_\mfp-\lambda_{\tilde t}\ovl\tau_\mfp \ebm=\frac{\ovl\lambda_{\tilde t}}{\tau_\mfp-\ovl \tau_\mfp}\bbm t\tau_\mfp-\ovl \tau_\mfp & \tau_\mfp\ovl\tau_\mfp(1-t) \\ t-1  & \tau_\mfp-t\ovl\tau_\mfp \ebm,
\end{equation}
where $t=\lambda_{\tilde t}/\bar\lambda_{\tilde t}$. Hence, if we have
\[
\bbm a&b\\c&d\ebm=\bbm x_1&y\\&x_2\ebm\imath(\tilde t)=\frac{\ovl\lambda_{\tilde t}}{\tau_\mfp-\ovl \tau_\mfp}\bbm x_1&y\\&x_2\ebm\bbm t\tau_\mfp-\ovl \tau_\mfp & \tau_\mfp\ovl\tau_\mfp(1-t) \\ t-1  & \tau_\mfp-t\ovl\tau_\mfp \ebm
\]
we obtain the identities
\[
 -\frac{d}{c}  =  \frac{t\inv \tau_\mfp-\ovl\tau_\mfp}{t\inv -1},\qquad
 ad-bc=x_1x_2\lambda_{\tilde t}\ovl\lambda_{\tilde t},\qquad 
 t\inv =  \frac{\ovl\tau_\mfp+\frac{d}{c}}{\tau_\mfp+\frac{d}{c}}=\frac{c\ovl\tau_\mfp+d}{c\tau_\mfp+d} ,\qquad c=\frac{x_2(t-1)\ovl\lambda_{\tilde t}}{\tau_\mfp-\ovl\tau_\mfp}.
\]
Using such identities, the result follows from a straightforward computation.
\end{proof}

\subsection{Interpolation properties}\label{section:IntProp}

As we have previously emphasized, we have to think of $\mu_{\phi_\lambda^p}$ as a generalization of Bertolini-Darmon anticyclotomic $p$-adic L-function. Hence, it should have a link to the classical L-function, namely, an interpolation property.

In order to have an explicit interpolation formula, write $N\subseteq\mfo_F$ for the level of $\pi$. Write $\Pi$ for the Jacquet-Langlands lift of $\pi$ to $\PGL_2$, and let $U_0(N)\subseteq \GL_2(\A_F^\infty)$ be the usual open compact subgroup of upper triangular matrices modulo $N$. 
Notice that the usual space of automorphic forms for $\PGL_2$ of weight $\underline k+2\in (2\N)^d$ and level $U_0(N)$ can be described as
\begin{equation}\label{defS2}
M_{\udl k+2}(U_0(N)):=\Hom_{(\mfg_\infty,\mcK_\infty)}\ipa{D(\underline{k}+2),\mcA(\PGL_2)^{U_0(N)}},    
\end{equation}
where $\mcA(\PGL_2)^{U_0(N)}$ is the space of $U_0(N)$-invariant automorphic forms for $\PGL_2/F$ and $D(\underline k+2)$ is the $(\mfg_\infty,\mcK_\infty)$-module $D(\underline k+2)=\bigotimes_{\sigma\mid\infty}D_\sigma(k_\sigma+2)$, where $D_\sigma(k_\sigma+2)$ is the $(\mfg_\sigma,\mcK_\sigma)$-module of discrete series of weight $k_\sigma+2$.
In \cite[\S 3.2]{preprintsanti2}, a normalized automorphic form $\Psi\in M_{\udl k+2}(U_0(N))$ generating $\Pi$ is introduced, so that certain period integral of $\Psi$ agrees with the L-function associated with $\Pi$. In fact, $\Psi$ corresponds to the normalized Hilbert newform.

To provide similar definitions for the group $G$, let $D$ be the discriminant of $K$ and let us consider $c\subset\mfo_F$ an ideal prime to $p$. In order to simplify the resulting formulas, we will assume the following hypothesis:
\begin{hypothesis}\label{simplhyp}
    For every finite place $v$, either $\ord_v(c)\geq \ord_v(N)$ or $\ord_v(c)=0$, with $\ord_v(N)\leq 1$ in case $\ord_v(c)=0$ and $K_v$ non-split. In addition, we will assume that $V_p$ is a product of local representations that are either spherical or Steinberg, equivalently $\mfp^2\nmid N$ for all $\mfp\mid p$, hence $V_\mfp$ are quotients of $\Ind_P^G(\hat\chi^0_\mfp)$ with $\hat\chi^0_\mfp$ unramified. 
\end{hypothesis} 
In a situation where these conditions are not fulfilled one can still use the results of \cite{preprintsanti3} to compute the interpolation formulas, but they become much more complicated. For any finite place $v$, we say that an Eichler order $\mcO_{N,v}\subseteq B_v$ of discriminant $N\mfo_{F_v}$ is \emph{$c$-admissible} if the intersection $\mfo_{c,v}=K_v\cap\mcO_{N,v}$ is an order of conductor $c\mfo_{F_v}$ and, in case $v\mid (N,c)$, then $\mcO_{N,v}$ is the intersection of two maximal orders $\mcO_{B,v}'$ and $\mcO_{B,v}''$ of $B_v$ such that 
    $\mcO_{B,v}'\cap K_v=\mfo_{c,v}$ and $\mcO_{B,v}''\cap K_v=\mfo_{c/N,v}$.
By \cite[lemma 3.6]{preprintsanti3}, for $c$ and $N$ satisfying the above hypothesis \ref{simplhyp} admissible Eichler orders always exist. Thus, we will choose a $c$-admissible $\mcO_{N,v}$ for every finite place $v$.

 Let $U_{Np}=U_N^p\times U_0(p)$, where $U_N^p=\prod_{v\nmid p\cup\infty}\mcO_{N,v}^\ti$ and $U_0(p)=\prod_{\mfp\mid p}U_0(\mfp)$.
In analogy with \eqref{defS2} 
\begin{equation}\label{defS22}
M_{\udl k+2}(U_{Np}):=\Hom_{(\mfg_\infty,\mcK_\infty)}\ipa{D(\underline{k}+2),\mcA(G)^{U_{Np}}},    
\end{equation}
where $\mcA(G)^{U_{Np}}$ is the space of $U_{Np}$-invariant automorphic forms for $G/F$ and now the $(\mcG_\infty,\mcK_\infty)$-module $D(\underline k+2)=\bigotimes_{\sigma\in\Sigma_B}D_\sigma(k_\sigma+2)\oti\bigotimes_{\sigma\not\in\Sigma_B}V(k_\sigma)$. In \cite[\S 3.2]{preprintsanti2}, a natural bilinear inner product $\langle\;,\;\rangle:M_{{\udl k}+2}(U_{Np})\times M_{{\udl k}+2}(U_{Np})\rightarrow\C$ is defined.
As explained in \cite[remark 3.2]{preprintsanti2}, if $G=\PGL_2$ then we have the equality $\langle\phi,\bar\phi\rangle=2^{{\underline k}+1}\pi^d{\rm vol}(U_0(N))(\phi,\phi)_{U_0(N)}$, where $(\;,\;)_{U_0(N)}$ is the usual Petersson inner product.
In \cite[\S 3.1]{preprintsanti2} an explicit Harder-Eichler-Shimura morphism ${\rm ES}_\lambda$ is described for a fixed characted $\lambda$ (see also \cite{ESsanti}):
\[
{\rm ES}_\lambda:M_{\udl k+2}(U_{Np})\longrightarrow H^u(G(F)_+,\mcA^{\iy}(V(\udl k))^{U_{Np}})^{\lambda}.
\] 
Since for any $v\in V_p$, we have a $G(F)$-equivariant morphism $\cdot(v):\mcA^{p\cup\iy}(V_p,V(\udl k))(\lambda) \ra \mcA^\iy(V(\udl k))(\lambda)$ 
\[
\phi(v)(g_p,g^p):=\phi(g^{p})(g_p  v);\qquad g_p\in G(F_p),\;g^p\in G(\A_F^{p\cup\iy}),
\]
assuming that $\phi^p_\lambda\in H^u(G(F)_+,\mcA^{p\cup\iy}(V_p,V(\udl k))^{U_N^p})^{\lambda}$ we can consider $\Phi_{\udl\alpha}\in M_{\udl k}(U_{Np})$ such that
\begin{equation}\label{deff0intprop}
    {\rm ES}_\lambda\ipa{\Phi_{\udl \alpha}}=\phi^p_\lambda(f_p),\qquad f_p=\bigotimes_{\mfp\mid p}f_{0,\mfp}\in V_p^{U_0(p)},
\end{equation}
where the vectors $f_{0,\mfp}\in V_\mfp^{U_0(\mfp)}$ are defined by $f_{0,\mfp}(bk)=\hat\chi_\mfp^0(b)\indi_{U_0(1)\setminus U_0(\mfp)}(k)$ for all $b\in P$ and $k\in U_0(1)=\GL_2(\mfo_{F_\mfp})$. Since $U_\mfp f_{0,\mfp}=\alpha_\mfp\varpi_\mfp^{-\udl k/2} f_{0,\mfp}$, where $U_\mfp$ is the usual Hecke operator, the form $\Phi_{\udl \alpha}$ can be seen as the $\mfp$-stabilization of a newform for all $\mfp\mid p$.

Recall that by equation \eqref{PtoCoverC} we have a natural $T(F)$-equivariant morphism
\begin{equation}\label{descCkcC}
    C^0(T(\A_F), \ovl\Q) \oti_{\bar\Q} \mcP(\udl k)_{\bar\Q} \longrightarrow C(T(\A_F), \C),
\end{equation}
analogously as in \eqref{descCkc}. Its image lie in the set of continuous functions that are locally polynomical at $T(F_\infty)$.
\begin{definition}
Let $\xi\in C_{\underline{k}}(\mcG_T,\C_p)$ be a locally polynomial character. Thus, in a neighbourhood $U$ of $1$ in $T(F_\mfp)$
\[
\Theta^\ast\xi\mid_{U}(t_p)=\prod_{\mfp\mid p}\prod_{\sigma\in\Sigma_\mfp}\sigma_K\ipa{\frac{t_\mfp}{\bar t_\mfp}}^{m_{\sigma}}, \qquad \underline{m}=(m_{\sigma});\quad -\frac{\underline{k}}{2}\leq\underline{m}\leq\frac{\underline{k}}{2}.
\]
We define the archimedean avatar of $\xi$:
\[
\tilde\xi:T(\A_F)/T(F)\longrightarrow\C^\ti;\qquad \tilde\xi(t)=\Theta^\ast\xi(t)\cdot\prod_{\mfp\mid p}\prod_{\sigma\in\Sigma_\mfp}\sigma_K\ipa{\frac{t_\mfp}{\bar t_\mfp}}^{-m_{\sigma}}\cdot\prod_{\sigma\in\iy}\sigma_K\ipa{\frac{t_\sigma}{\bar t_\sigma}}^{m_{\sigma}},
\]
once identified the set of embeddings $\sigma:F\hookrightarrow\R$  with $\bigcup_{\mfp}\Sigma_\mfp$.
\end{definition}
\begin{remark}
   Notice that $\Theta^\ast\xi$ and $\tilde \xi$ correspond to the same element of $\xi_0\oti P_{\underline m}\in C^0(T(\A_F), \ovl\Q) \oti_{\bar\Q} \mcP(\udl k)_{\bar\Q}$ under the morphisms of \eqref{descCkc} and \eqref{descCkcC}.
\end{remark}
We say that a a locally polynomial character $\xi\in C_{\underline{k}}(\mcG_T,\C_p)$ has conductor $c$ outside $p$ if the conductor $c_\xi$ of 
$\tilde\xi$ satisfies $c\mfo_{F_v}=c_\xi\mfo_{F_v}$ for any finite place $v\nmid p$. For every $\mfp\mid p$ we write $\xi_\mfp=\tilde\xi\mid_{T(F_\mfp)}$. 

Let $\epsilon(B_v)=1$ if $B_v\simeq M_2(F_v)$ and $\epsilon(B_v)=-1$ otherwise.
The embedding $K\subset B$ provides a decomposition $B=K\oplus KJ_0$, with  $J_0$ in the normalizer of $K$ in $B$ satisfying $J_0^2=M\in F^\times$. By \cite[lemma 3.6]{preprintsanti3} there exists $k_0=(k_{0,v})_v\in T(\A_F^p)$ such that $k_0^{-1}J_0\in G(F_\infty)_+\times w^{Dcp}U_N^p$, where $w^{Dcp}$ is the Atkin-Lehner involution $w^{Dcp}=\prod_{v\mid N;v\nmid Dcp}w_v$. 
We denote by $\varepsilon^{Dcp}$ the eigenvalue of $w^{Dcp}$ on $\pi^{U_N^p}$.

\begin{theorem}\label{THMintprop}
Let $\xi\in C_{\underline{k}}(\mcG_T,\C_p)$ be a locally polynomial character of conductor $c$ outside $p$.  If $\xi_0\mid_{T(F_\iy)}\neq\lambda$ or the local root number $\epsilon(1/2,\pi_v,\xi_v)\neq\xi_v\psi_{K_v}(-1)\epsilon(B_v)$ for any finite place $v\nmid p$, then $\mu_{\phi^p_\lambda}(\xi)=\int_{\mcG_T}\xi d\mu_{\phi^p_\lambda}=0$. Otherwise, 
\[
\ipa{\int_{\mcG_T}\xi d\mu_{\phi^p_\lambda}}^2=
\frac{2^{\#\Sigma_D^p}L_{c}(1,\psi_{K})^2h^2
C(\underline k,\underline m)}{\varepsilon^{Dcp}\xi_0^{-1}(\tilde k_{0})|c_\xi^2 D|^{\frac{1}{2}}} L^{\Sigma}(1/2,\Pi,\tilde\xi)\frac{\langle \Phi_{\udl\alpha},\Phi_{\udl\alpha}\rangle}{\langle \Psi,\Psi\rangle}\frac{{\rm vol}(U_0(N))}{{\rm vol}(U_N){\rm vol}(\mfo_{K_p}^\ti/\mfo_{F_p}^\ti)^{2}}\prod_{\mfp\mid p}C(\pi_\mfp)C(\pi_\mfp,\xi_\mfp),
\]
where $L_c(1,\psi_{K})$ is the product of local L-functions at places $v\mid c$, $\Sigma_D^p=\{v\mid (N,D);\;v\nmid c p\}$, $\Sigma=\{v\mid (N,c)\}\cup\{\mfp\mid p;\;\mfp\nmid(N,D)\}$, $L^\Sigma(1/2,\Pi,\tilde\xi)$ is the L-function with the local factors at places $v\mid \Sigma\cup\infty$ excluded, $\tilde k_0=(\tilde k_{0,v})\in T(\A_F)$ is such that $\tilde k_{0,v}=k_{0,v}$ at places $v\nmid p$ and $\tilde k_{0,\mfp}$ satisfies $J_0\tilde k_{0,\mfp} \in P$ at places $\mfp\mid p$, $h=\#T(\mfo_F)_{+,{\rm tors}}$. Finally, 
\begin{eqnarray*}
C(\underline k,\underline m)&=&(-1)^{\left(\sum_{\sigma\in\Sigma_B}\frac{k_{\sigma}+2}{2}\right)}\left(\frac{1}{4\pi}\right)^{r_{K/F}}\prod_{\sigma\mid\infty}\Gamma\left(\frac{k_{\sigma}+2}{2}-m_{\sigma}\right)\Gamma\left(\frac{k_{\sigma}+2}{2}+m_{\sigma}\right)(2\pi)^{-k_{\sigma}-2};\\
    C(\pi_\mfp)&=&\left\{
    \begin{array}{ll}
       \zeta_\mfp(1)^2L(0,\chi_\mfp^2)\zeta_{\mfp}(2)^{-1},  &\mfp\nmid N;  \\
        \zeta_\mfp(1)\zeta_{\mfp}(2)^{-1},&\mfp\mid N,\\
    \end{array}\right.\\
     C(\pi_\mfp,\xi_\mfp)&=&\left\{
    \begin{array}{ll}
    \varepsilon(\pi_\mfp,\xi_\mfp)\hat\chi^0_{\mfp}\left(\frac{D}{(\tau_\mfp-\ovl\tau_\mfp)^2}\right),&\ord_\mfp(c_\xi)=0;\\
    \ipa{\hat\chi^0_\mfp(\varpi_\mfp)q_\mfp}^{2\ord_\mfp(c_\xi)}\hat\chi^0_{\mfp}\left(\frac{D}{(\tau_\mfp-\ovl\tau_\mfp)^2}\right),&\ord_\mfp(c_\xi)>0,
    \end{array}\right.,\qquad \varepsilon(\pi_\mfp,\xi_\mfp)=\frac{L(-1,\xi_\mfp\cdot(\hat\chi^0_\mfp)^{-1}\circ{\rm N}_{K_\mfp/F_\mfp})}{L(0,\xi_\mfp\cdot\hat\chi^0_\mfp\circ{\rm N}_{K_\mfp/F_\mfp})}.
\end{eqnarray*}
\end{theorem}
\begin{proof}
Let us consider the following $T(F)$-equivariant morphism: For all $z\in T(F_\iy)$ and $t\in T(\A_F^\iy)$
\[
\varphi:(C^0(T(\A_F),\C)\oti \mcP(\underline{k}))\otimes\mcA^\iy(V(\underline{k}))(\lambda)\longrightarrow C^0(T(\A_F),\C);\qquad\varphi((f\oti P)\oti\phi)(z,t):=\frac{f(z,t)}{\lambda(z)}\cdot\phi(t)\ipa{P}.
\]
Let us also consider the natural pairing $\langle\cdot,\cdot\rangle_T:C^0(T(\A_F),\C)\ti C_{\rm fc}^0(T(\A_F),\C)\ra\C$ given by the Haar measure
\[
\langle f_1,f_2\rangle_T:=\sum_{z\in CC}\int_{\mbT}\int_{T(\A_F^\iy)}f_1(zs,t)\cdot f_2(zs,t)d^\ti t d^\times s,\qquad CC=T(F_\iy)/T(F_\iy)_+,
\]
where $\mbT$ is the maximal compact quotient of $T(F_\iy)_+$.
For any $f\in C^0(T(\A_F),\C)$ and $f_2\in C_{\rm fc}^0(T(\A_F),\C)$, write $f\cdot f_2=f_p\otimes f^p$, where $f_p\in C_c(T(F_p),\C)$ and $f^p\in C_c(T(\A_F^p),\C)$, and let $H=\prod_\mfp H_\mfp\subseteq T(F_p)$ be a small enough open compact subgroup so that $f_p$ is $H$-invariant, namely, $f_p=\sum_{t_p\in T(F_p)/H}f_p(t_p)\indi_{t_pH}$. We compute using the concrete description of $\langle\cdot|\cdot\rangle$ provided in \eqref{equation.finalpairing}: 
\begin{eqnarray*}
\langle\varphi(f\oti P,\phi(\delta_p(\indi_H))),f_2\rangle_T
&=&{\rm vol}(\mbT)\sum_{z\in CC}\lambda(z)\inv\cdot\int_{T(\A_F^{p\cup\iy})}\int_{T(F_{p})}f^p(z,t^p)\cdot f_p(t_p)\cdot\phi(t^p)(\delta_p(\indi_{t_pH}))\ipa{P}d^\ti t^p d^\ti t_p\\
&=&{\rm vol}(\mbT)\cdot{\rm vol}(H)\sum_{z\in CC}\lambda(z)\inv\cdot\int_{T(\A_F^{p\cup\iy})}f^p(z,t^p)\cdot\delta_p^\ast\phi(t^p)(f_p)(P)d^\ti t^p\\
&=&[T(F):T(F)_+]\cdot{\rm vol}(\mbT)\cdot{\rm vol}(H)\cdot\langle f\cdot f_2\otimes P|\delta_p^\ast\phi\rangle,
\end{eqnarray*}
for all $\phi\in \mcA^{p\cup\iy}(V_p,V(\underline{k}))(\lambda)$. 
Since $[T(F):T(F)_+]=2^{u}$ and ${\rm vol}(\mbT)=2^{r_{K/F}}$, we obtain by definition
\[\int_{\mcG_T}\xi d\mu_{\phi^p_\lambda}=(\Theta^\ast\xi\cap\eta)\cap\delta_p^\ast\phi^p_\lambda=\frac{1}{2^{[F:\Q]}\cdot{\rm vol}(H)}(\tilde\xi\cup\phi_H)\cap\eta,\qquad \phi_H=\phi_\lambda^p(\delta_p(\indi_H)),
\]
where the cap and cup products in $(\tilde\xi\cup\phi_H)\cap\eta$ correspond to the pairings $\langle\cdot,\cdot\rangle_T$ and $\varphi$, and $\tilde \xi$ is seen as an element of $C^0(T(\A_F),\C)\oti \mcP(\underline{k})$ since lies on the image of \eqref{descCkcC}.
In \cite[Theorem 1.3]{preprintsanti3} and expression for $\ipa{(\tilde\xi\cup \phi_H)\cap\eta}^2$ is obtained in terms of the classical L-function. More precisely, $\left((\rho^\ast\xi\cup \phi_H)\cap\eta\right)=0$ unless $\epsilon(1/2,\pi_v,\xi_v)=\xi_v\psi_{K_v}(-1)\epsilon(B_v)$ and $\xi_0\mid_{T(F_\iy)}=\lambda$ and, if this is the case, then 
\begin{equation*}\label{equation.theorem}
\ipa{\frac{(\tilde\xi\cup \phi_H)\cap\eta}{2^{[F:\Q]}\cdot{\rm vol}(H)}}^2 = \frac{2^{\#\Sigma_D}L_{c_\xi}(1,\psi_{K})^2h^2
C(\udl k,\udl m)}{\varepsilon^{Dc_\xi}\xi_0^{-1}(\bar k_{0})M^{-\underline m}|c_\xi^2 D|^{\frac{1}{2}}}\cdot L^{\Sigma_{c_\xi}}(1/2,\Pi,\tilde\xi)\cdot\frac{\langle \Phi_H,\Phi_H\rangle}{\langle \Psi,\Psi\rangle}\cdot\frac{{\rm vol}(U_0(N))}{{\rm vol}(U_N)}\prod_{\mfp\mid p}\frac{\lambda_\mfp(\delta_\mfp(\indi_{H_\mfp}),J_0)}{{\rm vol}(H_\mfp)^2\lambda_\mfp(\phi_{0,\mfp},J_0)},
\end{equation*}
where  $\Sigma_{c_\xi}:=\{v\mid (N,c_\xi)\}$, $\Sigma_D=\{v\mid (N,D);\;\ord_v(c_\xi)=0\}$, $\Phi_H\in M_{\underline k}(U)$ for some $U\subseteq G(\A_F^\infty)$ is such that ${\rm ES}_\lambda(\Phi_H)=\phi_H$, the open $U_N=\prod_{v\nmid\infty}\mcO_{N,v}^\times$ with $\mcO_{N,\mfp}$ a $c_\xi$-admissible Eichler order, $\bar k_0=(k_0,\bar k_{0,p})\in T(\A_F)$ satisfies $\bar k_0^{-1}J_0\in G(F_\infty)_+\times w^{Dc_\xi}U_N$ analogously as above, $\phi_{0,\mfp}=\pi^{\mcO_{N,\mfp}^\times}$ and
    \begin{equation}\label{localfactlambda}
\lambda_\mfp(\phi_\mfp,J_0)=\int_{T(F_\mfp)}\xi_\mfp(t)\frac{\langle\pi_\mfp(t)\phi_{\mfp}^\lambda,\pi_\mfp(J_0)\phi_{\mfp}^\lambda\rangle_\mfp}{\langle\phi_{\mfp}^\lambda,\phi_{\mfp}^\lambda\rangle_\mfp} d^\times t,
    \end{equation}
being $\langle\;,\;\rangle_\mfp$ any $G(F_\mfp)$-invariant bilinear inner product, and $d^\times t$ any Haar measure. 
In \cite[proof of theorem 5.3]{HerMol1} (see also \cite[corollary 8.3]{HerMol1} and \cite[proposition 3.12]{CST}) it is computed that
\begin{equation}\label{PhiPhiint}
    \frac{\langle \Phi_H,\Phi_H\rangle}{\langle \Phi_{\udl\alpha},\Phi_{\udl\alpha}\rangle}\prod_{\mfp\mid p}\frac{\lambda_\mfp(\delta_\mfp(\indi_{H_\mfp}),J_0)}{{\rm vol}(H_\mfp)^2\lambda_\mfp(\phi_{0,\mfp},J_0)}=\prod_{\mfp\mid p}\frac{\xi_\mfp(\tilde k_{0,\mfp}\bar k_{0,\mfp}^{-1})}{{\rm vol}(\mfo_{K_\mfp}^\ti/\mfo_{F_\mfp}^\ti)^{2}}\frac{C(\pi_\mfp,\xi_\mfp)C(\pi_\mfp)}{L(1/2,\Pi_\mfp,\xi_\mfp)}
\left\{
    \begin{array}{ll}
       1,  &\mfp\nmid N,c_\xi;  \\
        L(1,\psi_{K_\mfp})^{-2}, &\mfp\mid c_\xi;  \\
        -\varepsilon_\mfp,&\mfp\mid N,\;\mfp\nmid D c_\xi;\\
        \frac{L(1/2,\Pi_\mfp,\xi_\mfp)}{2},&\mfp\mid (N,D),\;\mfp\nmid c_\xi;
    \end{array}\right.
\end{equation}
where $-\varepsilon_\mfp$ is the eigenvalue of the Atkin-Lehner operator $w_\mfp$.
The result then follows.
\end{proof}



\begin{remark}
Notice that the constants $C(\underline k,\underline m)$ and $C(\pi_\mfp)$ never vanish. Moreover, $C(\pi_\mfp,\xi_\mfp)$ only vanishes when 
\[
0=\varepsilon(\pi_\mfp,\xi_\mfp)=\left\{
\begin{array}{ll}
     \ipa{1-\xi_\mfp^{-1}\hat\chi_\mfp^0(\varpi_\mfp)}\ipa{1-\xi_\mfp\hat\chi_\mfp^0(\varpi_\mfp)}\ipa{1-\frac{\xi_\mfp^{-1}(\hat\chi_\mfp^0)^{-1}(\varpi_\mfp)}{q_\mfp}}^{-1}\ipa{1-\frac{\xi_\mfp(\hat\chi_\mfp^0)^{-1}(\varpi_\mfp)}{q_\mfp}}^{-1},& T(F_\mfp)\simeq F_\mfp^\times;  \\
     \ipa{1-\hat\chi_\mfp^0(\varpi_\mfp)^2}\ipa{1-\hat\chi_\mfp^0(\varpi_\mfp)^{-2}q_\mfp^{-2}}^{-1},& \mbox{$\mfp$ inert in $K$};\\
     \ipa{1-\xi_\mfp(\varpi_\mfp^{1/2})\hat\chi_\mfp^0(\varpi_\mfp)}\ipa{1-\xi_\mfp(\varpi_\mfp^{1/2})\hat\chi_\mfp^0(\varpi_\mfp)^{-1}q_\mfp^{-1}}^{-1},&\mfp\mid D.
\end{array}
\right.
\]
    Thus, the vanishing of $\mu_{\phi^p_\lambda}(\xi)$ depends on the vanishing of the L-function $L\left(\frac{1}{2},\Pi,\tilde\xi\right)$ or the epsilon factors $\varepsilon(\pi_\mfp,\xi_\mfp)$. 
    If $\udl k=\udl 0$, $V_\mfp={\rm St}(F_\mfp)(\varepsilon_\mfp)$ and $\xi_\mfp=\varepsilon_\mfp\mid_{T(F_\mfp)}$ (see remark \ref{remintMS}), then the epsilon factor $\varepsilon(\pi_\mfp,\xi_\mfp)$ vanishes. If in addition $T(F_\mfp)=F_\mfp^\ti$ (split case), then the L-function $L\left(\frac{1}{2},\Pi,\tilde\xi\right)$ may not vanish giving rise to an exceptional zero (see \cite{blanco2015anticyclotomic} or \cite{HerMol1}). If otherwise $T(F_\mfp)\neq F_\mfp^\ti$ (non-split case), then the L-function $L\left(\frac{1}{2},\Pi,\tilde\xi\right)$ also vanishes. Is in this last situation where we can construct plectic points.  
\end{remark}


\subsection{Continuous measures on the torus}

In previous sections we have constructed a distribution
\[
\mu_{\phi_\lambda^p}\in {\rm Dist}_{\udl k}(\mcG_T,\ovl\Q_p):=\Hom(C_{\udl k}(\mcG_T,\ovl\Q_p),\ovl\Q_p).
\]
In this section we aim to extend it to a continuous measure
\[
\mu_{\phi_\lambda^p}\in {\rm Meas}(\mcG_T,\C_p):=\Hom_{\rm cnt}(C(\mcG_T,\C_p),\C_p).
\]
Write also ${\rm Meas}(T(F_p),\C_p):=\Hom_{\rm cnt}(C_{c}(T(F_p),\C_p),\C_p)$.

Assume that we are in the setting of Proposition \ref{OCmodsymb}, namely, for all $\mfp\mid p$ we have $\alpha_\mfp\in\mfo_{\C_p}^\times$.
Then attached to our modular symbol $\phi_\lambda^p$, we have a unique overconvergent cohomology class
\[
 \hat\phi_\lambda^p\in H_\ast^u\ipa{G(F)_+,\mcA^{p\cup\infty}(\D_{\chi_p^{-2}}(\C_p)_{\udl\alpha})}^\lambda.
\]
On the other side, the formula used to define $\delta_\mfp$ in \eqref{eq:deltaSexplicit} extends to
\[
    \delta_p: C_{c}(T(F_p),\C_p) \longrightarrow {\rm Ind}_P^G(\hat\chi_p)\stackrel{\varphi_p}{\simeq}C_{\chi_p^{-2}}(\mcL_p,\C_p).
\]
Notice that, for any $f\in C(\mcG_T,\C_p)$, we have that $\Theta^\ast f\in H^0(T(F),C(T(\A_F),\C_p))$. Moreover, one can deduce $C_c(T(\A_F),\C_p)=\Ind_{T(F)_+}^{T(F)}C_c^0(T(\A_F^{\iy\cup p}),C_c(T(F_p),\C_p))$. Hence, with respect to the pairing \eqref{equation.pairingnoplus}
\[
\ipa{(\Theta^\ast f)\cap\eta}\cap \delta_p^\ast(\hat\phi_\lambda^p)\in\C_p,
\]
since
$\delta_p^\ast(\hat\phi_\lambda^p)\in H^u\ipa{G(F)_+,\mcA^{p\cup\infty}({\rm Meas}(T(F_p),\C_p))}^\lambda$.
The following result follows directly from the definitions:
\begin{theorem}\label{relOCpadicLfunct}
Assume that for all $\mfp\mid p$ we have $\alpha_\mfp\in\mfo_{\C_p}^\times$. 
Then the locally polynomial distribution $\mu_{\phi_\lambda^p}$ extends to a continuous measure defined by
\[
\int_{\mcG_T}fd\mu_{\phi_\lambda^p}=\ipa{(\Theta^\ast f)\cap\eta}\cap\delta_p^\ast(\hat\phi_\lambda^p),\qquad f\in C(\mcG_T,\C_p).
\]
\end{theorem}

\section{Hida families}\label{HidaFamilies}
As in most of the paper, we assume that our modular symbol $\phi_\lambda^p\in H^u(G(F)_+,\mcA^{p\cup\infty}(V_p,V(\underline{k})))^\lambda$ is ordinary, namely, $\alpha_\mfp=\hat\chi_\mfp(\varpi_\mfp)\inv=\varpi_\mfp^{\udl k/2}\hat\chi_\mfp^0(\varpi_\mfp)\inv\in\mfo_{\C_p}^\ti$ for all $\mfp\mid p$. We are convinced that our work can be generalized to the finite slope situation by working with locally analytic distributions. 

Let $\Lambda_F$ be the Iwasawa algebra associated with $\mfo_{F_p}^\ti$, and let 
${\bf k}_p:\mfo_{F_p}^\ti\rightarrow \Lambda_F^\ti$,
be the universal character. Recall that ${\bf k}_p$ is characterized by the following property: For any adic Noetherean $\Z_p$-algebra $A$, and any continuous character
$\chi_p:\mfo_{F_p}^\ti\rightarrow A^\ti$,
there exists a continuous morphism $\varphi_{\chi_p}:\Lambda_F\rightarrow A$ such that $\chi_p=\varphi_{\chi_p}\circ{\bf k}_p$.


Let $\udl{\bf a}=({\bf a}_\mfp)_\mfp$, where ${\bf a}_\mfp\in\Lambda_F^\ti$. In the next section, we will introduce the specialization $G(F_p)$-equivariant morphisms
\begin{equation}\label{eqspecialization}
    s_{\varphi_{\chi_p}}:\D_{{\bf k}_p^{-2}}(\Lambda_F)_{\udl{\bf a}}\longrightarrow \D_{\chi_p^{-2}}(A)_{\udl{\alpha}},\qquad {\udl \alpha}=\varphi_{\chi_p}(\udl{\bf a}),
\end{equation}
for any pair $(A,\chi_p)$ as above.

\subsection{Specialization morphisms}\label{spe-mor}

Let $\varphi:R\rightarrow A$ be a continuous surjective morphism between two adic complete $\Z_p$-algebras and assume that $R$ is also a complete local ring. Notice that, for any set $S$ of places above $p$ and any continuous character $\chi_S:\mfo_{F_S}^\ti=\prod_{\mfp\in S}\mfo_{F_\mfp}^\ti\rightarrow R^\ti$, we have a continuous morphism 
\[
\varphi_\ast: C_{\chi_S}(\mcL_S,R)\longrightarrow C_{\varphi\circ\chi_S}(\mcL_S,A);\qquad \varphi^\ast(f)=\varphi\circ f.
\]
Notice that $C_{\chi_S}(\mcL_S,R)$ is topologically generated by functions of the form 
\[
F=\prod_{\mfp\in S}(f_{U_\mfp}\;\mbox{ or }\;g_{U_\mfp}),\qquad f_{U_\mfp}(c,d)=\chi_\mfp(c)\cdot\indi_{U_\mfp}\ipa{\frac{d}{c}},\qquad g_{U_\mfp}(\varpi_\mfp c,d)=\chi_\mfp(d)\cdot\indi_{U_\mfp}\ipa{\frac{c}{d}},
\]
where $U_\mfp\subset \mfo_{F_\mfp}$ are open compact subgroups. Write $C^0_{\chi_S}(\mcL_S,R)\subseteq C_{\chi_S}(\mcL_S,R)$ (and analogously $C^0_{\varphi\circ\chi_S}(\mcL_S,A)$) for the $R$-module generated by the functions $F$ above. Notice that $C^0_{\chi_S}(\mcL_S,R)\simeq \bigotimes_{\mfp\in S}C^0(\PP^1(F_\mfp),R)$ is a free $R$-module, hence, the induced morphism $\varphi_\ast:C^0_{\chi_S}(\mcL_S,R)\rightarrow C^0_{\varphi\circ\chi_S}(\mcL_S,A)$ is surjective. Moreover, if $f\in C_{\varphi\circ\chi_S}(\mcL_S,A)$ is $f=\lim_i F_i$ with $F_i\in C^0_{\varphi\circ\chi_S}(\mcL_S,A)$ and $\tilde F_i\in C^0_{\chi_S}(\mcL_S,R)$ are preimages of $F_i$, then the sequence $\{\tilde F_i\}_i$ is Cauchy. Indeed, for $i,j>N$ big enough, 
\[
\tilde F_i-\tilde F_j\in\varphi^{-1}(J)^nC^0_{\chi_S}(\mcL_S,R)\subseteq \mfm^n C^0_{\chi_S}(\mcL_S,R),
\]
where $J$ and $\mfm$ are the ideals of definition of $A$ and $R$, respectively. This shows that any $f\in C_{\varphi\circ\chi_S}(\mcL_S,A)$ has preimage $F=\lim_i \tilde F_i\in C_{\chi_S}(\mcL_S,R)$ and $\varphi^\ast$ is then surjective. Notice that $\D_{\chi_S}(R)=\Hom_R\ipa{C^0_{\chi_S}(\mcL_S,R),R}$. Moreover, since $C^0_{\chi_S}(\mcL_S,R)$ is free, one has the exact sequence 
\[
0\longrightarrow \Hom_R\ipa{C^0_{\chi_S}(\mcL_S,R),\ker\varphi}\longrightarrow \D_{\chi_S}(R) \longrightarrow \Hom_R\ipa{C^0_{\chi_S}(\mcL_S,R),A}\longrightarrow 0.
\]
Also by the freeness of $C^0_{\chi_S}(\mathcal{L}_S,R)$ (and $C^0_{\varphi\circ\chi_S}(\mathcal{L}_S,A))$ we have
\[
\Hom_R\ipa{C^0_{\chi_S}(\mcL_S,R),A}=\Hom_A\ipa{C^0_{\varphi\circ\chi_S}(\mcL_S,A),A}=\D_{\varphi\circ\chi_S}(A).
\]
Hence one obtains, for any $\udl \alpha=(\alpha_\mfp)_{\mfp\in S}\in (R^\times)^S$, a surjective $G(F_S)$-equivariant morphism  
\[
s_\varphi:\D_{\chi_S}(R)_{\udl \alpha}\longrightarrow \D_{\varphi\circ\chi_S}(A)_{\varphi\ipa{\udl \alpha}},\qquad (s_\varphi\mu)(f):=\varphi\ipa{\mu(\varphi_\ast^{-1}(f))}.
\]
\begin{remark}\label{remonhomcont}
    In particular, the above argument shows that 
    \[
    \Hom_{{\rm cont},R}\ipa{C_{\chi_S}(\mcL_S,R),A}=\Hom_{{\rm cont},A}\ipa{C_{\varphi\circ\chi_S}(\mcL_S,A),A}=\D_{\varphi\circ\chi_S}(A).
    \]
\end{remark}
\begin{lemma}\label{measuresonI}
    Given a continuous morphism $\varphi:R\rightarrow A$ 
    \[
    \ipa{\ker\varphi}\oti_R\D_{\chi_S}(R)=\left\{\mu\in\D_{\chi_S}(R);\;\; \mu(f)\in \ker\varphi\;\mbox{ for all }f\in C_{\chi_S}(\mcL_S,R)\right\}=\ker(s_{\varphi}).
    \]
\end{lemma}
\begin{proof}
We have seen that $\ker(s_{\varphi})=\Hom_R\ipa{C^0_{\chi_S}(\mcL_S,R),\ker\varphi}$. Moreover, since $R$ is noetherian, $\ker\varphi$ is finitelly generated. The result follows easily from the freeness of $C^0_{\chi_S}(\mcL_S,R)$.
\end{proof}

Assume that $\mfp$ is such that $\varphi\circ\chi_\mfp=1$ and $\varphi(\alpha_\mfp)=1$ (hence $\varphi\circ\hat\chi_\mfp=1$). Then $A\subset C_{\varphi\circ\chi_\mfp}(\mcL_\mfp,A)=C(\PP^1(F_\mfp),A)$ is a $G(F_\mfp)$-invariant space. Thus, 
\[
\tilde \mcE_\mfp:=\varphi_\ast^{-1}(A)\subset C_{\chi_\mfp}(\mcL_\mfp,R)
\]
is a $G(F_\mfp)$-invariant space.
 Moreover, by equation \eqref{relPp} we have a natural $G(F_{S_1})$-equivariant morphism for any $S_1\subset S$ and any $f^{S_1}\in C_{\chi_{S\setminus S_1}}(\mcL_{S\setminus S_1},R)$
\begin{equation}\label{restmus}
    \bullet(f^{S_1}):\D_{\chi_{S}}(R)\longrightarrow \D_{\chi_{S_1}}(R);\qquad \int_{\mcL_{S_1}}f_{S_1} d\mu(f^{S_1}):=\int_{\mcL_{S}}\ipa{f^{S_1}\oti f_{S_1}}d\mu.
\end{equation}

\begin{definition}\label{defspeSt}
    Let $S_1\subseteq S$ be a subset of places $\mfp$ such that $\varphi\circ\hat\chi_\mfp=1$. We define inductively 
    \[
        \D^{\varphi}_{\chi_{S}}(R)^{S_1}=\left\{\mu\in \D_{\chi_{S}}(R);\;\mu(\tilde\mcE_\mfp)\subseteq (\ker\varphi)\oti_R \D^\varphi_{\chi_{S\setminus\mfp}}(R)^{S_1\setminus\mfp},\;\mbox{for all }\mfp\in S_1\right\}.
    \]
    It is clear that the corresponding subspace $\D^\varphi_{\chi_{S}}(R)^{S_1}_{\udl\alpha}\subseteq\D_{\chi_{S}}(R)_{\udl\alpha}$ is $G(F_S)$-invariant.
\end{definition}
Notice that the specialization morphism $s_\varphi$ provides a $G(F_S)$-equivariant morphism
\[
\D^\varphi_{\chi_{S}}(R)^{S_1}_{\udl\alpha}\longrightarrow\Hom_{{\rm cont},A}\ipa{{\rm Ind}_P^G(\varphi\circ\hat\chi_{S\setminus S_1})\oti_A\bigotimes_{\mfp\in S_1}\tno{St}_A(F_\mfp),A};\qquad \tno{St}_A(F_\mfp)=C(\PP^1(F_\mfp),A)/A.
\]
In fact, $\D^\varphi_{\chi_{S}}(R)^{\mfp}_{\udl\alpha}=s_\varphi^{-1}\ipa{\Hom_{{\rm cont},A}\ipa{{\rm Ind}_P^G(\varphi\circ\hat\chi_{S\setminus S_1})\oti_A\tno{St}_A(F_\mfp),A}}$, but this is not the case for larger $S_1$.

\subsection{Lifting the form to a Hida family}\label{locsystevsgroupcoho}

Notice that $\D_{\chi_p^{-2}}(\C_p)=\D_{\chi_p^{-2}}(\mfo_{\C_p})\oti_{\mfo_{\C_p}}\C_p$ because continuous distributions are those that are bounded. Hence,
by proposition \ref{OCmodsymb} and remark \ref{remonadmrepcoho}, our $\phi_\lambda^p$ 
lifts (up-to-constant) to an overconvergent modular symbol 
\[
\hat\phi_\lambda^p\in H_\ast^u\ipa{G(F)_+,\mcA^{p\cup\infty}(\D_{\chi_p^{-2}}(\mfo_{\C_p})_{\underline{\alpha}})}^\lambda.
\]
Let us consider $\varphi=\varphi_{\chi_p}:\Lambda_F\rightarrow\mfo_{\C_p}$, and assume that there exists $\underline{{\bf a}}=({\bf a}_\mfp)_\mfp$, with ${\bf a}_\mfp\in\Lambda_F^\ti$ such that $\varphi({\bf a}_\mfp)=\alpha_\mfp$. This defines an extension $\hat {\bf k}_p$ of the universal character ${\bf k}_p$. 
In this section we will discuss the existence of both $\udl{\bf a}=({\bf a}_\mfp)_\mfp$ specializing to $\udl\alpha$ and a class $\Phi_\lambda^p\in H_\ast^u(G(F)_+,\mcA^{p\cup\infty}(\D_{{\bf k}_p^{-2}}(\Lambda_F)_{\underline{\bf a}}))^\lambda$ lifting the overconvergent modular symbol $\hat\phi_\lambda^p$, where the specialization map is provided by \eqref{eqspecialization}. For this purpose, we will make use of local systems with values in $\Lambda_F$ and eigenvarieties constructed by means of them. 

In case $\chi_p=1$ is the trivial character and $\hat\phi_\lambda^p\in H_\ast^u(G(F)_+,\mcA^{p\cup\infty}(\D_{1}(\Z_p)_{\underline{\alpha}}))^\lambda$, let $S$ be any set of places above $p$ where $V_\mfp^{\Z_p}$ is Steinberg, namely, $V_\mfp^{\Z_p}\simeq \tno{St}_{\Z_p}(F_\mfp)(\varepsilon_\mfp)$ for some character $\varepsilon_\mfp(g)=(\pm1)^{\nu(\det(g))}$. Therefore, $\hat\phi_\lambda^p$ has a preimage in
\[
\hat\phi_\lambda^p\in H_\ast^u\ipa{G(F)_+,\mcA^{p\cup\iy}(\D^{\rm id}_{1}(\Z_p)_{\underline{\alpha}^\ast}^S(\varepsilon_S))}^\lambda, \qquad \varepsilon_S=\prod_{\mfp\in S}\varepsilon_\mfp,\qquad \underline{\alpha}^\ast=\ipa{\varepsilon_\mfp\bsm\varpi_\mfp&\\&1\esm\cdot{\alpha}_\mfp}_\mfp,
\]
where ${\rm id}:\Z_p\rightarrow \Z_p$ is the identity morphism. In this situation, we need to refine the cohomology space where $\Phi_\lambda^p$ lives to take its Steinberg reduction into account. 

In order to address technical issues associated with possible shrinkage of the weight space, we need to introduce some adic algebras. Notice that $1+p\mfo_{F_p}$ is a free $\Z_p$-module of rank $d=[F:\Q]$. Hence, for a given basis $\{e_1,\cdots,e_d\}\subset 1+p\mfo_{F_p}$, we have an isomorphism $\Z_p[H][[T_1,\cdots ,T_d]]\stackrel{\simeq}{\rightarrow}\Lambda_F$, where $H\subset\mfo_{F_p}^\ti$ is the torsion subgroup and $T_i+1$ is sent to $e_i$. For any $n\in\N$, we consider the adic $\Lambda_F$-algebra
$\Lambda_n=\Lambda_F\left\langle\frac{T_1}{p^n},\cdots,\frac{T_d}{p^n}\right\rangle$.
The $\C_p$-valued points of the formal spectrum of $\Lambda_n$ classify the following characters:
\[
{\rm Spf}(\Lambda_n)(\C_p)=\{\chi\in \Hom_{\rm cnt}(\mfo_{F_p}^\times,\C_p^\times); \;|\chi(e_1)-1|<p^{-n}\}.
\]
Thus, the algebras $\Lambda_n$ define a basis of neighborhoods for the trivial character, and $\varphi$ factors through $\Lambda_F\hookrightarrow\Lambda_n\stackrel{\varphi_n}{\rightarrow}\Z_p$. Since $\D_{{\bf k}_p^{-2}}(\Lambda_F)=\Hom_{\Lambda_F}(C^0_{{\bf k}_p^{-2}}(\mcL_S,\Lambda_F),\Lambda_F)$ (similarly for $\D_{{\bf k}_p^{-2}}(\Lambda_n)$) 
and $C^0_{{\bf k}_p^{-2}}(\mcL_S,\Lambda_n)=C^0_{{\bf k}_p^{-2}}(\mcL_S,\Lambda_F)\otimes_{\Lambda_F}\Lambda_n$, we have that $\D_{{\bf k}_p^{-2}}(\Lambda_F)\subseteq \D_{{\bf k}_p^{-2}}(\Lambda_n)$ and we can
define 
\[
\D_{{\bf k}_p^{-2}}(\Lambda_F)^S=\varinjlim_n\D^{\varphi_n}_{{\bf k}_p^{-2}}(\Lambda_{n})^S\cap \D_{{\bf k}_p^{-2}}(\Lambda_F).
\]
We can interpret $\D_{{\bf k}_p^{-2}}(\Lambda_F)^S$ as the space of the measures in $\D_{{\bf k}_p^{-2}}(\Lambda_F)$ lying in $\D^{\varphi}_{{\bf k}_p^{-2}}(\Lambda_{F})^S$ after possibly shrinking ${\rm Spf}(\Lambda_F)$.  
The specialization map $s_\varphi$ of \eqref{eqspecialization} provides a morphism
\begin{equation}\label{eq:specialization}
H_\ast^u\ipa{G(F)_+,\mcA^{p\cup\infty}(\D_{{\bf k}_p^{-2}}(\Lambda_F)_{\underline{\bf a}^\ast}^S(\varepsilon_S))}^\lambda\stackrel{s_{\varphi}}{\longrightarrow}H_\ast^u\ipa{G(F)_+,\mcA^{p\cup\infty}(\D^{\rm id}_{1}(\Z_p)_{\underline{\alpha}^\ast}^S(\varepsilon_S))}^\lambda,\qquad \underline{\bf a}^\ast=\ipa{\varepsilon_\mfp\bsm\varpi_\mfp&\\&1\esm\cdot{\bf a}_\mfp}_\mfp.
\end{equation}
We will dedicate the last part of the section to prove that, in this particular situation where $\chi_p=1$, $\hat\phi_\lambda^p$ lifts to an element in $H_\ast^u(G(F)_+,\mcA^{p\cup\infty}(\D_{{\bf k}_p^{-2}}(\Lambda_F)_{\underline{\bf a}^\ast}^S(\varepsilon_S)))^\lambda$.
By abuse of notation, we will consider $\Phi_\lambda^p$ as a Hida family either with coefficients in $\D_{{\bf k}_p^{-2}}(\Lambda_F)_{\underline{\bf a}^\ast}^S(\varepsilon_S)$ or in $\D_{{\bf k}_p^{-2}}(\Lambda_F)_{\underline{\bf a}}$.

\subsubsection{Local systems and group cohomology}

Given a compact subgroup $C\subseteq G(\A_F^\infty)$, we can construct the locally symmetric space
\[
Y_C:=G(F)_+\backslash G(F_\infty)_+\ti G(\A_F^\infty)/C_\infty C,
\]
where $C_\infty$ is the maximal compact subgroup of $G(F_\infty)_+$. Notice that, when $C$ is small enough, $Y_C$ is in correspondence with the set of complex points of a Shimura variety. 

Let $C_p:=C\cap G(F_p)$. Given a $C_p$-module $V$, we can define the local system 
\[
\mcV:=G(F)_+\backslash \ipa{G(F_\infty)_+\ti G(\A_F^\infty)\ti V}/C_\infty C\longrightarrow Y_C,
\]
where the left $G(F)_+$-action and right $C_\infty C$-action on $G(\A_F)\ti V$ is given by $$\gamma(g_\infty,g^\infty,v)(c_\infty,c)=(\gamma g_\infty c_\infty, \gamma g^\infty c,c_p\inv v),$$ being $c_p\in C_p$ the $p$-component of $c\in C$. 

A locally constant section of the local system $\mcV$ amounts to a function $s:G(\A_F^\infty)\rightarrow V$ such that $s(g^\infty)=c_ps(\gamma g^\infty c)$, for all $\gamma \in G(F)_+$ and $c\in C$. 
Thus, to provide such an $s$ is equivalent to provide an element
\[
\hat s\in H^0\ipa{G(F)_+,\mcA^{p\cup\infty}({\rm coInd}_{C_p}^{G(F_p)}V)^{C^p}},\qquad \hat s(g^p)(g_p):=s(g_p^{-1},g^p),
\]
where $C^p:=C\cap G(\A_F^{p\cup\infty})$. In fact, 
we can identify the cohomology of the sheaf of local sections of $\mcV$ with the group cohomology of $\mcA^{p\cup\infty}({\rm coInd}_{C_p}^{G(F_p)}V)^{C^p}$, namely,
\begin{equation}\label{relSCGC}
H^k(Y_C,\mcV)\simeq H^k\ipa{G(F)_+,\mcA^{p\cup\infty}({\rm coInd}_{C_p}^{G(F_p)}V)^{C^p}}    
\end{equation}
(see \cite[Lemma 1.2]{L-P}).
Indeed, the assignment $V \mapsto H^\ast\ipa{G(F)_+,\mcA^{p\cup\infty}({\rm coInd}_{C_p}^{G(F_p)}V)^{C^p}}$ defines an effaceable $\delta$-functor because, similarly as in remark \ref{remonadmrepcoho}, we have $\mcA^{p\cup\infty}({\rm coInd}_{C_p}^{G(F_p)}V)^{C^p}\simeq\bigoplus_{g_i\in {\rm Cl}(C)}{\rm coInd}_{\Gamma_{g_i}}^{G(F)_+}V$, with ${\rm Cl}(C)=G(F)_+\backslash G(\A_F^\infty)/C$ and $\Gamma_{g_i}:=g_i Cg_i^{-1}\cap G(F)_+$. Hence, it is enough to prove that the assignment $V\mapsto H^\ast(Y_C,\mcV)$ defines an effaceable $\delta$-functor as well. 

To show that it is a $\delta$-functor, it is enough to prove that the functor $V\mapsto \mcV$ is exact. Hence, it is enough to prove exactness at the level of stalks. Recall that we can write $Y_C=\bigsqcup_i \Gamma_{g_i}\backslash H$, where $H=G(F_\infty)/C_\infty$, and for a point $P\in Y_C$ corresponding to $x\in H$, its stalk $\mcV_P$ is $V^{\Gamma_{g_i}^x}$, where $\Gamma_{g_i}^x$ is the stabilizer of $x$ in $\Gamma_{g_i}$. Since $\Gamma_{g_i}^x$ lies in the center of $\Gamma_{g_i}$ which is trivial, the claim follows.

To prove that it is effaceable, notice that we have a monomorphism $V\mapsto {\rm coInd}_1^{C_p}V$. Thus we have to show that the local system $\mcM$ associated with $M={\rm coInd}_1^{C_p}V$ is acyclic. Write $Y_{C^p}:=G(F)_+\backslash G(F_\infty)_+\times G(\A_F^\infty)/C_\infty C^p$ and let $\pi:Y_{C^p}\rightarrow Y_C$ be the natural projection. One can describe the sheaf associated with $\mcM$ as 
\[
\mcM(U)=\{s:\pi^{-1}(U)\longrightarrow M;\;s(u c_p)=c_p^{-1}s(u),\mbox{ for all }c_p\in C_p\}.
\]
Since $M={\rm coInd}_1^{C_p}V$, we directly deduce that $\mcM=\pi_\ast(\tilde V)$, where $\tilde V$ the constant sheaf on $Y_{C^p}$ associated with $V$. Since the fibres of $\pi$ are discrete, the higher direct images $R^\bullet(\pi_\ast)(\tilde V)$ vanish and we get $H^\bullet (Y_{C^p},\tilde V)=H^\bullet(Y_C,\mcM)$.
Since $Y_{C^p}=\sqcup_{G(F_p)/G(F)_+}\ipa{G(F_\infty)_+/C_\infty}\times \ipa{G(\A_F^{\infty\cup p})/C^p}$ is a disjoint union of contractible spaces, $H^k(Y_C,\mcM)=0$ for $k>0$ and the claim \eqref{relSCGC} follows.

\subsubsection{Families and eigenvarieties}

Write $\mcD(\mfo_{F_p},\Lambda_F)$ for the local system associated with $D(\mfo_{F_p},\Lambda_F)$ with the natural $\mcI_p$-action provided by ${\bf k}_p$. The classical strategy to construct the ordinary eigenvariety is to consider the subspace of $H^k(Y_C,\mcD(\mfo_{F_p},\Lambda_F))$ where $U_\mfp$ acts as invertible operator, being $C\subset G(\A_F^{\infty})$ a compact open subgroup such that $C_p=\mcI_p$. 
A connected component of the eigenvariety passing through $\pi$ provides a system of eigenvalues for the Hecke operators in $\Lambda_F$. 
Once we have such a system of eigenvalues, and in particular eigenvalues ${\bf a}_\mfp$ for the $U_\mfp$-operators, 
one can make use of the étaleness of the eigenvariety to prove the existence of eigenvectors in $H^u(Y_C,\mcD(\mfo_{F_p},\Lambda_F))^\lambda$.
(see \cite[theorem 2.14]{BDJ} for the case $G=\PGL_2$, and notice that techniques of \cite[\S 2.5]{BDJ} can be extended to general $G$). 
By equation \eqref{relSCGC}
\[
H^u(Y_C,\mcD(\mfo_{F_p},\Lambda_F))^{\lambda,U_\mfp={\bf a}_\mfp}=H^u\ipa{G(F)_+,\mcA^{p\cup\infty}({\rm coInd}_{\mcI_p}^{G(F_p)}D(\mfo_{F_p},\Lambda_F))^{C^p}}^{\lambda,U_\mfp={\bf a}_\mfp}.
\]
Moreover, by propositions \ref{propDLDO} and \ref{propKoszul}, we have the Koszul resolution $0\rightarrow \D_{{\bf k}_p^{-2}}(\Lambda_F)_{{\bf\underline{a}}}\stackrel{\varphi}{\rightarrow}K^\bullet$
\[
K^\bullet:{\rm coInd}_{\mcI_p}^{G(F_S)}D(\mfo_{F_p},\Lambda_F)\stackrel{\oplus_\mfp(U_\mfp-\alpha_\mfp)}{\longrightarrow}\left({\rm coInd}_{\mcI_p}^{G(F_p)}D(\mfo_{F_p},\Lambda_F)\right)^{g_p}\longrightarrow\cdots\longrightarrow \left({\rm coInd}_{\mcI_p}^{G(F_p)}D(\mfo_{F_p},\Lambda_F)\right)^m\longrightarrow 0
\]
Thus, we have convergence of the spectral sequence $H^s\left(H^r(G(F)_+,(K^\bullet)^{C^p})^\lambda\right)\Rightarrow H^{r+s}(G(F)_+,\D_{{\bf k}_p^{-2}}(\Lambda_F)_{{\bf\underline{a}}}^{C^p})^\lambda$. If we denote by $M_{\Phi_\lambda^p}$ the subspace of $M$ where Hecke operators associated with $G(\A_F^{p\cup\infty})/C^p$ act as the eigenvalues of $\Phi_\lambda^p$, then $H^r(G(F)_+,(K^\bullet)^{C^p})_{\Phi_\lambda^p}^\lambda=0$ for $r<u$. Indeed, $u$ is the minimum degree where these eigenvalues appear in $H^r(Y_C,\mcD(\mfo_{F_p},\Lambda_F))^{\lambda}$ and $K^\bullet$ is made up of powers of ${\rm coInd}_{\mcI_p}^{G(F_p)}D(\mfo_{F_p},\Lambda_F)$.
We obtain
\[
H^{u}(G(F)_+,\D_{{\bf k}_p^{-2}}(\Lambda_F)_{{\bf\underline{a}}}^{C^p})_{\Phi_\lambda^p}^\lambda=H^0\left(H^u(G(F)_+,(K^\bullet)^{C^p})^\lambda_{\Phi_\lambda^p}\right)=H^u\ipa{G(F)_+,\mcA^{p\cup\infty}({\rm coInd}_{\mcI_p}^{G(F_p)}D(\mfo_{F_p},\Lambda_F))^{C^p}}^{\lambda,U_\mfp={\bf a}_\mfp}_{\Phi_\lambda^p}.
\]
Thus, the existence of a classical family $\Phi_\lambda^p\in H^u(Y_C,\mcD_{{\bf k}_p}(\mfo_{F_p},\Lambda_F))^{\lambda,U_\mfp={\bf a}_\mfp}$ ensures the existence of a lift $\Phi_\lambda^p\in H_\ast^u(G(F)_+,\mcA^{p\cup\infty}(\D_{{\bf k}_p^{-2}}(\Lambda_F)_{{\bf\underline{a}}}))^{\lambda}$. 

Assume that $\chi_p=1$ and $\hat\phi_\lambda^p\in H_\ast^u(G(F)_+,\mcA^{p\cup\infty}(\D^{\rm id}_{1}(\Z_p)_{\underline{\alpha}^\ast}^S(\varepsilon_S)))^\lambda$.
It remains to prove that the family $\Phi_\lambda^p$ has a preimage in $H_\ast^u(G(F)_+,\mcA^{p\cup\infty}(\D_{{\bf k}_p^{-2}}(\Lambda_F)_{\underline{\bf a}^\ast}^S(\varepsilon_S)))^{\lambda}$. Notice that $\Lambda_F=\widehat\bigotimes_{\mfp\mid p}\Lambda_\mfp$, where $\Lambda_\mfp$ is the Iwasawa algebra associated with $\mfo_{F_\mfp}^\times$, and we have specialization maps $\varphi_\mfp:\Lambda_F\rightarrow \Lambda_{p\setminus\mfp}:=\widehat\bigotimes_{\mfq\neq\mfp}\Lambda_\mfq$ associated to the natural characters $\mfo_F^\times\rightarrow\Lambda_{p\setminus\mfp}^\times$.
By \cite[Lemma 7.2]{BDJ}, for every $\mfp\in S$, the family $\Phi_\lambda^p$ specializes at $\varphi_\mfp$ to 
\[
\varphi_\mfp(\Phi_\lambda^p)\in H_\ast^u\ipa{G(F)_+,\mcA^{p\cup\infty}(\tno{St}_{\Z_p}(F_\mfp),\D_{{\bf k}_{p\setminus\mfp}^{-2}}(\Lambda_{p\setminus\mfp})_{\underline{\bf a}_{p\setminus\mfp}^\ast})(\varepsilon_S)}^{\lambda},\qquad \underline{\bf a}_{p\setminus\mfp}^\ast=\ipa{\varepsilon_\mfq\bsm\varpi_\mfq&\\&1\esm\cdot\varphi_\mfp({\bf a}_\mfq)}_{\mfq\neq\mfp},
\]
after possibly shrinking $\Lambda_{p\setminus\mfp}$. 
Hence, \cite[lemma 7.2]{BDJ} and lemma \ref{measuresonI} imply that $\Phi_\lambda^p$ has values in 
\[
D_S:=\left\{\mu\in \D_{{\bf k}_p^{-2}}(\Lambda_F);\quad \mu(\tilde\mcE_\mfp)\subset \ker(\varphi_\mfp)\oti_{\Lambda_F}\D_{{\bf k}_{p\setminus\mfp}^{-2}}(\Lambda_n)\;\mbox{for all $\mfp\in S$ and some $n\in\N$}\right\}.
\]
The following lemma shows that $D_S\subseteq \D_{{\bf k}_p^{-2}}(\Lambda_F)^S$ and, thus, ensures the existence of our desired family $\Phi_\lambda^p\in H_\ast^u(G(F)_+,\mcA^{p\cup\infty}(\D_{{\bf k}_p^{-2}}(\Lambda_F)_{\underline{\bf a}^\ast}^S(\varepsilon_S)))^{\lambda}$ specializing to $\phi_\lambda^p$:
\begin{lemma}
    For any subset $S'\subseteq S$, we have 
    $\bigcap_{\mfp\in S'}\ker(\varphi_\mfp)=\prod_{\mfp\in S'}\ker(\varphi_\mfp)\subseteq \ipa{\ker \varphi}^{\#S'}$.   
\end{lemma}
\begin{proof}
    It is clear that $\ker(\varphi_\mfp)\subseteq \ker \varphi$. 
    Moreover, $\bigcap_{\mfp\in S'}\ker(\varphi_\mfp)\supseteq\prod_{\mfp\in S'}\ker(\varphi_\mfp)$. We will prove the remaining inclusion by induction: Notice that there exists a morphism $\imath_\mfp:\Lambda_{p\setminus\mfp}\hookrightarrow\Lambda_F$ such that $\varphi_\mfp\circ\imath_\mfp={\rm id}$. If we assume that $S'=\{\mfp,\mfq\}$, the product $\Lambda_{p\setminus\mfp}\Lambda_{p\setminus\mfq}$ topologically generates $\Lambda_F$. Thus, any $x\in \ker(\varphi_\mfp)\cap \ker(\varphi_\mfq)$ is given by a convergent series
    $x=\sum_i \imath_\mfp(a_{i})\imath_\mfq(b_{i})$, where  $a_{i}\in\Lambda_{p\setminus\mfp}$ and $b_{i}\in\Lambda_{p\setminus\mfq}$. We obtain
    \[
    \ker(\varphi_\mfq)\cdot \ker(\varphi_\mfp)\ni\sum_i\ipa{\imath_{\mfp}(a_{i})-\imath_{\mfq}\varphi_{\mfq}\imath_{\mfp}(a_{i})}\ipa{\imath_{\mfq}(b_{i})-\imath_{\mfp}\varphi_{\mfp}\imath_{\mfq}(b_{i})}=x-\imath_\mfp \varphi_\mfp(x)-\imath_\mfq \varphi_\mfq(x)+\imath_\mfp \varphi_\mfp\imath_\mfq \varphi_\mfq(x)=x,
    \]
    since $\imath_\mfp \varphi_\mfp\imath_\mfq \varphi_\mfq=\imath_\mfq \varphi_\mfq\imath_\mfp \varphi_\mfp$. This shows the claim for $\#S'=2$. If we assume that the result is true for any smaller subgroup of $S'$, for any $\mfq\in S'$ we have that 
    \[
    \bigcap_{\mfp\in S'}\ker(\varphi_\mfp)=\ipa{\prod_{\mfp\in S'\setminus\mfq}\ker(\varphi_\mfp)}\cap\ker(\varphi_\mfq)=\bigcup_{\mfp\in S'\setminus\mfq}\ipa{\ker(\varphi_\mfq)\cap\ker(\varphi_\mfp)}\cdot\ipa{\prod_{\mathfrak{r}\in S'\setminus\{\mfq,\mfp\}}\ker(\varphi_{\mathfrak{r}})}=\prod_{\mfp\in S'}\ker(\varphi_\mfp),
    \]
    where the second equality follows from the fact that $\ker(\varphi_\mfq)$ is prime. Thus, we have proved the result.
\end{proof}

\subsection{Extensions of the Steinberg representation}\label{subsection.steinbergrep}

If our local representation $\pi_\mfp$ is Steinberg, then $\pi_\mfp$ is the quotient of the induced representation ${\rm Ind}_P^G1$ modulo constant functions. Since $G(F_\mfp)/P\simeq\PP^1(F_\mfp)$, we can give the simpler description  $V_\mfp^\Z=\tno{St}_\Z(F_\mfp)=C^0(\PP^1(F_\mfp),\Z)/\Z$,
where $C^0(\PP^1(F_\mfp),\Z)$ is the set of locally constant $\Z$-valued functions of the projective line, with the natural action of $G(F_\mfp)=\PGL_2(F_\mfp)$ induced by 
fractional linear transformations. In this section we will extend these objects to arbitrary coefficients. 
\newcommand{\ddivz}{\Delta^0}
\newcommand{\ddivv}{\Delta}

Denote by $\tno{Cov}(X)$ the poset of open coverings of a topological space $X$ ordered by refinement. For any topological $\Z_p$-algebra $A$ and $A$-module $M$, we write  $\tno{St}_{M}:=C(\PP^1(F_\mfp),M)/M$. Then
we have a natural morphism
\begin{equation}\label{morextSt2}
   \lambda_{\tno{unv}}:\Hom_{{\rm cont},A}\ipa{\bigotimes_{\mfp\in S}\tno{St}_{A}(F_\mfp),A}\longrightarrow \Hom_{{\rm cont},A}\ipa{\bigotimes_{\mfp\in S}\tno{St}_{M},\widehat{\bigotimes_{\mfp\in S}M}},
\end{equation}
where the tensor products are as $A$-modules, $\widehat{(\cdot)}$ stands for the completion and, for any choice of $x_{U_\mfp}\in U_\mfp$,
\[
\lambda_{\tno{unv}}(\psi)\ipa{\bigotimes_{\mfp\in S}f_\mfp}:=\lim_{\mcU_\mfp\in\tno{Cov}(\PP^1(F_\mfp))}\sum_{U_\mfp\in\mcU_\mfp}\bigotimes_{\mfp\in S}f(x_{U_\mfp})\cdot{\psi\ipa{\bigotimes_{\mfp\in S}\indi_{U_\mfp}}}.
\]

For any continuous character $\ell:F_\mfp^\times\rightarrow M$, we can define the extension of $\tno{St}_{M}$:
\begin{equation}\label{defextSt}
    \tno{St}_{M}\stackrel{f\mapsto (f,0)}{\hookrightarrow}\mcE(\ell):=\left\{(\phi,y)\in C(\GL_2(F_\mfp),M)\ti A:\phi\left(\bbm s & x \\  & t \ebm g\right)=y\ell(t)+\phi(g) \right\}/(M,0).
\end{equation}

As in \S \ref{spe-mor}, let $R$ and $A$ be complete $\Z_p$-algebras and let $\varphi:R\rightarrow A$ be a continuous surjective morphism. Let $S$ be a finite set of places above $p$, let ${\udl \alpha}\in (R^{\ti})^S$ and let $\chi_S:\mfo_{F_S}^\ti\rightarrow R^\times$ be a continuous character. For a subset $S_1\subset S$ of places $\mfp$ such that $\varphi\circ\chi_\mfp=1$ and $\varphi(\alpha_\mfp)=1$, we have constructed the $G(F_S)$-invariant subspace $\D^\varphi_{\chi_S}(R)_{\udl\alpha}^{S_1}\subseteq \D_{\chi_S}(R)_{\udl\alpha}$.

Let $\hat\chi_{\mfp}:F_\mfp^\ti\ra R^\ti$ be the extension of $\chi_\mfp$ provided by $\alpha_\mfp$. Since $\varphi(\alpha_\mfp)=1$, we define the group morphism: 
\begin{equation}\label{defeqell}
    \ell_{\alpha_\mfp}:F_\mfp^\ti\longrightarrow I/I^2;\qquad \alpha\mapsto (\hat\chi_{\mfp}(\alpha)^{-1}-1)+I^2,
\end{equation}
where $I=\ker\varphi$.
Hence, we can define the $A$-module extension $\tno{St}_{I/I^2}\stackrel{\iota}{\hookrightarrow}\mcE(2\ell_{\alpha_\mfp})$.
\begin{lemma}
    For any $\mfp\in S_1$, we have a $G(F_\mfp)$-equivariant isomorphism of $A$-modules
    \[
    \kappa_\mfp:\mcE(2\ell_{\alpha_\mfp})\stackrel{\simeq}{\longrightarrow}\varphi^{(2)}_\ast(\tilde \mcE_\mfp)/(I/I^2),\qquad \kappa_\mfp(\phi,y)(c,d) :=\tilde y+\phi\bbm a&b\\c&d\ebm,\qquad \bbm a&b\\c&d\ebm\in \GL_2(\mfo_{F_\mfp}),
\]
    where $\tilde y\in R/I^2$ is any preimage of $y$ and $\varphi^{(2)}:R\rightarrow R/I^2$ is the natural projection.
\end{lemma}
\begin{proof}
       Notice that the above expression does not depend on the class of $(\phi,y)$ in $\mcE(2\ell_{\alpha_\mfp})$  and  does not depend on the choice $\tilde y$. Moreover, $\kappa$ is well defined since for any other $\bsm a'&b'\\\lambda c&\lambda d\esm\in \GL_2(\mfo_{F_\mfp})$ there exist $\bsm t&x\\&\lambda\esm\in \GL_2(\mfo_{F_\mfp})$ such that $\bsm a'&b'\\\lambda c&\lambda d\esm=\bsm t&x\\&\lambda\esm\bsm a&b\\c&d\esm$, 
hence,
\[
\tilde y+\phi\bbm a'& b'\\\lambda c&\lambda d\ebm=\tilde y+\tilde y2\ell_{\alpha_\mfp}(\lambda)+\phi\bbm a&b\\c&d\ebm=\tilde y\varphi^{(2)}\ipa{\chi_\mfp(\lambda)^{-2}}+\phi\bbm a&b\\c&d\ebm= \varphi^{(2)}\circ\chi_\mfp(\lambda)^{-2}\ipa{\tilde y+\phi\bbm a&b\\c&d\ebm}.
\]
Moreover, $\kappa_\mfp$ is $G(F_\mfp)$-equivariant since, 
 by equation \eqref{actiononC}, if $x\in F_\mfp^\ti$ is such that $x\inv(c,d)g\in \mcL_\mfp$,
\begin{eqnarray*}
    (g\kappa_\mfp(\phi,y))(c,d)&=&\varphi^{(2)}\circ\hat\chi_\mfp(x)^{-2}\cdot \varphi^{(2)}\circ\hat\chi_\mfp(\det g)\cdot \kappa(\phi,y)(x\inv(c,d)g)\\
    &=&\varphi^{(2)}\circ\hat\chi_\mfp(x)^{-2}\cdot \varphi^{(2)}\circ\hat\chi_\mfp(\det g)\cdot\ipa{\tilde y+\phi\left(\bsm s&t\\&x^{-1}\esm\bsm a&b\\c&d\esm g\right)}\\
    &=&\varphi^{(2)}\circ\hat\chi_\mfp(\det g)\cdot \ipa{\tilde y+\phi\left(\bsm a&b\\c&d\esm g\right)}=\varphi^{(2)}\circ\hat\chi_\mfp(\det g)\cdot \kappa_\mfp(g(\phi,y))(c,d),
\end{eqnarray*}
where $\bsm s&t\\&x^{-1}\esm\in P$ is so that $\bsm s&t\\&x^{-1}\esm\bsm a&b\\c&d\esm g\in G(\mfo_{F_\mfp})$.
Since $\varphi\ipa{\hat\chi_\mfp(\det g)}=1$, we obtain
$(g\kappa_\mfp(\phi,y))\equiv \kappa_\mfp(g(\phi,y))$ modulo $I/I^2$. Bijectivity is obvious, hence, the result follows. 
\end{proof}

\begin{proposition}\label{D-EforgenS}
If we write $r=\#S_1$, then we have a $G(F_S)$-equivariant morphism 
\[
r_{S_1}:\D^\varphi_{\chi^{-2}_S}(R)^{S_1}_{\udl \alpha}\longrightarrow \Hom_{{\rm cont},A}\ipa{{\rm Ind}_P^G(\varphi\circ\hat\chi_{S\setminus S_1})\oti_A\bigotimes_{\mfp\in S_1}\mcE(2\ell_{\alpha_\mfp}),I^r/I^{r+1}},
\]
where, if we write $\varphi^{(r+1)}:R\rightarrow R/I^{r+1}$ for the natural projection, it is provided by
\[
r_{S_1}(\mu)\ipa{\psi_{S\setminus S_1}f_{S\setminus S_1}\oti\bigotimes_{\mfp\in S_1} (\phi_{\mfp},y_\mfp)}=\int_{\mcL_S}\ipa{\widetilde{f_{S\setminus S_1}}\cdot\prod_{\mfp\in S_1} \widetilde{\kappa_\mfp(\phi_\mfp,y_\mfp)}}ds_{\varphi^{(r+1)}}\mu,
\]
where $\widetilde{f_{S\setminus S_1}}\in C_{\varphi^{(r+1)}\circ\chi_{S\setminus S_1}}(\mcL_{S\setminus S_1},R/I^{r+1})$ is any preimage of $f_{S\setminus S_1}$ and $\widetilde{\kappa_\mfp(\phi_\mfp,y_\mfp)}\in \varphi^{(r+1)}_\ast\tilde\mcE_\mfp$ is any preimage of $\kappa_\mfp(\phi_\mfp,y_\mfp)$.
Moreover, it fits in the commutative diagram
\[\small
\xymatrix{
\D^\varphi_{\chi^{-2}_S}(R)^{S_1}_{\udl \alpha}\ar[r]^{r_{S_1}\qquad\qquad}\ar[d]^{s_\varphi}&\Hom_{{\rm cont},A}\ipa{{\rm Ind}_P^G(\varphi\circ\hat\chi_{S\setminus S_1})\oti_A\bigotimes_{\mfp\in S}\mcE(2\ell_{\alpha_\mfp}),I^r/I^{r+1}}\ar[d]^{\iota^\ast}\\
\Hom_{{\rm cont},A}\ipa{{\rm Ind}_P^G(\varphi\circ\hat\chi_{S\setminus S_1})\oti_A\bigotimes_{\mfp\in S}{\rm St}_A,A}\ar[r]^{m_\ast\circ\lambda_{\rm unv}\qquad}&\Hom_{{\rm cont},A}\ipa{{\rm Ind}_P^G(\varphi\circ\hat\chi_{S\setminus S_1})\oti_A\bigotimes_{\mfp\in S}\tno{St}_{I/I^2},I^r/I^{r+1}}
}
\]
where $m:\bigotimes_{\mfp\in S}I/I^2\rightarrow I^r/I^{r+1}$ is the natural morphism.
\end{proposition}
\begin{proof}
    We will prove the result by induction on $r$: If $r=0$ the result is obvious. Assume it is true for $r-1$. By definition, given $\mfp\in S_1$ and $e\in\tilde\mcE_\mfp$, we have
  $\mu(e)\in I\oti_R\D^{S_1\setminus\mfp}_{\chi_{S\setminus\mfp}^{-2}}(R)_{\udl\alpha^\mfp}$, where $\udl\alpha^\mfp=(\alpha_\mfq)_{\mfq\in S\setminus\mfp}$. Thus, we obtain applying induction hypothesis:
    \begin{eqnarray*}
        \D^\varphi_{\chi^{-2}_S}(R)^{S_1}_{\udl \alpha}&\longrightarrow&\Hom_{\mbox{\tiny cont, $R$}}\ipa{\tilde\mcE_\mfp,I\oti_R\D_{\chi_{S\setminus\mfp}^{-2}}(R)^{S_1\setminus\mfp}_{\udl\alpha^\mfp}}\stackrel{r_{S_1\setminus\mfp}}{\longrightarrow} \Hom_{\mbox{\tiny cont, $R$}}\ipa{\tilde\mcE_\mfp,I\oti_R\Hom_{\mbox{\tiny cont, $A$}}\ipa{\bigotimes_{\mfq\in S\setminus\mfp}\mcE(2\ell_{\alpha_\mfq}),I^{r-1}/I^{r}}}\\
        &\stackrel{\simeq}{\longrightarrow}&\Hom_{\mbox{\tiny cont, $R$}}\ipa{\tilde\mcE_\mfp,I/I^2\oti_A\Hom_{\mbox{\tiny cont, $A$}}\ipa{\bigotimes_{\mfq\in S\setminus\mfp}\mcE(2\ell_{\alpha_\mfq}),I^{r-1}/I^{r}}}\\
        &\stackrel{\simeq}{\longrightarrow}&\Hom_{\mbox{\tiny cont, $R$}}\ipa{\varphi^{(2)}_\ast(\tilde\mcE_\mfp),I/I^2\oti_A\Hom_{\mbox{\tiny cont, $A$}}\ipa{\bigotimes_{\mfq\in S\setminus\mfp}\mcE(2\ell_{\alpha_\mfq}),I^{r-1}/I^{r}}}\\
        &\stackrel{\simeq}{\longrightarrow}&\Hom_{\mbox{\tiny cont, $A$}}\ipa{\mcE(2\ell_{\alpha_\mfp}),I/I^2\oti_A\Hom_{\mbox{\tiny cont, $A$}}\ipa{\bigotimes_{\mfq\in S\setminus\mfp}\mcE(2\ell_{\alpha_\mfq}),I^{r-1}/I^{r}}}\\
        &\longrightarrow& \Hom_{\mbox{\tiny cont, $A$}}\ipa{\bigotimes_{\mfp\in S}\mcE(2\ell_{\alpha_\mfp}),I^r/I^{r+1}}
    \end{eqnarray*}
    where the forth arrow follows from remark \ref{remonhomcont}. This shows the existence of $r_{S_1}$ and its precise description.

    Finally, for the commutativity of the diagram, notice that given for every $\mfp\in S_1$ and $h_\mfp\in C(\PP^1(F_\mfp),I/I^2)$
\[
h_\mfp=\lim_{\mcU_\mfp\in\tno{Cov}(\PP^1(F_\mfp))}\sum_{U_\mfp\in\mcU_\mfp}m_{U_\mfp}\cdot\indi_{U_\mfp}, 
\qquad m_{U_\mfp}\in I/I^2.
\] 
Hence, if we choose $\tilde m_{U_\mfp}\in I/I^{r+1}$ preimages of $m_{U_\mfp}$, by remark \ref{remonhomcont}
\begin{eqnarray*}
    r_{S_1} \ipa{\psi_{S\setminus S_1}f_{S\setminus S_1}\oti\bigotimes_{\mfp\in S_1} (h_{\mfp},0)}&=&\lim_{\mcU_\mfp\in\tno{Cov}(\PP^1(F_\mfp))}\sum_{U_\mfp\in\mcU_\mfp}\varphi^{(r+1)}\ipa{\prod_{\mfp\in S_1}\tilde m_{U_\mfp}}\int_{\mcL_S}\ipa{f_{S\setminus S_1}\cdot\prod_{\mfp\in S_1} \indi_{U_\mfp}}ds_{\varphi}\mu\\
    &=&m_\ast \lambda_{\rm unv}(s_\varphi\mu)\ipa{\psi_{S\setminus S_1}f_{S\setminus S_1}\oti\bigotimes_{\mfp\in S_1} (h_{\mfp},0)},
\end{eqnarray*}
and the result follows.
    
\end{proof}

\subsection{Extensions and divisors}\label{ExandDiv}

Let us consider  $\uhp_\mfp=K_\mfp\setminus F_\mfp$ the $p$-adic upper half plane.
Let $\Delta_\mfp:=\tno{Div}(\uhp_\mfp)$ and $\Delta_\mfp^0:=\tno{Div}^0(\uhp_\mfp)\subset \Delta_\mfp$ be the set of divisors and degree zero divisors. 
Notice that the multiplicative integral provides a morphism 
\begin{equation}\label{definition.multiplicativeintegral}
\Hom(\tno{St}_{\Z}(F_\mfp),\Z) \longrightarrow \Hom(\ddivz_\mfp,K_\mfp^\ti);\quad  \psi\longmapsto\ipa{z_2-z_1\mapsto \mint_{{\scriptscriptstyle \PP^1(F_\mfp)}} \frac{x-z_2}{x-z_1}d\mu_{\psi}(x):=\lim_{\mcU\in\tno{Cov}(\PP^1(F_\mfp))}\prod_{U\in\mcU}\ipa{\frac{x_U-z_2}{x_U-z_1}}^{\psi(\indi_U)}},
\end{equation}
where each $x_U\in U$. 
In \cite{guitart2017automorphic} this multiplicative integral is described alternatively as 
\[
\mint_{{\scriptscriptstyle \PP^1(F_\mfp)}} \frac{x-z_2}{x-z_1}d\mu_{\psi}(x)=\ipa{{\rm ev_\mfp^\ast}\circ\lambda_{\rm unv}}(z_2-z_1)
\]
where $\lambda_{\rm unv}$ is the morphism described in \eqref{morextSt2}, and ${\rm ev_\mfp^\ast}:\Hom(\tno{St}_{K_\mfp^\ti},K_{\mfp}^\ti)\rightarrow \Hom(\ddivv_\mfp^0,K_{\mfp}^\ti)$ is the pullback provided by the morphism of the commutative diagram 
\begin{equation}\label{commdiagramevuniv}
\begin{tikzcd}
0\ar[r]&\ddivv_\mfp^0\ar[r]\ar[d,"\tno{ev}_\mfp"]&\ddivv_\mfp\ar[r]\ar[d,"\tno{ev}_\mfp"]&\Z\ar[r]\ar[d]&0\\
0\ar[r]&\tno{St}_{K_\mfp^\ti}\ar[r,"\imath_{\rm St}"]&\mcE_{K_\mfp^\ti}\ar[r]&\Z_p\ar[r]&0
\end{tikzcd}
\end{equation}
where $\mcE_{K_\mfp^\ti}=\mcE(F_\mfp^\ti\hookrightarrow K_\mfp^\ti)$ and 
\begin{equation}\label{defev}
    \tno{ev}(z):=(\phi_z,1),\qquad \phi_z\bbm a&b\\c&d\ebm:=cz+d\in K_\mfp^\ti.
\end{equation}

Let $R$ be a complete $\Z_p$-algebra endowed with a continuous $\Z_p$-algebra morphism $\varphi: R\rightarrow \Z_p$. Let $\chi_p:\mfo_{F_p}^\ti\rightarrow R^\ti$ be a continuous character such that $\varphi\circ\chi_p=1$, and fix ${\udl \alpha}=(\alpha_\mfp)_{\mfp\mid p}$, where $\alpha_\mfp\in R^\ti$, providing an extension $\hat\chi_p$ of $\chi_p$ to $F_p^\ti$. 
In the following sections $R$ will be the Iwasawa algebra $\Lambda_F$, $\chi_p$ the universal character ${\bf k}_p$, and $\varphi=\varphi_1$ for the trivial character $1:\Lambda_F\rightarrow\Z_p$, but in this section we will work with this more general scenario.
Let $S$ be a set of primes above $\mfp$ such that $\varphi(\alpha_\mfp)=1$ and write $I=\ker(\varphi)$. Write $\Delta_{F_\mfp}\subset\Delta_\mfp$ for the subgroup of $\Gal(K_\mfp/F_\mfp)$-invariant divisors, namely, the even degree divisors generated by those of the form $\tau+\bar\tau$. Write $\Delta_{F_\mfp}^0$ for the degree zero subgroup.
Let 
\[
\mcE(2\ell_{\alpha_S}):=\bigotimes_{\mfp\in S}\mcE(2\ell_{\alpha_\mfp}),\quad \hat F_S^\ti=\bigotimes_{\mfp\in S}\hat F_\mfp^\ti,\quad \hat K_S^\ti=\bigotimes_{\mfp\in S}\hat K_\mfp^\ti,\quad \tno{St}_{\Z_p}(F_S):=\bigotimes_{\mfp\in S}\tno{St}_{\Z_p}(F_\mfp),
\]
where $\hat F_\mfp^\ti$ and $\hat K_\mfp^\ti$ denote $p$-adic completions and tensor products are taken with respect to $\Z_p$.
We recall the morphisms $\ell_{\alpha_\mfp}:F_\mfp^\ti\rightarrow I/I^2$ of \eqref{defeqell}, providing the morphism
\begin{equation}
    \ell_{\alpha_S}:\hat F_S^\ti\stackrel{\bigotimes\ell_{\alpha_\mfp}}{\longrightarrow} \bigotimes_{\mfp\in S}I/I^2\stackrel{m}{\longrightarrow} I^r/I^{r+1}.
\end{equation}
Moreover, we define
\[
\Delta_{S}:=\bigotimes_{\mfp\in S}\Delta_{\mfp},\quad \Delta_{S}^0:=\bigotimes_{\mfp\in S}\Delta_{\mfp}^0,\quad\Delta_{F_S}:=\bigotimes_{\mfp\in S}\Delta_{F_\mfp},\quad \Delta_{F_S}^0:=\bigotimes_{\mfp\in S}\Delta_{F_\mfp}^0,\quad \tno{St}_{F_S^\ti}:=\bigotimes_{\mfp\in S}\tno{St}_{F_\mfp^\ti},\quad \mcE_{K_S^\ti}:=\bigotimes_{\mfp\in S}\mcE_{K_\mfp^\ti},
\]
where the tensor products are with respect to $\Z$. 
Note that ${\rm ev}$ of \eqref{defev} extends to
$\tno{ev}_S:\Delta_S\longrightarrow\mcE_{K_S^\ti}$.

\begin{proposition}\label{corevdiv}
If we write $V^S=\bigotimes_{\mfp\not\in S}\Ind_P^G(\varphi\circ\hat\chi_\mfp)$, then we have $G(F_p)$-equivariant morphisms
\begin{eqnarray*}
\bar{\rm ev}_S:\Hom_{\Z_p}\ipa{V^S\oti\mcE(2\ell_{\alpha_S}),I^r/I^{r+1}}&\longrightarrow& \Hom_{\Z_p}\ipa{V^S\oti\Delta_{F_S},I^r/I^{r+1}},\\ 
\varphi&\longmapsto& \ipa{v^S\oti\bigotimes_{\mfp\in S}(z_\mfp+\bar z_\mfp)\mapsto\varphi\ipa{v^S\oti\bigotimes_{\mfp\in S}(\ell_{\alpha_\mfp}(\phi_{z_\mfp}\cdot\phi_{\bar z_\mfp}),1)}},
\end{eqnarray*}
making the following diagram commutative
\begin{center}
\begin{tikzcd}[nodes={inner sep=2pt}]
& \Hom_{\Z_p}(V^S\oti\Delta_{F_S},I^r/I^{r+1})\ar[rr]&& \Hom_{\Z_p}(V^S\oti\Delta_{F_S}^0,I^r/I^{r+1}) \\[-1.5em]
\D^\varphi_{\chi_p^{-2}}(R)^S_{\underline{\alpha}}\ar[ru,"{\bar{\rm ev}_S\circ r_S}"]\ar[rd,"{s_\varphi}"] \\[-1.5em]
& \Hom_{\Z_p}(V^S\oti\tno{St}_{\Z_p}(F_S),\Z_p)\ar[rr,"{{\rm ev}_S^\ast\circ\lambda_{\rm unv}}"]&&\Hom_{\Z_p}(V^S\oti\Delta_{F_S}^0,\hat F_S^\ti)\ar[uu,"\ell_{\alpha_S}"]
\end{tikzcd}
\end{center}
\end{proposition}
\newcommand{\limcovpu}{{\lim_{\mcU\in\tno{Cov}(\PP^1(F_\mfp))}}}
\begin{proof}
    By proposition \ref{D-EforgenS},
    \[
    \bar{\rm ev}_S\circ r_S\mid_{V^S\oti\Delta_{F_S}^0}=\bar{\rm ev}_S\circ \iota^\ast\circ r_S=\bar{\rm ev}_S\circ m_\ast\circ \lambda_{\rm unv}\circ s_\varphi,
    \]
    where in this expression $\lambda_{\rm unv}:\Hom_{\Z_p}(V^S\oti\tno{St}_{\Z_p}(F_S),\Z_p)\rightarrow \Hom_{\Z_p}(V^S\oti\bigotimes_{\mfp\in S}\tno{St}_{\bar I/I^2},\bigotimes_{\mfp\in S}I/I^2)$. The result then follows from the easily verifiable relation $\bar{\rm ev}_S\circ m_\ast\circ  \lambda_{\rm unv}=\ell_{\alpha_S}{\rm ev}_S^\ast\circ\lambda_{\rm unv}$.
\end{proof}

\subsection{$p$-adic periods and L-invariants}\label{Linv}

Fix a prime $\mfp\mid p$ and let $\phi_\lambda^\mfp\in H_\ast^u(G(F)_+,\mcA^{\mfp\cup\infty}({\rm St}_{\Z_p}(F_\mfp),{\Z_p}))^\lambda$ be a modular symbol associated with an elliptic curve $E/F$ with multiplicative reduction at $\mfp$ as in remark \ref{remintMS}. 
The short exact sequence 
\[
0\longrightarrow \Delta_\mfp^0\longrightarrow\Delta_\mfp\longrightarrow\Z\longrightarrow 0,
\]
provides a connection morphism
\[
H_\ast^u(G(F)_+,\mcA^{\mfp\cup\infty}(\Delta_\mfp^0,\hat K_\mfp^\ti))^\lambda\stackrel{c}{\longrightarrow} H_\ast^{u+1}(G(F)_+,\mcA^{\mfp\cup\infty}(\hat K_\mfp^\ti))^\lambda.
\]
Moreover, commutative diagram \eqref{commdiagramevuniv} shows that 
\[
c({\rm ev}_\mfp^\ast\circ\lambda_{\rm unv})\phi_\lambda^\mfp\in H_\ast^{u+1}(G(F)_+,\mcA^{\mfp\cup\infty}(q_{E,\mfp}^{\Z_p}))^\lambda\subset H_\ast^{u+1}(G(F)_+,\mcA^{\mfp\cup\infty}(\hat F_\mfp^\ti))^\lambda,
\]
for some $q_{E,\mfp}\in F_\mfp^\ti$. 
The generalized Oda's conjecture (\cite[Conjecture 3.8]{guitart2017automorphic}) asserts that, in fact, the elliptic curve $E/F_\mfp$ is isogenous to the Tate curve defined by the quotient $\bar F_\mfp/q_{E,\mfp}^\Z$. Inspired by the classical work of Greenberg and Stevens (see \cite{GS}), Gehrmann and Rosso prove Oda's conjecture in \cite{gehrmann} and their strategy relies on relating the automorphic period $q_{E,\mfp}$ with the derivative of the $U_\mfp$-eigenvalue of a family passing through $\phi_\lambda^\mfp$. It turns out that our work allows us to easily deduce their same result:
\begin{theorem}\label{THMvanishq}
Let $\Phi_\lambda^\mfp\in H_\ast^u(G(F)_+,\mcA^{\mfp\cup\infty}(\D_{{\bf k}_\mfp^{-2}}(\Lambda_F)_{{\bf a}_\mfp}))^\lambda$ be a family lifting $\phi_\lambda^\mfp$ for some ${\bf a}_\mfp\in\Lambda_F^\ti$. Then
\[
\ell_{{\bf a}_\mfp}(q_{E,\mfp})=0\in I/I^2,
\]
where $I=\ker(\varphi_1:\Lambda_F\rightarrow\Z_p)$ is the kernel of the specialization morphism.
\end{theorem}
\begin{proof}
Since the family $\Phi_\lambda^\mfp$ specializes to the Steinberg representation, we deduce that it has a preimage
$\Phi_\lambda^\mfp\in H_\ast^u(G(F)_+,\mcA^{\mfp\cup\infty}(\D^\varphi_{{\bf k}_\mfp^{-2}}(\Lambda_F)^{\mfp}_{{\bf a}_\mfp}))^\lambda$, where $\varphi=\varphi_1$.
If we consider the exact sequence 
\[
H_\ast^u(G(F)_+,\mcA^{\mfp\cup\infty}(\Delta_{F_\mfp},\hat F_\mfp^\ti))^\lambda\stackrel{{\rm res}}{\longrightarrow} H_\ast^u(G(F)_+,\mcA^{\mfp\cup\infty}(\Delta_{F_\mfp}^0,\hat F_\mfp^\ti))^\lambda\stackrel{c}{\longrightarrow} H_\ast^{u+1}(G(F)_+,\mcA^{\mfp\cup\infty}(\hat F_\mfp^\ti))^\lambda,
\]
then by Proposition \ref{corevdiv},
\[
{\rm res}(\bar{\rm ev}_\mfp\circ r_\mfp)\Phi_\lambda^\mfp=\ell_{{\bf a}_\mfp}({\rm ev}_\mfp^\ast\circ\lambda_{\rm unv})\phi_\lambda^\mfp\mid_{\Delta_{F_\mfp}^0}.
\]
This implies that 
$\ell_{{\bf a}_\mfp}c({\rm ev}_\mfp^\ast\circ\lambda_{\rm unv})\phi_\lambda^\mfp\mid_{\Delta_{F_\mfp}^0}=0$.
Since $c({\rm ev}_\mfp^\ast\circ\lambda_{\rm unv})\phi_\lambda^\mfp\mid_{\Delta_{F_\mfp}^0}$ lies in $H_\ast^{u+1}(G(F)_+,\mcA^{\mfp\cup\infty}(q_{E,\mfp}^{2\Z_p}))^\lambda$, we deduce that $\ell_{{\bf a}_\mfp}(q_{E,\mfp})=0$ from the fact that $I/I^2$ is torsion free.
\end{proof}

As showed in \S \ref{locsystevsgroupcoho}, when $F$ is totally real the elements ${\bf a}_\mfp\in \Lambda_F$ are the eigenvalues of the $U_\mfp$-operators acting on the Hida family. By definition
$\ell_{{\bf a}_\mfp}(\varpi_\mfp)= {\bf a}_\mfp\inv-1\in I/I^2$.
Thus, the relation $\ell_{{\bf a}_\mfp}(q_{E,\mfp})=0$ implies
\[
\ord_{\mfp}(q_{E,\mfp})\ipa{{\bf a}_\mfp\inv-1}+\log_\mfp(q_{E,\mfp})\equiv0\;\;{\rm mod}\;I^2. 
\]
Hence, we have obtained that the image of ${\bf a}_\mfp-1$ in $I/I^2$ agrees with the automorphic $\mcL$-invariant $\mcL_\mfp:=\log_\mfp(q_{E,\mfp})/\ord_\mfp(q_{E,\mfp})$. With the formalism previously described, we have shown that the derivative of ${\bf a}_\mfp$ with respect to the weight variable is given by the L-invariant $\mcL_\mfp$. To conclude the proof of Oda's conjecture, one has to show a similar relation between the the derivative of  ${\bf a}_\mfp$ and the (arithmetic) $\mcL$-invariant attached to the Tate's period of $E$, but such a relation has been obtained in several degrees of generality by Greenberg-Stevens, Colmez, Ding and others (see for example \cite{Ding}). We note that our approach is very similar to that of Gehrmann and Rosso because, as we do in proposition \ref{D-EforgenS}, they relate big principal series $C_{{\bf k}_\mfp^{-2}}(\mcL_\mfp,\Lambda_F)\simeq {\rm Ind}_P^G(\hat{\bf k}_\mfp)$ with the extensions of the Steinberg representation described in \S \ref{subsection.steinbergrep}.

\section{Plectic points and Hida-Rankin $p$-adic L-functions}

\subsection{Plectic points}\label{plecticpoints}

Let $S_0$ be a set of places $\mfp$ above $p$ that are non-split at $T$ and the representation $V_\mfp^{\Z_p}$ is $\tno{St}_{\Z_p}(F_\mfp)(\varepsilon_\mfp)$, the twist of $\tno{St}_{\Z_p}(F_\mfp)$ by a character $\varepsilon_\mfp:G(F_\mfp)\ra \pm 1$ given by $g\mapsto (\pm 1)^{v_\mfp\det g}$.
Let $\phi_\lambda^{S_0}\in H_\ast^u(G(F)_+,\mcA^{{S_0}\cup\infty}(V_{S_0}^{\Z_p},\Z_p))^\lambda$ be the modular symbol associated with an elliptic curve $E/F$ as in remark \ref{remintMS}. 
Notice that the restriction $\varepsilon_\mfp:T(F_\mfp)\ra\pm1$ is non trivial only if $\varepsilon_\mfp\neq 1$ and $T$ ramifies at $\mfp$.

We observe that, for all $\mfp\in {S_0}$, morphism \eqref{definition.multiplicativeintegral} gives rise to an analogous $G(F_\mfp)$-equivariant morphism
\[
\Hom(\tno{St}_{\Z_p}(F_\mfp)(\varepsilon_\mfp),{\Z_p}) \longrightarrow \Hom(\ddivz_\mfp,\hat K_\mfp^\ti)(\varepsilon_\mfp).
\]
Moreover, the composition $\tno{ev}_{S_0}^\ast\circ\lambda_{\tno{unv}}$ provides a morphism
\[
\tno{ev}_{S_0}^\ast\circ\lambda_{\tno{unv}}:\Hom(\tno{St}_{\Z_p}(F_{S_0})(\varepsilon_S),\Z_p) \longrightarrow \Hom(\ddivz_{S_0},\hat K_{S_0}^\ti)(\varepsilon_{S_0}),\qquad \varepsilon_{S_0}=\prod_{\mfp\in {S_0}}\varepsilon_\mfp:G(F_{S_0})\rightarrow\pm 1.
\]

If we write $H_{\varepsilon_\mfp}/K_\mfp$ for the extension cut out by $\varepsilon_\mfp$, then by \cite[Corollary 5.4]{Silverman}
\begin{equation}\label{non-splitTate}
E(K_\mfp)=\left\{u\in H_{\varepsilon_\mfp}^\times/q_\mfp^\Z:\;u^{-\varepsilon_\mfp(\tau)}u^\tau\in q_\mfp^\Z\right\},\qquad 1\neq\tau\in\Gal(H_{\varepsilon_\mfp}/K_\mfp).    
\end{equation}
Hence, we no longer have a Tate uniformization $K_\mfp^\times/q_\mfp^\Z\sim E(K_\mfp)$ but an isomorphism
\[
K_\mfp^\times/q_\mfp^\Z\sim E(K_\mfp)_{\varepsilon_\mfp}:=\{P\in E(H_{\varepsilon_\mfp});\;P^\tau=\varepsilon_\mfp(\tau)\cdot P\}.
\]

By Hypothesis, $V_{S_0}^{\Z_p}=\tno{St}_{\Z_p}(F_{S_0})(\varepsilon_{S_0})$. 
Hence, given a class $\phi_\lambda^{S_0}\in H_\ast^u(G(F)_+,\mcA^{{S_0}\cup\infty}(V_{S_0}^{\Z_p},\Z_p))^\lambda$ generating our automorphic representation $\rho\mid_{G(\A_F^{{S_0}\cup\iy})}$, we can consider 
\[
\tno{ev}_{S_0}^\ast\circ\lambda_{\tno{unv}}(\phi_\lambda^{S_0})\in H_\ast^u(G(F)_+,\mcA^{{S_0}\cup\iy}(\ddivv_{S_0}^0,\hat K_{S_0}^\ti)(\varepsilon_{S_0}))^\lambda,
\]
where $\varepsilon_{S_0}$ is now seen as a character of $G(F)$ by means of the composition $G(F)\hookrightarrow G(F_{S_0})\stackrel{\varepsilon_{S_0}}{\longrightarrow}\pm 1$. Fornea and Gehrmann show in \cite{fornea2021plectic} (see also \cite{HerMol1}) that this class can be uniquely extended to a class
\[
\psi_\lambda^{S_0}\in H_\ast^u(G(F)_+,\mcA^{{S_0}\cup\iy}(\ddivv_{S_0},\hat E(K_{S_0})_{\varepsilon_{S_0}})(\varepsilon_{S_0}))^\lambda,\qquad \hat E(K_{S_0})_{\varepsilon_{S_0}}:=\bigotimes_{\mfp\in {S_0}}\hat E(K_\mfp)_{\varepsilon_\mfp}=\bigotimes_{\mfp\in {S_0}}\hat K_\mfp^\ti/q_\mfp^{\Z_p},
\]
generating $\rho\mid_{G(\A_F^{{S_0}\cup\iy})}$, where $\hat E(K_\mfp)_{\varepsilon_\mfp}$ is the $p$-adic completion of $E(K_\mfp)_{\varepsilon_\mfp}$.

For all $\mfp\in S$, let $\tau_\mfp\in\uhp_\mfp$ be the point fixed by $T(F)$ as in \S \ref{deltap}.
Notice that we have a well-defined morphism of $G(F)_+$-modules
\begin{equation}
 \Z[G(F)_+/T(F)_+] \longrightarrow \ddivv_{S_0}; \qquad  n\cdot gT(F)_+ \longmapsto n\ipa{\bigotimes_{\mfp\in {S_0}}g\tau_\mfp}.
\end{equation}
It induces a $G(F)_+$-equivariant morphism
\[
\mcA^{{S_0}\cup\iy}(\ddivv_{S_0},\hat E(K_{S_0})_{\varepsilon_{S_0}})(\varepsilon_{{S_0}})\longrightarrow \mcA^{{S_0}\cup\iy}(\Z[G(F)_+/T(F)_+],\hat E(K_{S_0})_{\varepsilon_{S_0}})(\varepsilon_{{S_0}}),\qquad \psi\mapsto\psi\mid_{(\tau_\mfp)_\mfp}.
\]
Since  $\mcA^{{S_0}\cup\iy}(\Z[G(F)_+/T(F)_+],\hat E(K_{S_0})_{\varepsilon_{S_0}})(\varepsilon_{{S_0}})\simeq \Ind_{T(F)_+}^{G(F)_+}\ipa{\mcA^{{S_0}\cup\iy}(\hat E(K_{S_0})_{\varepsilon_{S_0}}(\varepsilon_{{S_0}}))}$ (see \cite[Lemma 4.1]{guitart2017automorphic}),
we obtain  
$\psi^{S_0}_\lambda\mid_{(\tau_\mfp)_\mfp}\in H_\ast^u(T(F)_+,\mcA^{{S_0}\cup\iy}(\hat E(K_{S_0})_{\varepsilon_{S_0}})(\varepsilon_{{S_0}}))^\lambda$ by Shapiro's lemma.

Let us consider now the subspace of functions
\[
C(\mcG_T,\ovl \Q)^{\varepsilon_{S_0}}:=\{f\in C(\mcG_T,\ovl \Q),\;\Theta^*f\mid_{T(F_{S_0})}=\varepsilon_{S_0}\},
\]
where, as above, $\varepsilon_{S_0}:T(F_{S_0})\ra\pm 1$ is the product of the local $\varepsilon_\mfp$.
In this situation, the Artin map provides a decomposition:
\begin{equation}
C(\mcG_T,\ovl \Q)^{\varepsilon_{S_0}}=\bigoplus_{\lambda:T(F)/T(F)_+\ra\icla{\pm 1}}H^0\ipa{T(F)_+,C^0(T(\A_F^{{S_0}\cup\iy}),\ovl \Q)(\varepsilon_{S_0})}.
\end{equation}
Given a locally constant character $\xi\in C(\mcG_T,\ovl \Z)^{\varepsilon_{S_0}}$ we can consider $\xi_\lambda\in H^0\ipa{T(F)_+,C^0(T(\A_F^{{S_0}\cup\iy},\ovl \Z)(\varepsilon_{S_0}))}$ its $\lambda$-component.
Then, we can define a twisted \tbf{plectic point} as
\[P_\xi^{S_0}=\left(\eta^{S_0}\cap \xi_\lambda\right)\cap \psi^{S_0}_\lambda\mid_{(\tau_\mfp)_\mfp}\in \hat E(K_{S_0})_{\varepsilon_{S_0}}\oti_{\Z_p} \ovl \Z_p,\]
where again the cap product is with respect to the pairing $\langle\cdot,\cdot\rangle_+$ of \eqref{equation.pairingplus} twisted by $\epsilon_{S_0}$.

\subsection{$p$-adic L-functions with variables of type weight and level}

Let
$\phi_\lambda^p\in H_\ast^u\ipa{G(F)_+,\mcA^{p\cup\infty}(V_p,V(\underline{k})_{\C_p})}^\lambda$ be an ordinary cohomology class generating an automorphic representation $\pi$. By Proposition \ref{OCmodsymb} together with the fact that $\D_{\chi_p^{-2}}(\C_p)=\D_{\chi_p^{-2}}(\mfo_{\C_p})\otimes_{\mfo_{\C_p}}\C_p$, the modular symbol $\phi_\lambda^p$ provides (up-to constant) an overconvergent 
$\hat\phi_\lambda^p\in H_\ast^u\ipa{G(F)_+,\mcA^{p\cup\infty}(\D_{\chi_p^{-2}}(\mfo_{\C_p})_{\underline{\alpha}})}^\lambda$,
for some locally polynomial character $\chi_p\inv$.
Write $\varphi=\varphi_{\chi_p}$ and let us consider the Hida family 
\[
\Phi_\lambda^p\in H_\ast^u\ipa{G(F)_+,\mcA^{p\cup\infty}(\D_{{\bf k}_p^{-2}}(\Lambda_F)_{{\bf\underline{a}}})}^\lambda
\]
specializing to $\hat\phi_\lambda^p$ through the morphism $s_{\varphi}$ of \eqref{eqspecialization}. Similarly as in Theorem \ref{relOCpadicLfunct},
we define 
\[
\int_{\mcG_T}fd\mu_{\Phi_\lambda^p}=\ipa{(\Theta^\ast f)\cap\eta}\cap \delta_p^\ast(\Phi_\lambda^p),\qquad f\in C(\mcG_T,\Lambda_F)
\]
where $\delta_p: C_{c}(T(F_p),\Lambda_F) \rightarrow {\rm Ind}_P^G(\hat{\bf k}_p)\stackrel{\varphi_p}{\simeq}C_{{\bf k}_p^{-2}}(\mcL_p,\Lambda_F)$ is defined as in \eqref{eq:deltaSexplicit}.
Indeed, we can think of $\delta_p^\ast(\Phi_\lambda^p)$ and $\rho^\ast f\cap\eta$ as elements
\begin{eqnarray*}
\delta_p^\ast(\Phi_\lambda^p)&\in& H_\ast^u\ipa{G(F)_+,\mcA^{p\cup\infty}(C_c(T(F_p),\Lambda_F),\Lambda_F)}^\lambda,\\
\rho^\ast f\cap\eta &\in& H_u(T(F),C_{\rm fc}(T(\A_F),\Lambda_F))=H_u\ipa{T(F),{\rm Ind}_{T(F)_+}^{T(F)}C_c^0\ipa{T(\A_F^{p\cup\infty}),C_{c}(T(F_p),\Lambda_F)}},
\end{eqnarray*}
hence the pairing \eqref{equation.pairingnoplus} applies and the cap-product is well defined. The measure $\mu_{\Phi^p_\lambda}$ is a $p$-adic L-function with variables of type weight and level since its specialization at different weights $\theta_p:\mfo_{F_p}^\ti\ra\C_p$ provides different measures $s_{\varphi_{\theta_p}}(\mu_{\Phi_\lambda^p})$ of $\mcG_T$. In particular
$s_{\varphi}(\mu_{\Phi_\lambda^p})=\mu_{\phi_\lambda^p}$,
by theorem \ref{relOCpadicLfunct}. Thus $\mu_{\Phi_\lambda^p}$ interpolates $\mu_{\phi_\lambda^p}$ as the weight varies.

\subsection{Hida-Rankin $p$-adic L-function}

Let $\xi:\mcG_T\ra\bar \Z^\ti$ be a locally constant character such that $\Theta^\ast\xi\mid_{T(F_\infty)}=\lambda$. Using the natural ring homomorphism $\bar\Z\subset\Lambda_F\otimes\bar\Z$, it can be seen as a function in $C(\mcG_T,\Lambda_F\otimes\bar\Z)$.
We consider
\[
L_p(\phi^p_\lambda,\xi,{\bf k}_p):=\int_{\mcG_T}\xi d\mu_{\Phi_\lambda^p}\in\Lambda_F\oti_{\Z_p}\bar\Z_p.
\]
We can think of $L_p(\phi^p_\lambda,\xi,{\bf k}_p)$ as a restriction of the $p$-adic L-function $\mu_{\Phi_\lambda^p}$ to the weight variable.

Let $1$ be the trivial weight, $\varphi=\varphi_1:\Lambda_F\rightarrow\Z_p$ and assume that $\hat\phi_\lambda^p=s_\varphi\ipa{\Phi_\lambda^p}$ is associated with an elliptic curve $E/F$.
Similarly as in \S \ref{plecticpoints}, let ${S_0}$ be a set of primes $\mfp$ above $p$ such that:
\begin{itemize}
    \item $T$ does not split at any $\mfp\in {S_0}$.
    
    \item  The representation $V_\mfp^{\Z_p}$ is $\tno{St}_{\Z_p}(F_\mfp)(\varepsilon_\mfp)$.
    
    \item  We have that $\Theta^\ast\xi\mid_{T(F_{S_0})}=\varepsilon_{S_0}\mid_{T(F_{S_0})}$.
\end{itemize}
As discussed in \S \ref{plecticpoints}, for any $\mfp\in {S_0}$ one can identify $E(K_\mfp)_{\varepsilon_\mfp}$ with $K_\mfp^\ti/q_\mfp^\Z$. Under this identification, the non-trivial automorphism $\sigma\in\Gal(K_\mfp/F_\mfp)$ provides an involution on $E(K_\mfp)_{\varepsilon_\mfp}$ that we will denote by $P\mapsto\ovl P$. Such an involution corresponds to $\ovl P=\varepsilon_\mfp\bsm\varpi_\mfp&\\&1\esm P^\sigma$ if $T(F_\mfp)$ is inert.
For any point $P\in E(K_\mfp)_{\varepsilon_\mfp}$, we can think its (twisted) trace $P+\bar P\in E(F_\mfp)_{\varepsilon_\mfp}$ as an element 
\[
{\rm Tr}(P):=P+\bar P\in F_\mfp^\ti/q_\mfp^{2\Z}\subset E(F_\mfp)_{\varepsilon_\mfp}:=\{Q\in E(H_{\varepsilon_\mfp}); \;Q^\tau=\varepsilon_\mfp(\tau)\cdot Q,\;\bar Q=Q\}.
\]
Hence, by theorem \ref{THMvanishq} it makes sense to consider 
$\ell_{{\bf a}_\mfp}\circ{\rm Tr}(P)
\in I/I^2$, where $I=\ker(\varphi)$.
This implies that, given a plectic point $P_\xi^{{S_0}}\in \hat E(K_{S_0})_{\varepsilon_{S_0}}\oti_\Z\bar\Z$, we can consider 
\[
\ell_{\underline{{\bf a}}}\circ{\rm Tr}(P_\xi^{{S_0}}):=\ipa{\prod_{\mfp\in {S_0}}\ell_{{\bf a}_\mfp}\circ{\rm Tr}}(P_\xi^{{S_0}})\in I^r/I^{r+1}\oti_\Z\bar\Z,\qquad r:=\#{S_0}.
\]
\begin{remark}\label{evvS}
Recall that $\phi_\lambda^p\in  H_\ast^u(G(F),\mcA^{p\cup\infty}(V^{S_0}\oti_{\Z_p} \tno{St}_{\Z_p}(F_{S_0})(\varepsilon_{S_0}),\Z_p)(\lambda))$. Since for any $v^{S_0}\in V^{S_0}$ we have a $G(F)$-equivariant morphism
\[
\mcA^{p\cup\infty}(V^{S_0}\oti_{\Z_p} \tno{St}_{\Z_p}(F_{S_0})(\varepsilon_{S_0}),\Z_p)\longrightarrow \mcA^{{S_0}\cup\infty}(\tno{St}_{\Z_p}(F_{S_0})(\varepsilon_{S_0}),\Z_p);\qquad \phi\longmapsto \phi(v^{S_0}),
\]
we can consider $\phi_\lambda^p(v^{S_0})\in H_\ast^u(G(F),\mcA^{{S_0}\cup\infty}(\tno{St}_{\Z_p}(F_{S_0})(\varepsilon_{S_0}),\Z_p)(\lambda))$. By means of $\phi_\lambda^p(v^{S_0})$ and a character $\xi$ as in \S \ref{plecticpoints}, we can construct the corresponding plectic point (depending on $v^{S_0}$)
\[
P_\xi^{S_0}(v^{S_0})\in \hat E(K_{S_0})_{\varepsilon_{S_0}}\oti_{\Z_p}\bar\Z_p.
\]
\end{remark}
For every $\mfq\in p\setminus {S_0}$, let $J_{\mfq}$ be the unique element in $P\subset G(F_\mfq)$ that also lies in the normalizer of $T(F_\mfp)$. Recall that, just before theorem \ref{THMintprop}, we define $\mfq$-stabilized elements $f_{0,\mfq}$ and we choose $c_\xi$-admissible Eichler orders $\mcO_{N,\mfq}$. Moreover, we discuss the existence of $k_{0,\mfq}\in T(F_\mfq)$ such that $k_{0,\mfq}^{-1}J_\mfq\in\mcO_{N,\mfq}$.
Let us choose a non-zero $\mcO_{N,\mfq}^\times$-invariant $v_{0,\mfq}\in V_\mfq$ and write $v_0^S=\bigotimes_{\mfq\in p\setminus S}v_{0,\mfp}\in V^S$. We recall the constants $C(\pi_\mfp)$ and $C(\pi_\mfp,\xi_\mfp)$ from theorem \ref{THMintprop} and define
\[
0\neq \epsilon^{S_0}(\pi,\xi)=\prod_{\mfq\in p\setminus {S_0}}\frac{\langle f_{0,\mfq},f_{0,\mfq}\rangle_\mfp}{\langle v_{0,\mfp},v_{0,\mfp}\rangle_\mfp}\frac{\xi_\mfq^{-1}(k_{0,\mfq})}{{\rm vol}(\mfo_{K_\mfq}^\ti/\mfo_{F_\mfq}^\ti)^2}\frac{C(\pi_\mfq,\xi_\mfq)C(\pi_\mfq)}{L(1/2,\Pi_\mfq,\xi_\mfq)}\cdot\left\{\begin{array}{ll}
        1,&n_{\xi_\mfq}= 0;\\
        L(1,\psi_{K_\mfp})^{-2},&n_{\xi_\mfq}> 0;
    \end{array}\right.
\]
The following theorem is the main result of the paper. 
\begin{theorem}\label{mainTHM}
Assume that $r=\#S_0\neq 0$,
then we have that 
\[
L_p(\phi^p_\lambda,\xi,{\bf k}_p)\in I^{r}\oti_\Z\bar\Q.
\]
Moreover, if the representations $\pi_\mfp$ are quotients of $\Ind_P^G(\hat\chi_\mfp^0)^0$, where $\hat\chi_\mfp^0$ is unramified, then
\[
L_p(\phi^p_\lambda,\xi,{\bf k}_p)\equiv [{\rm U}_+^{S_0}:{\rm U}_+]^{-1}\cdot\epsilon^{S_0}(\pi,\xi)^{\frac{1}{2}}\cdot\ell_{{\bf a}_{S_0}}\circ{\rm Tr}(P_\xi^{{S_0}}(v_0^{S_0}))\;\;({\rm mod}\;I^{r+1}\oti_\Z\bar\Q).
\]
\end{theorem}
\bpf 
Recall that $\Phi_\lambda^p$ has a preimage in $\Phi_\lambda^p\in H_\ast^u(G(F)_+,\mcA^{p\cup\infty}(\D_{{\bf k}_p^{-2}}(\Lambda_F)_{\underline{\bf a}^\ast}^S(\varepsilon_S)))^{\lambda}$.
Since the functor $N\mapsto H_\ast^r(G(F)_+,\mcA^{S\cup\iy}(N))$ commutes with direct limits,  
$\Phi_\lambda^p\in H_\ast^u(G(F),\mcA^{p\cup\iy}(\D^{\varphi_n}_{{\bf k}_p^{-2}}(\Lambda_{n})_{\underline{\bf a}^\ast}^{S_0}(\varepsilon_{S_0}))(\lambda))$, for some $n\in\N$, where $\underline{\bf a}^\ast=\ipa{\varepsilon_\mfp\bsm\varpi_\mfp&\\&1\esm\cdot{\bf a}_\mfp}_\mfp$.
Hence, we can apply $r_{S_0}$ and $\bar{\rm ev}_{S_0}$ of propositions \ref{D-EforgenS} and \ref{corevdiv} and we obtain
\[
\bar{\rm ev}_{S_0}\circ r_S\ipa{\Phi_\lambda^p}\in H^u(G(F),\mcA^{p\cup\iy}(V^{S_0}\oti_{\Z_p}\Delta_{F_{S_0}},I_n^r/I_n^{r+1})(\varepsilon_{S_0})(\lambda)),
\]
where $I_n=\ker(\varphi_n)$.

Let $\Delta_T:=\Z[G(F)/T(F)]$.
Since $\tau_\mfp$ and $\ovl\tau_\mfp$ are $T(F)$-invariant, one can construct a $G(F)$-module morphism
\begin{center}
\begin{tikzcd}
\Delta_T \ar[r,hook] & \Delta_{F_S}
\subset \Delta_S,
\qquad
\sum_{i}g_i\cdot T(F) \ar[r,mapsto] & \sum_i\bigotimes_{\mfp\in S} g_i\cdot (\tau_\mfp+\ovl\tau_\mfp) \\[-2em]
\end{tikzcd}
\end{center}
The morphism $\tno{ev}_{S_0}^\ast\circ\lambda_{\tno{unv}}$ restricts to a $G(F)$-equivariant $\Hom(\tno{St}_{\Z_p}(F_{S_0})(\varepsilon_{S_0}),\Z_p) \ra \Hom(\Delta_T^0,\hat F_{S_0}^\ti)(\varepsilon_{S_0})$, where $\Delta_T^0$ is the set of degree zero divisors in $\Delta_T$.
Thus, for all $v^{S_0}\in V^{S_0}$, we obtain the following diagram with $M=\hat F_{S_0}^\ti$ or $M=\bigotimes_{\mfp\in {S_0}}\hat F_\mfp^\ti/q_\mfp^{2\Z_p}\subseteq \bigotimes_{\mfp\in {S_0}}\hat E(F_\mfp)_{\varepsilon_\mfp}=:\hat E(F_{S_0})_{\varepsilon_{S_0}}$
\begin{center}\small
\begin{tikzcd}
 &  \phi^p_\lambda(v^{S_0})\in H_\ast^u(G(F),\mcA^{{S_0}\cup\iy}(\tno{St}_{\Z_p}(F_{S_0})(\varepsilon_{S_0}),\Z_p)(\lambda)) \ar[d,"\lambda^{\rm univ}"] \\
 H_\ast^u(G(F),\mcA^{{S_0}\cup\iy}(\mcE_{F_{S_0}^\ti},M)(\varepsilon_S)(\lambda))\ar[d,"{\rm ev}\mid_{\Delta_T}"]\ar[r] &  H_\ast^u(G(F),\mcA^{{S_0}\cup\iy}(\tno{St}_{F_{S_0}^\ti},M)(\varepsilon_{S_0})(\lambda)) \ar[d] \\
H_\ast^u(T(F),\mcA^{{S_0}\cup\iy}(M)(\varepsilon_{S_0})(\lambda)) \ar[r,"\tno{res}"] & H_\ast^u(G(F),\mcA^{{S_0}\cup\iy}(\Delta_T^0,M)(\varepsilon_{S_0})(\lambda)) \\[-2em]
\end{tikzcd}
\end{center}
since $\mcA^{{S_0}\cup\iy}(\Delta_T,M)(\lambda)={\rm coInd}_{T(F)}^{G(F)}(\mcA^{{S_0}\cup\iy}(M)(\lambda))$. If $M=\bigotimes_{\mfp\in {S_0}}\hat F_\mfp^\ti/q_\mfp^{2\Z}$, we have that $\lambda^{\rm univ}\phi^p_\lambda(v^{S_0})$ extends to an element $\bar\phi^p_\lambda(v^{S_0})\in  H_\ast^u(G(F),\mcA^{{S_0}\cup\iy}(\mcE_{F_{S_0}^\ti},M)(\varepsilon_{S_0})(\lambda))$.
It is clear that
\[{\rm Tr}\ipa{P_\xi^{S_0}(v^{S_0})}=\eta^{S_0}\cap\xi_\lambda\cap {\rm ev}\mid_{\Delta_T}\bar\phi^p_\lambda(v^{S_0})\in \hat E(F_{S_0})_{\varepsilon_{S_0}}\oti_{\Z_p}\bar\Z_p.\]
If we apply the logarithm $\ell_{{\bf a}_{S_0}}:M\rightarrow I^r/I^{r+1}\subseteq I_n^r/I_n^{r+1}$, by proposition \ref{corevdiv}, we have that   
\[\ell_{{\bf a}_{S_0}}\circ{\rm Tr}\ipa{P_\xi^{S_0}(v^{S_0})}=\eta^{S_0}\cap\xi_\lambda\cap \ell_{{\bf a}_{S_0}}{\rm ev}\mid_{\Delta_T}\bar\phi^p_\lambda(v^{S_0})=\eta^{S_0}\cap\xi_\lambda\cap {\rm ev}\mid_{\Delta_T}r_{S_0}\Phi_\lambda^p(v^{S_0}).
\]
where ${\rm ev}\mid_{\Delta_T}r_{S_0}\Phi_\lambda^p(v^{S_0})\in H_\ast^u(T(F),\mcA^{{S_0}\cup\iy}(I_n^r/I_n^{r+1})(\varepsilon_{S_0})(\lambda))$ is defined analogously as in remark \ref{evvS}.

Write $\ovl{(\cdot)}$ for reduction modulo $I_n^{r+1}$. Let $\Phi\in \mcA^{p\cup\infty}(\D^{\varphi_n}_{{\bf k}_p^{-2}}(\Lambda_{n})_{\udl{\bf a}^\ast}^{S_0})$ and let $f^{S_0}=f_{p\setminus {S_0}}\oti f^p\in C^0_c(T(\A_F^{{S_0}\cup\infty}),\bar\Z_p)$, where $f_{p\setminus {S_0}}\in C^0_c(T(F_{p\setminus {S_0}}),\bar\Z_p)$ and $f^p\in C^0_c(T(\A_F^{p\cup\infty}),\bar\Z_p)$. Moreover, let $H\subseteq T(F_{p\setminus {S_0}})$ be a compact subgroup small enough so that $f_{p\setminus {S_0}}=\sum_xf_{p\setminus {S_0}}(x)\cdot\indi_{xH}$. If we write $\delta_{p}^{S_0}:=\bigotimes_{\mfp\in p\setminus {S_0}}\delta_\mfp$, we compute the pairing \eqref{equation.pairingplus}:
\begin{eqnarray*}
&&\langle f^{S_0}, {\rm ev}\mid_{\Delta_T}r_{S_0}\Phi(\delta_{p}^{S_0}(\indi_H))\rangle_+=\\
&&\qquad\qquad=\int_{T(\A_F^{p\cup\infty})}f^p(z)\int_{T(F_{p\setminus S})}f_{p\setminus {S_0}}(x)\cdot\ipa{\bar{\rm ev}_{S_0}\circ r_{S_0}}(\Phi)(z)\ipa{x\delta_{p}^{S_0}(\indi_H)\oti\bigotimes_{\mfp\in {S_0}}(\tau_\mfp+\ovl\tau_\mfp)}d^\ti xd^\ti z\\
&&\qquad\qquad={\rm vol}(H)\int_{T(\A_F^{p\cup\infty})}f^p(z)\cdot r_{S_0}\Phi(z)\ipa{\delta_{p}^{S_0}(f_{p\setminus {S_0}})\oti\bigotimes_{\mfp\in {S_0}}\ipa{\ell_{{\bf a}_\mfp}(\phi_{\tau_\mfp}\phi_{\bar\tau_\mfp}),1}}d^\ti z\\
&&\qquad\qquad={\rm vol}(H)\int_{T(\A_F^{p\cup\infty})}f^p(z)\cdot \ovl{\int_{\mcL_p}\delta_{p}^{S_0}(f_{p\setminus {S_0}})\cdot\prod_{\mfp\in {S_0}}\ipa{1+\ell_{{\bf a}_\mfp}\ipa{(c\tau_\mfp+d)(c\bar\tau_\mfp+d)}}d\Phi(z)}d^\ti z\\
&&\qquad\qquad={\rm vol}(H)\int_{T(\A_F^{p\cup\infty})}f^p(z)\cdot \ovl{\int_{\mcL_p}\delta_{p}^{S_0}(f_{p\setminus {S_0}})\cdot\prod_{\mfp\in {S_0}}{\bf k}_{\mfp}\inv\ipa{(c\tau_\mfp+d)(c\bar\tau_\mfp+d)}d\Phi(z)}d^\ti z\\
&&\qquad\qquad=
{\rm vol}(H)\cdot\ovl{\langle f^{S_0}\oti\indi_{T(F_{S_0})},\delta_p^\ast\Phi \rangle_+},
\end{eqnarray*}
where the  fifth equality follows from lemma \ref{lemma:mordelta}.
Hence, by lemma \ref{lemma.onedotseven} and relation \eqref{equation:capprod}, we obtain 
\begin{eqnarray*}
\ovl{L_p(\phi^p_\lambda,\xi,{\bf k}_p)}&=&\ovl{\ipa{(\Theta^\ast \xi)\cap\eta}\cap\delta_p^\ast(\Phi_\lambda^p)}=[{\rm U}_+^{S_0}:{\rm U}_+]\inv\ovl{\ipa{(\Theta^\ast \xi^{S_0}\oti\varepsilon_{S_0})\cap(\eta^{S_0}\cap\indi_{T(F_{S_0})})}\cap\delta_p^\ast(\Phi_\lambda^p)}\\
&=&[{\rm U}_+^{S_0}:{\rm U}_+]^{-1}{\rm vol}(H)^{-1}\eta^{S_0}\cap\xi_\lambda\cap {\rm ev}\mid_{\Delta_T}r_{S_0}\Phi_\lambda^p(\delta_{p}^{S_0}(\indi_H))\\
&=&[{\rm U}_+^{S_0}:{\rm U}_+]^{-1}{\rm vol}(H)^{-1}\ell_{{\bf a}_{S_0}}\circ{\rm Tr}\ipa{P_\xi^{S_0}(\delta_{p}^{S_0}(\indi_H))}.
\end{eqnarray*}
Since $I_n^r/I_n^{r+1}\oti_\Z\Q=I^r/I^{r+1}\oti_\Z\Q$ and $\phi^p_\lambda(tv^{S_0})=t\phi^p_\lambda(v^{S_0})$, for all $t\in T(F_{p\setminus {S_0}})$, the morphism
\[
\psi^{S_0}:V^{S_0}\longrightarrow I^r/I^{r+1}\oti_\Z\bar\Q,\qquad v^{S_0}\longmapsto \ell_{{\bf a}_{S_0}}\circ{\rm Tr}\ipa{P_\xi^{S_0}(v^{S_0})},
\]
satisfies $\psi^{S_0}(tv^{S_0})=\Theta^\ast\xi(t)\inv\cdot \psi^{S_0}(v^{S_0})$ for all $t\in T(F_{p\setminus {S_0}})$. By Saito-Tunnel (see \cite{Sa} and \cite{Tu}) the space 
\[\Hom_{T(F_{p\setminus {S_0}})}(V^{S_0}\oti\Theta^\ast\xi,I^r/I^{r+1}\oti_\Z\bar\Q)\] 
is at most one dimensional. Moreover, similarly as in \eqref{localfactlambda}, we have the pairing 
\begin{equation}\label{auxdeflambdaq}
    \lambda_\mfq(f_1,f_2):=\int_{T(F_\mfq)}\xi_\mfq(t)\langle\pi_\mfq(t)f_1,\pi_\mfq(J_\mfq)f_2\rangle_\mfq d^\times t;
\end{equation}
and $\lambda^{S_0}:=\prod_{\mfq\in p\setminus {S_0}}\lambda_\mfq\in \Hom_{T(F_{p\setminus {S_0}})}(V^{S_0}\oti\Theta^\ast\xi,\bar\Q)^{\oti 2}$.
By  Saito-Tunnel
\begin{eqnarray*}
    \psi^{S_0}(\delta_{p}^{S_0}(\indi_H))&=&\psi^{S_0}(v^{S_0}_0)\cdot \lambda^{S_0}\ipa{\delta_{p}^{S_0}(\indi_H),v_0^{S_0}}\cdot \lambda^{S_0}(v_0^{S_0},v_0^{S_0})^{-1}\\
    \lambda^{S_0}(v^{S_0},\delta_{p}^{S_0}(\indi_H))&=&\lambda^{S_0}\ipa{v_0^{S_0},\delta_{p}^{S_0}(\indi_H)} \cdot \lambda^{S_0}\ipa{v^{S_0},v_0^{S_0}}\cdot\lambda^{S_0}(v_0^{S_0},v_0^{S_0})^{-1}.
\end{eqnarray*}
By the symmetry of \eqref{auxdeflambdaq}, $\lambda^{S_0}(\delta_{p}^{S_0}(\indi_H),v_0^{S_0})\lambda^{S_0}(v_0^{S_0},v_0^{S_0})^{-1}=\lambda^{S_0}(\delta_{p}^{S_0}(\indi_H),\delta_{p}^{S_0}(\indi_H))^{\frac{1}{2}}\lambda^{S_0}(v_0^{S_0},v_0^{S_0})^{-\frac{1
}{2}}$.
Thus,
\begin{equation*}
    \frac{\psi^{S_0}(\delta_{p}^{S_0}(\indi_H))}{{\rm vol}(H)\cdot\psi^{S_0}(v^{S_0}_0)}=\frac{\lambda^{S_0}(\delta_{p}^{S_0}(\indi_H),\delta_{p}^{S_0}(\indi_H))^{\frac{1}{2}}}{{\rm vol}(H)\cdot\lambda^{S_0}(v_0^{S_0},v_0^{S_0})^{\frac{1
}{2}}}=\prod_{\mfq\in p\setminus {S_0}}\frac{\langle \delta_\mfq(\indi_{H_\mfq}),\delta_\mfq(\indi_{H_\mfq})\rangle_\mfq^{\frac{1}{2}}}{\langle v_{0,\mfq},v_{0,\mfq}\rangle_\mfq^{\frac{1}{2}}}\frac{\lambda_\mfq(\delta_\mfq(\indi_{H_\mfq}),J_\mfq)^{\frac{1}{2}}}{{\rm vol}(H_\mfq)\cdot\lambda_\mfp(v_{0,\mfq},J_\mfq)^{\frac{1}{2}}}=\epsilon^{S_0}(\pi,\xi)^{\frac{1}{2}},
\end{equation*}
by \eqref{PhiPhiint}.
The result then follows from the definition of $\psi^{S_0}$.
\epf

\printbibliography

\end{document}